\documentclass{amsart}
\usepackage{amsthm}
\usepackage{amssymb, latexsym}
\usepackage{enumerate}
\usepackage{mathtools,color}
\usepackage{graphics}
\usepackage{dsfont}
\usepackage{bbm}

\usepackage{amsmath}

\numberwithin{equation}{section}
\newtheorem{thm}{Theorem}[section]

\newtheorem{prop}[thm]{Proposition}
\newtheorem{lem}[thm]{Lemma}
\newtheorem{res}[thm]{Result}
\newtheorem{cor}[thm]{Corollary}

\theoremstyle{remark}
\newtheorem{rem}{Remark}[section]

\newcommand{\BBB}{\mathbb}
\newcommand{\R}{{\BBB R}}
\newcommand{\Z}{{\BBB Z}}
\newcommand{\T}{{\BBB T}}

\newcommand{\N}{{\BBB N}}

\newcommand{\p}{\partial}

\newcommand{\EQQ}[1]{\begin{equation*} \begin{split} #1
 \end{split} \end{equation*}}

\newcommand{\EQQARR}[1]{\begin{equation} \begin{array}{ll} #1 \end{array} \nonumber \end{equation}}
\newcommand{\EQQARRLAB}[1]{\begin{equation} \begin{array}{ll} #1 \end{array} \end{equation}}

\newcommand{\Combin}[2]
{
\left(
\begin{array}{l}
#1 \\
#2
\end{array}
\right)
}

\begin{document}

\title[Nonlinear third-order equations] {Local well-posedness and Parabolic smoothing of solutions of fully nonlinear third-order equations on the torus }
\author{Tristan Roy}
\address{American University of Beirut, Department of Mathematics}
\email{tr14@aub.edu.lb}

\begin{abstract}
In this paper we study the initial value problem of fully nonlinear third-order equations on the torus, that is
$\p_{t} u = F  \left( \p_{x}^{3} u, \p_{x}^{2}u, \p_{x}u , u, x,t \right) $ with $F$ a smooth function depending on the space variable $x$,
the time variable $t$, the first three derivatives of $u$ with respect to $x$, and $u$. In particular we find conditions on $u(0)$ and $F$ for which
one can construct a local and unique solution $u$. In particular if $F$ and $u(0)$ satisfy some conditions then the equation behaves like a diffusive one and
it has a parabolic smoothing property: the solution is infinitely smooth in one direction of time and the problem is ill-posed in the other direction of time.
If $F$ and $u(0)$ satisfy some other conditions then the equation behaves like a dispersive one. We also prove continuous dependence with respect to the data. The proof relies upon energy estimates combined with a gauge transformation (see e.g \cite{Hayashi,HayashiOzawa1, HayashiOzawa2}) and the Bona-Smith argument
(see \cite{BonaS}).
\end{abstract}

\maketitle

\tableofcontents
\setcounter{page}{001}

\section{Introduction}

We consider the third-order nonlinear Cauchy problem on the torus  $ \T := \mathbb{R} / 2 \pi \Z $

\EQQARRLAB{
\p_{t} u & = F (\vec{u},t) \cdot
\label{Eqn:CauchyPb}
}
Here $ F := \mathbb{R}^{4} \times \T \times \mathbb{R} \rightarrow \mathbb{R} := ( \vec{\omega},t ) \rightarrow F(\vec{\omega},t) $ is a
$\mathcal{C}^{\infty} -$ function with arguments $\vec{\omega} := (\omega_{3}, \omega_{2}, \omega_{1}, \omega_{0}, \omega_{-1}) \in
\mathbb{R}^{4} \times \T $  and $t \in \mathbb{R}$ the time variable; $ \vec{u}:= \left( \p_{x}^{3} u, \p_{x}^{2} u, \p_{x} u, u, x \right) $ with
$x \in \T$ the space variable and $u$ a function (lying in an appropriate space) depending on $x$ and $t$. \\
Equations of the form (\ref{Eqn:CauchyPb}) have been extensively studied in the literature, in particular on the real line (i.e $x \in \mathbb{R}$) and for
equations

\EQQARRLAB{
\partial_{t} u + \partial_{x}^{3} u = R(\p_{x}^{2} u, \p_{x} u, u) \cdot
\label{Eqn:Parta}
}
Here $R$ is a polynomial that satisfies some conditions. We mention some results on the real line. In this case, it is well-known that the linear part of these equations creates a strong dispersive effect, which should overcome the effect of the nonlinearity at least locally and be a plus for the construction of local solutions. In \cite{KPV2}, local well-posedness (L.W.P) on a small interval $[-T,T]$ (with $T > 0$) was proved for small data in \cite{KPV2} and for large data in \cite{KPV} in weighted Sobolev spaces with regularity exponent high enough for equations (\ref{Eqn:Parta}) by using, among other things, the dispersive effect of these equations. We also refer to \cite{KS} who studied L.W.P for generalizations of (\ref{Eqn:Parta}) to systems. In \cite{Pilod}, L.W.P was proved in weighted Besov spaces for small data. In \cite{HarGriff2}, L.W.P was studied in some translation-invariant subspaces of Sobolev spaces. In \cite{HarGriff}, under some additional assumptions on $R$, L.W.P was proved in (standard) Sobolev spaces $H^{k}$. Recently, in \cite{Akh}, the authors considered the full nonlinear problem (\ref{Eqn:CauchyPb}) and found that under some assumptions on $F$ and the data, the problem is locally well-posed and that the dispersive effect of the equation dominates. \\
In this paper we aim at constructing local strong solutions of the fully nonlinear problem (\ref{Eqn:CauchyPb2}) on the torus $\T$. By strong solutions of
(\ref{Eqn:CauchyPb}) we mean solutions $u \in \mathcal{C} \left( I, H^{k} \right)$ of the following integral equation on a interval $I$ containing $0$:

\EQQARRLAB{
t \in I: \; u(t)  =  u(0) + \int_{0}^{t} F  \left( \overrightarrow{u(t')}, t' \right) \; dt',
\label{Eqn:CauchyPb2}
}
with $u(0)$ a given function that lies in $H^{k}$. Here $H^{k}$ denotes the usual Sobolev space with index $k$, i.e the closure of smooth functions $f$  with respect to the norm  $ \| f \|_{H^{k}} :=  \| \{ \langle n \rangle^{k} \hat{f}(n) \}_{n \in \mathbb{Z}} \|_{l^{2}} $ that are defined on the torus. Here $\langle n \rangle := (1 + n^{2})^{\frac{1}{2}}$ and $\hat{f}(n)$ denotes the $n^{th}-$ Fourier coefficient of a function $f$ defined on the torus, that is $\hat{f}(n) := \frac{1}{2 \pi} \int_{0}^{2 \pi} f(x) e^{-inx} \; dx $. \\
To this end we define some quantities that we use throughout this paper and that appear in the statement of your theorem.  Let $f$ be function defined on the torus and let
$t \in \mathbb{R}$. Let $[f]_{ave}$ denote the average value of a function $f$ over $\T$, that is $ [ f ]_{ave} := \frac{1}{2 \pi} \int_{\T} f(x') \; dx'$. We define

\EQQ{
\delta(\vec{f}, t ) := \inf_{x \in \T} \left| \p_{\omega_{3}} F(\vec{f},t)(x) \right|, \; \text{and} \; \tilde{\delta}(\vec{f}) := \delta(\vec{f},0) \cdot
}
Here $\vec{f} := ( \partial_{x}^{3} f, \partial_{x}^{2} f, \partial_{x} f, f, x ) $.

It is well-known that $ \tilde{\delta} \left( \overrightarrow{u(0)} \right) $ measures the strength of the dispersion of (\ref{Eqn:CauchyPb2}). Intuitively, dispersion should facilitate the rule-out of blow-up and the construction of local solutions: see e.g \cite{Akh} for discussions related to this number. Therefore in the sequel we only consider solutions of (\ref{Eqn:CauchyPb}) with initial data $u(0)$ that satisfy the well-known non-degeneracy dispersion property, that is $ \tilde{\delta}
\left( \overrightarrow{u(0)} \right) > 0 $. Let \footnote{Notation convention in the formula of $P(f,t)$: $\p_{x}^{0} f := 1$.}

\EQQARR{
P (f,t) :=  \p_{\omega_2} F (\vec{f},t) + 6 \sum \limits_{p=-1}^{3} \p^{2}_{\omega_3 \omega_p} F (\vec{f},t) \p_{x}^{p+1} f , \\
Q(f,t) := \p_{\omega_{3}} F( \vec{f}, t) \left[ \frac{P(f,t)}{ _{\p_{\omega_{3}}F (\vec{f},t)}}  \right]_{ave},
}

\EQQARR{
\delta^{'}(f,t) := \inf \limits_{x \in \T}  | Q(f,t)(x)|, \; \text{and} \; \tilde{\delta}^{'}(f) := \delta^{'}(f,0) \cdot
}
We define now the notion of parabolic resonance for (\ref{Eqn:CauchyPb2}), by analogy with that in \cite{Tsugawa} in the framework of  fifth-order (semilinear) dispersive equations. We say that (\ref{Eqn:CauchyPb2}) is of \textit{non-parabolic resonance type} if
$ \left[ \frac{P(f,t)}{\partial_{\omega_{3}} F(\vec{f},t)} \right]_{ave} = 0$
for all $t \in \mathbb{R}$ and for all $ f \in  E := \left\{ h \in \mathcal{C}^{\infty} (\T): \; \delta(\vec{h},t) > 0 \; \text{for all} \;
t \in \mathbb{R} \right\} $
\footnote{ Let $ k > \frac{19}{2} $. We will be able to use the non-parabolic resonance property of (\ref{Eqn:CauchyPb2}) with data $\phi \in \widetilde{\mathcal{P}}_{k}$, with
$\widetilde{\mathcal{P}}_{k}$ defined below: see Theorem \ref{Thm:MainDisp}. Indeed observe that if (\ref{Eqn:CauchyPb2}) is of non-parabolic resonance type, then $ \left[ \frac{P(f,t)}{\partial_{\omega_{3}} F(\vec{f},t)} \right]_{ave} = 0 $ for all $ t \in \mathbb{R} $ and for all $ f \in \widetilde{\mathcal{P}}_{k} $: this follows from the density of $ \mathcal{C}^{\infty} (\T)$ in $H^{k}$ and arguments in  Section \ref{Sec:App}.}. If not, we say that (\ref{Eqn:CauchyPb2}) is of \textit{parabolic resonance type}. \\
Examples of non-parabolic resonance type equations are: \\
the \textit{KdV equation} $ \partial_{t} u = -  \partial_{x}^{3} u - 6 u \partial_{x} u $ and more generally the \textit{transition KdV equation}
$ \partial_{t} u = - \partial_{x}^{3} u - 6 h(t) u \partial_{x} u $ with $h \in \mathcal{C}^{\infty}(\mathbb{R})$ : $ \tilde{\delta}(\vec{f}) = 1$, $P(f,t)=0$ and $E = \mathcal{C}^{\infty}(\T)$; \\
the \textit{Rosenau-Hyman equation} (also called the $K(2,2)$ equation) $ \partial_{t} u =  \partial_{x}(u^{2}) + \partial_{xxx} (u^{2})$:
$ \tilde{\delta}(\vec{f}) =  2 \inf \limits_{x \in \T} |f(x)|$,
$P(\vec{f},t) = 18 \partial_{x} f $, and    $ E := \left\{ h \in \mathcal{C}^{\infty}(\T): \; h > 0 \; \text{or} \;  h < 0 \right\} $; \\
the \textit{Harry-Dym equation} $\partial_{t} u = u^{3} u_{xxx}$ :  $ \tilde{\delta}(\vec{f}) = \inf \limits_{x \in \T} |f(x)|^{3}$,  $P(\vec{f},t) = 18 f^{2} \partial_{x} f $, and $ E := \left\{ h \in \mathcal{C}^{\infty}(\T): \; h > 0 \; \text{or} \;  h < 0 \right\} $. \\
Examples of parabolic resonance type equations are: \\
the \textit{Korteweg-de Vries-Burgers equation} $\partial_{t} u  = -  \partial_{x}^{3} u +  \partial_{x}^{2} u  + 2 u \partial_{x} u = 0$: $\tilde{\delta}(\vec{f})= 1$,  $P(\vec{f},t)= 1$, and $ \tilde{\delta}^{'}(f) =1  $; \\
the \textit{Korteweg-de Vries type equation} $\p_{t} u = a(x,t) \p_{xxx} u + b(x,t) \p_{xx} u + c(x,t) \p_{x} u + d(x,t) u + e(t,x)$ with $\mathcal{C}^{\infty} ( \T \times \mathbb{R})-$ variable coefficients $a$,$b$,$c$,$d$, and $e$ such that  $ |a(x,0)| > 0 $ for all $x \in \T$ and $ \left[ \frac{b(0)}{a(0)} \right]_{ave} \neq 0$:  $ \tilde{\delta}(\vec{f}) = \inf \limits_{x \in \T} |a(0,x)|$, $ P(f,t) := 6 \partial_{x} a(t)  + b(t)$, and $ \tilde{\delta}^{'}(f) =  \tilde{\delta}(\vec{f})
\left| \left[ \frac{b(0)}{a(0)} \right]_{ave} \right| $. \\
We shall see that L.W.P holds on a small forward-in-time interval $[0,T]$ (resp. a small backward-in-time interval $[-T,0]$)
if we assume that (\ref{Eqn:CauchyPb}) is of parabolic resonance type, that the data has enough regularity, and that $\inf \limits_{x \in \T} Q(u(0),0)(x) > 0$  (resp. $
\sup \limits_{x \in \T} Q(u(0),0)(x) < 0 $); moreover the behavior of the solution depends on the sign of $Q (u(0),0)$. Therefore we define

\EQQARR{
\widetilde{\mathcal{P}}_{+,k} := \{ f \in H^{k}: \;  \tilde{\delta}(\vec{f}) > 0 \; \text{and} \; \inf \limits_{x \in \T} Q(f,0)(x) > 0   \}, \; \text{and} \\
\widetilde{\mathcal{P}}_{-,k} := \{ f \in H^{k}: \;  \tilde{\delta}(\vec{f}) > 0 \; \text{and} \; \sup \limits_{x \in \T} Q(f,0)(x) < 0   \} \cdot
}
More generally we define for $t \in \mathbb{R}$

\EQQARR{
\mathcal{P}_{+,k}(t) := \{ f \in H^{k}: \;  \delta(\vec{f},t) > 0 \; \text{and} \; \inf \limits_{x \in \T} Q(f,t)(x) > 0  \}, \; \text{and} \\
\mathcal{P}_{-,k}(t) := \{ f \in H^{k}: \;  \delta(\vec{f},t) > 0 \; \text{and} \; \sup \limits_{x \in \T} Q(f,t)(x) < 0   \} \cdot
}
We shall also see that L.W.P holds on a small interval $[-T,T]$ if we assume that (\ref{Eqn:CauchyPb}) is of non-parabolic resonance type and that the data has
enough regularity. Therefore we define

\EQQARR{
\widetilde{\mathcal{P}}_{k} := \{ f \in H^{k}: \; \tilde{\delta}(\vec{f}) > 0 \}, \; \text{and, more generally, for} \; t \in \mathbb{R}, \\
\\
\mathcal{P}_{k}(t) :=  \{ f \in H^{k}: \; \delta(\vec{f},t) > 0 \} \cdot
}
We now state the first theorem of this paper:

\begin{thm}

Assume that (\ref{Eqn:CauchyPb2}) is of parabolic resonance type. Let $ k \geq k_{0} > \frac{19}{2}$. Then the following properties hold:

\begin{enumerate}

\item $(Local \; existence \; and \; Parabolic \; smoothing)$
Let $ \phi \in \widetilde{\mathcal{P}}_{+,k}$ $ \left( \; \text{resp.} \; \phi \in  \widetilde{\mathcal{P}}_{-,k} \;  \right)$. Then one can find $ T := T \left( \| \phi \|_{H^{k_{0}}}, \tilde{\delta}(\vec{\phi}), \tilde{\delta}^{'}(\phi) \right) > 0 $ for which there exists a solution
$ u  \in \mathcal{C} \left( [0,T], H^{k}  \right) $
$ \left( \; \text{resp.} \; u \in \mathcal{C} \left( [-T,0], H^{k} \right) \; \right) $  of (\ref{Eqn:CauchyPb2}) with $ u(0) := \phi$ such that $ u(t) \in \mathcal{P}_{+,k}(t) $ for $ t \in [0,T] $ (resp. $ u(t) \in \mathcal{P}_{-,k}(t) $ for $t \in [-T,0]$ ). Moreover a parabolic smoothing effect holds, that is $ u \in \mathcal{C}^{\infty} \left( (0,T] \times \T \right) $ $\left( \text{resp.} \; u \in \mathcal{C}^{\infty} \left( [-T,0) \times \T \right) \; \right) $ if $\phi \in \widetilde{\mathcal{P}}_{+,k} $ (resp. $\phi \in \widetilde{\mathcal{P}}_{-,k} $).

\item [$$]

\item $(Uniqueness)$ Assume that for some $\breve{T} > 0$  there exist $ u_{1},u_{2} \in \mathcal{C} \left( [0, \breve{T}], H^{k} \right)$
$\left( \; \text{resp.} \; u_{1},u_{2} \in \mathcal{C} \left( [ - \breve{T},  0], H^{k} \right)  \right)$  solutions of
    (\ref{Eqn:CauchyPb2}) with $u_{1}(0) = u_{2}(0)$ such that $u_{q}(t) \in \mathcal{P}_{+,k}(t)$  $ \left( \; \text{resp.} \; u_{q}(t) \in \mathcal{P}_{-,k}(t) \right)$ holds for $q \in \{ 1,2 \} $ and for $ t \in [0, \breve{T}]$ $ \left( \; \text{resp.} \; t \in [-\breve{T},0] \right)$. Then $u_{1}(t) = u_{2}(t)$ on $ \left[ 0, \breve{T} \right]$ $\left( \; \text{resp.} \; \left[ -\breve{T}, 0 \right] \; \right)$.

\item [$$]

\item $(Continuous \; dependence \; on \;  initial \; data)$ Let $ \phi_{\infty} \in \widetilde{\mathcal{P}}_{+,k} $  $\left( \; \text{resp.} \;
\phi_{\infty} \in \widetilde{\mathcal{P}}_{-,k} \; \right)$. Let $\phi_{n} \in H^{k}$ be such that  $ \phi_{n} \rightarrow \phi_{\infty} $ in $ H^{k} $ as $ n \rightarrow \infty $. Then there exists  $N \in \mathbb{N}$ such that for $ \infty \geq n \geq N$ one can find $ \breve{T} := \breve{T} \left( \| \phi_{\infty} \|_{H^{k_{0}}}, \tilde{\delta}(\overrightarrow{\phi_{\infty}}), \tilde{\delta}^{'}(\phi_{\infty}) \right) > 0 $ for which
there is a unique solution $u_{n} \in \mathcal{C} \left( [0, \breve{T}], H^{k} \right)$
$ \left( \, \text{resp.} \; u_{n} \in \mathcal{C} \left( [-\breve{T},0], H^{k} \right) \; \right)$ of (\ref{Eqn:CauchyPb2}) with data $ u_{n}(0) := \phi_{n}$
that satisfies  $ u_{n}(t) \in \mathcal{P}_{+,k}(t) $ for $ t \in [0, \breve{T}] $
$ \left( \; \text{resp.} \; u_{n}(t) \in \mathcal{P}_{-,k}(t) \; \text{for} \;  t \in \left[ -\breve{T}, 0 \right] \right)$.
Moreover $ \sup_{t \in [0,\breve{T}]} \| u_{n}(t) - u_{\infty}(t) \|_{H^{k}} \rightarrow 0 $ $ \left(  \text{resp.} \; \sup_{t \in [-\breve{T},0]}
\| u_{n}(t) - u_{\infty}(t) \|_{H^{k}} \rightarrow 0  \right)$ as $ n \rightarrow \infty$.

\end{enumerate}

\label{Thm:MainDiff}
\end{thm}

\begin{rem}
The proof of the local existence part of the theorem shows that $T$ can be chosen as a continuous function that decreases as $\tilde{\delta}(\vec{\phi})$ decreases, that decreases as $\tilde{\delta}^{'}(\phi)$ decreases, and that decreases as $\| \phi \|_{H^{k_{0}}}$ increases. \\
\label{Rem:Exist}
\end{rem}

\begin{rem}

Roughly speaking, the first part of the theorem says that if $ \inf \limits_{x \in \T} Q \left( u(0),0 \right)(x) > 0 $ (resp.
$ \sup \limits_{x \in \T} Q \left( u(0),0 \right)(x) < 0$ ) then the diffusive part of the equation dominates in the forward (resp. backward) direction, which creates a parabolic smoothing effect for the solution (it is infinitely smooth) in the forward (resp. backward) direction.
Assume that $ \phi \in \widetilde{\mathcal{P}}_{+,k}$  or that $ \phi \in \widetilde{\mathcal{P}}_{-,k}$. Observe that $ \p_{\omega _{3}}  F \left( \vec{\phi}, 0 \right) \in \mathcal{C}^{0}$, the space of continuous functions on
$\T$: this follows from Lemma \ref{lem:EstDerivF} and the well-known Sobolev embedding $ H^{m} \hookrightarrow \mathcal{C}^{0}$, with $ m  > \frac{1}{2}$.  Hence
the function $\p_{\omega_{3}} F (\vec{\phi},0)$ has a constant sign. Hence $Q(\phi,0) = \left| \partial_{\omega_{3}} F(\vec{\phi},0) \right| \breve{\delta}(\phi) $ with

\EQQARR{
\breve{\delta}(f) :=   \left[ \frac {P(f,0)} { \left| \partial_{\omega_{3}} F ( \vec{f},0) \right|} \right]_{ave} \cdot
 }
If $\phi \in \widetilde{\mathcal{P}}_{+,k}$ ( resp. $\phi \in \widetilde{\mathcal{P}}_{-,k}$ ) then
$ \inf \limits_{x \in \T} Q(\phi,0)(x) $ (resp. $ \sup \limits_{x \in \T} Q(\phi,0)(x) $ ) is equal to $\tilde{\delta} (\vec{\phi}) \breve{\delta}(\phi)$ .
Hence if $\phi \in \widetilde{\mathcal{P}}_{+,k}$ (resp. $\phi \in \widetilde{\mathcal{P}}_{-,k} $) then  $ \breve{\delta}(\phi) > 0 $  (resp. $ \breve{\delta}(\phi) < 0 $) and $\breve{\delta}(\phi) $ (resp. $ -\breve{\delta} (\phi) $) measures the strength of the diffusion
in the forward (resp. backward) direction, holding $ \tilde{\delta}(\vec{\phi})$ constant: see e.g \cite{Akh2} for discussions related to this number,
for particular functions $F$ of (\ref{Eqn:CauchyPb}), such as $F(\vec{u},t) := a(x,t) \partial^{3}_{x} u + b(x,t) \partial^{2}_{x} u
+ c(x,t) \partial_{x} u + d(x,t) u + e(x,t)$.

\label{Rem:Qother}
\end{rem}





\begin{rem}
Consider $ \phi_{n} $ and $\phi_{\infty}$  that are defined in the continuous dependence on initial data part of the theorem.
We see from Appendix with $(f,g):= (\phi_{n},\phi_{\infty})$ that $\tilde{\delta}(\overrightarrow{\phi_{n}}) \rightarrow  \tilde{\delta}(\overrightarrow{\phi_{\infty}}) $
and $\tilde{\delta}^{'}(\phi_{n}) \rightarrow \tilde{\delta}^{'}(\phi_{\infty})$ as $ n \rightarrow \infty $, and that $\phi_{n} \in \widetilde{\mathcal{P}}_{+,k}$
(resp. $\phi_{n} \in \widetilde{\mathcal{P}}_{-,k}$ ) if $\phi_{\infty} \in \widetilde{\mathcal{P}}_{+,k} $
(resp. $\phi_{\infty} \in \widetilde{\mathcal{P}}_{-,k} $ ) for $n$ large. Hence we see from the local existence part of the theorem, the uniqueness part of the theorem, Remarks \ref{Rem:Exist} and \ref{Rem:Qother},  that the first statement of the continuous dependence on initial data part of the theorem holds. It remains to prove the second statement: see Section \ref{Sec:ProofThm}.

\label{Rem:Cont}
\end{rem}

\begin{rem}
The proof of the local existence part of the theorem also shows that if $ \phi \in \widetilde{\mathcal{P}}_{+,k}$ (resp. $ \phi \in \widetilde{\mathcal{P}}_{-,k}$) holds, then
$u(t) \in \mathcal{P}_{+,k}(t)$ (resp. $ u(t) \in \mathcal{P}_{-,k}(t)$ ) holds for $ t \in [0,T] $ (resp. $ t \in [-T,0] $), but also that $\delta \left( \overrightarrow{u(t)},t  \right) \gtrsim \tilde{\delta}(\vec{\phi})$ and that $\delta^{'} (u(t),t) \gtrsim \tilde{\delta}^{'}(\phi) $  for $ t \in [0,T]$ (resp. $ t \in [-T,0]$).
\end{rem}
We now state a corollary of Theorem \ref{Thm:MainDiff} that is a non-existence result for equations  (\ref{Eqn:CauchyPb2})
that are of parabolic resonance type:

\begin{cor}
Assume that (\ref{Eqn:CauchyPb2}) is of parabolic resonance type. Let $k > \frac{19}{2}$. Assume that $ \phi \in \widetilde{\mathcal{P}}_{+,k}$
$ \left( \; \text{resp.}  \; \phi \in \widetilde{\mathcal{P}}_{-,k} \; \right) $ and
that $ \phi \notin \mathcal{C}^{\infty}(\T) $. Then for any $T > 0$, there does not exist any solution $ u \in \mathcal{C} \left( [-T,0], H^{k} \right)$
$ \left( \; \text{resp.} \; u \in \mathcal{C} \left( [0,T], H^{k} \right) \; \right) $ of (\ref{Eqn:CauchyPb2}) with $ u(0) := \phi $.
\label{Cor:MainDiff}
\end{cor}

\begin{rem}
Corollary \ref{Cor:MainDiff} can be viewed as an ill-posedness result for equations (\ref{Eqn:CauchyPb2})
that are of parabolic resonance type. It shows that these equations are by nature parabolic: in other words, the parabolic
part of the equation dominates the dispersive part.
\end{rem}
We now state the second theorem of this paper:

\begin{thm}

Assume that (\ref{Eqn:CauchyPb2}) is of non-parabolic resonance type. Let $ k \geq k_{0} > \frac{19}{2}$. Then the following properties hold:

\begin{enumerate}

\item $(Local \; existence)$ Let $ \phi \in \widetilde{\mathcal{P}}_{k} $. Then one can find $ T := T \left( \| \phi \|_{H^{k_{0}}}, \tilde{\delta}(\vec{\phi}) \right)  > 0 $ for which there exists a solution $ u  \in \mathcal{C} \left( [-T,T], H^{k}  \right) $ of (\ref{Eqn:CauchyPb2}) with $ u(0) := \phi$ such that $ u(t) \in \mathcal{P}_{k}(t) $ for $ t \in [-T,T] $.

\item [$$]

\item $(Uniqueness)$ Assume that for some $\breve{T} > 0$  there exist $u_{1},u_{2} \in \mathcal{C} \left( [-\breve{T}, \breve{T}], H^{k} \right)$  solutions of
    (\ref{Eqn:CauchyPb2}) with $u_{1}(0) = u_{2}(0)$ such that $u_{q}(t) \in \mathcal{P}_{k}(t)$ holds for $q \in \{ 1,2 \} $ and for $ t \in [-\breve{T}, \breve{T}]$. Then $u_{1}(t) = u_{2}(t)$ on $ \left[ -\breve{T}, \breve{T} \right]$.

\item [$$]

\item $(Continuous \; dependence \; on \;  initial \; data)$ Let $ \phi_{\infty} \in \widetilde{\mathcal{P}}_{k} $. Let $\phi_{n} \in H^{k}$ be such that  $ \phi_{n} \rightarrow \phi_{\infty} $ in $ H^{k} $ as $ n \rightarrow \infty $. Then there exists  $N \in \mathbb{N}$ such that for $ \infty \geq n \geq N$ one can find $ \breve{T} := \breve{T} \left( \| \phi_{\infty} \|_{H^{k_{0}}}, \tilde{\delta}(\overrightarrow{\phi_{\infty}}) \right) > 0 $ for which
there is a unique solution $u_{n} \in \mathcal{C} \left( [-\breve{T}, \breve{T}], H^{k} \right) $ of (\ref{Eqn:CauchyPb2}) with data $ u_{n}(0) := \phi_{n}$
that satisfies  $u_{n}(t) \in \mathcal{P}_{k}(t)$ for $ t \in [-\breve{T}, \breve{T}]$. Moreover $ \sup_{t \in [-\breve{T},\breve{T}]} \| u_{n}(t) - u_{\infty}(t) \|_{H^{k}} \rightarrow 0 $ as $ n \rightarrow \infty $.

\end{enumerate}

\label{Thm:MainDisp}
\end{thm}

\begin{rem}
The proof of the local existence part of the theorem shows that $T$ can be chosen as a continuous function that decreases as $\tilde{\delta}(\vec{\phi})$ decreases, and
that decreases as $\| \phi \|_{H^{k_{0}}}$ increases. \\
\label{Rem:ExistDisp}
\end{rem}

\begin{rem}
Roughly speaking, the first part of the theorem says that the dispersive part of the equation dominates. Moreover the strength of the dispersion increases as the amplitude of $\tilde{\delta}(\vec{\phi})$ increases.
\label{Rem:QotherDisp}
\end{rem}

\begin{rem}
Consider $ \phi_{n} $ and $\phi_{\infty}$  that are defined in the continuous dependence on initial data part of the theorem.
We see from Appendix with $(f,g):= (\phi_{n},\phi_{\infty})$ that  $ \tilde{\delta} \left( \overrightarrow{\phi_{n}} \right) \rightarrow  \tilde{\delta} \left(\overrightarrow{\phi_{\infty}} \right) $ as $ n \rightarrow \infty $. Hence we see from the local existence part of the theorem, the uniqueness part of the theorem, and Remark \ref{Rem:ExistDisp} that the first statement of the continuous dependence on initial data part of the theorem holds. It remains to prove the second statement of the theorem.
\label{Rem:Cont}
\end{rem}

\begin{rem}
The proof of the local existence part of the theorem also shows that if $ \phi \in \widetilde{\mathcal{P}}_{k}$, then
$u(t) \in \mathcal{P}_{k}(t)$ holds for $ t \in [-T,T] $, but also that $\delta \left( \overrightarrow{u(t)}, t \right) \gtrsim \tilde{\delta}(\vec{\phi})$ for $ t \in [-T,T]$.
\end{rem}

\underline{Assumption}:  Observe that if $v(t) := u(-t)$ then it satisfies (\ref{Eqn:CauchyPb2}), replacing $F$ with $\tilde{F}$ defined by
$\tilde{F}(\vec{w},t) := -F (\vec{w},-t) $. Hence we may WLOG assume that $\phi \in \widetilde{\mathcal{P}}_{+,k}$ in the ``Local existence and parabolic smoothing'' part of Theorem \ref{Thm:MainDiff}, that for some $\breve{T} > 0 $ there exist $u_{1},u_{2} \in \mathcal{C} \left( [0, \breve{T}], H^{k} \right)$  solutions of (\ref{Eqn:CauchyPb2}) with $u_{1}(0) = u_{2}(0)$ such that $u_{q}(t) \in \mathcal{P}_{+,k}(t)$ holds for $q \in \{ 1,2 \} $ and for $ t \in [0, \breve{T}]$ in the ``Uniqueness'' part of Theorem \ref{Thm:MainDiff}, and that $\phi_{\infty} \in \widetilde{\mathcal{P}}_{+,k}$ in the ``Continuous dependence on initial data '' part of Theorem \ref{Thm:MainDiff}. We may also WLOG restrict our analysis to positive times in the proof of Theorem \ref{Thm:MainDisp}. Actually the reader can check the proof of  Theorem \ref{Thm:MainDisp} is a straightforward modification of that of Theorem \ref{Thm:MainDiff}: therefore it is omitted. \\
\\
We now discuss the main interest of this paper. Our goal is to find conditions as unrestricted as possible on smooth functions $F$ and on the data $u(0)$ for which
local existence, uniqueness, and continuous dependence on initial data hold for (\ref{Eqn:CauchyPb2}). Our second objective is to find properties of the solutions of
(\ref{Eqn:CauchyPb2}). In particular we would like to find conditions for which the equation shares similar properties as a diffusive one (local existence and parabolic smoothing in one direction, ill-posedness in the other direction) or a dispersive one (local existence in both directions). Observe that $F$ is not necessarily linear, which makes the linear analysis (in particular the Fourier analysis)
much less efficient. Observe also that $F$ depends on the first three derivatives of $u$, which makes our analysis delicate since the R.H.S of (\ref{Eqn:CauchyPb2})
seems (\text{a priori}) to have less regularity than the L.H.S of (\ref{Eqn:CauchyPb2}).  \\
In order to overcome the difficulties it is natural to try at first sight to prove energy estimates for a candidate of a solution of (\ref{Eqn:CauchyPb}) and energy estimates for differences of candidates of solutions of (\ref{Eqn:CauchyPb}): the goal of these estimates is, roughly speaking, to prove that for some $T > 0$ small enough at least, some relevant norms (such as the $H^{k}-$ norm) of these candidates lying in an appropriate space (such as $\mathcal{C} ([0,T], H^{k})$ ) are \textit{a priori} bounded on
$[0,T]$; then we aim at proving that these bound hold \textit{a posteriori} and that at least one of this candidates is indeed a solution of (\ref{Eqn:CauchyPb})
by a standard argument (such as the fixed point theorem). \\
Unfortunately it seems to be impossible to implement this strategy to our problem without any important modification since the R.H.S of (\ref{Eqn:CauchyPb2}) contains terms that are less regular than $u$. So we consider a problem (\ref{Eqn:Regul}) that approximates well (\ref{Eqn:CauchyPb}) (the rate of approximation is measured by a parameter $\epsilon > 0 $)  and that is regularizing: this is the parabolic regularization step. The regularizing property of (\ref{Eqn:Regul}) allows to construct a solution $u$ on an interval $[0, T_{\epsilon})$: see Section \ref{Sec:RegulLocWell}. Then we aim at proving energy estimates for  $u$. These estimates should allow, among other things, to control relevant norms of $u$ on a small interval $[0,T^{'}]$, with $T^{'} > 0$ that does not depend on $\epsilon$, which is prerequisite
to prove that there is a solution of (\ref{Eqn:CauchyPb2}) by a limit process. In order to derive the energy estimates it is natural to linearize w.r.t the excess of derivatives on the R.H.S of (\ref{Eqn:Regul}) by applying the Leibnitz rule \footnote{see Section \ref{Sec:ProofLemmas}, replacing the function $ (f,t) \rightarrow \Phi_{k'}(f,t)$ with the function $(f,t) \rightarrow 1$ \label{Foot:Repl}. } and then integration by parts to force signed quantities or easy quantities to control to appear \footnote{see Section \ref{Sec:EnergyEst}, with the substitution indicated in Footnote \ref{Foot:Repl}.}. It occurs that this strategy works provided that the pointwise-in-space-and-time condition below holds, namely

\begin{equation}
\begin{array}{l}
t \in [0,T^{'}]: \;  \left( k - \frac{15}{2} \right) \partial_{x} a_{3}(t) + a_{2} (t) \geq 0, \; \text{with}  \\
(a_{3}(t),a_{2}(t)) := \left(  \p_{\omega_{3}} F ( \overrightarrow{u(t)},t ), P (u(t),t) \right):
\end{array}
\label{Condition:Init}
\end{equation}
see Section \ref{Sec:EnergyEst} \footnote{In the expression above the values of $a_{3}(t)$ and $a_{2}(t)$ are given by the analysis in Section \ref{Sec:EnergyEst},
with the substitution  indicated in Footnote \ref{Foot:Repl}.}. In particular the condition above should be satisfied at $ t = 0 $ by the data $u(0)$. Ideally we would like this condition not to exist, since we would eventually like to construct a solution of (\ref{Eqn:CauchyPb2}) with a set of data $u(0)$ as large as possible. \\
To this end we use a gauge transformation (see Section \ref{Sec:Notation}): more precisely, for a gauge function $\Phi_{k}$ to be chosen, we work with the function $\Phi_{k} (u(t),t) \partial_{x}^{6} u(t) $ instead of the function $\partial_{x}^{6} u(t)$. We introduce the gauged energy functionals in Section \ref{Sec:EnergyDef} and we prove gauged energy estimates in Section \ref{Sec:EnergyEst}. The gauged energy estimates are supposed to play the same role as the standard energy estimates. For then to be effective, the condition below should be satisfied

\begin{equation}
\begin{array}{l}
t \in [0,T^{'}]: \,  \left( k  -  \frac{15}{2}  \right) \partial_{x} a_{3}(t) + a_{2} (t) \geq 0, \; \text{with}  \\
\left( a_{3}(t), a_{2}(t) \right) :=  \left( \p_{\omega_{3}} F \left( \overrightarrow{u(t)},t \right), 3 \Phi_{k}(u(t),t) \p_{x} \Phi_{k}^{-1}(u(t),t) +
P \left( u(t), t \right) \right):
\end{array}
\label{Condition:Final}
\end{equation}
see again Section \ref{Sec:EnergyEst}. So for the condition (\ref{Condition:Final}) not to exist one would ideally like to choose $\Phi_{k}$ such that
$ t \in [0,T^{'}]: \; \left(  k - \frac{15}{2} \right) \partial_{x} a_{3}(t) + a_{2} (t) = 0 $. This equation can be solved explicitly on the real line; nevertheless the solutions that we get on the real line are not periodic. Pick one of these solutions. In order to make it periodic we substract an average: see (\ref{Eqn:DefPhip}).
Then choose the gauge function to be this new function. With this choice for the gauge function, we get $ t \in [0,T^{'}]: \, \left( k - \frac{15}{2}  \right) \partial_{x} a_{3}(t) + a_{2} (t) = Q(f,t) $: see (\ref{Eqn:CrucEq}). So the condition becomes $ (*): \, t \in [0,T']: \, \inf \limits_{x \in \T} Q (u(t),t)(x) \geq 0 $. In particular this condition should be satisfied at $t=0$, i.e $ \inf \limits_{x \in \T} Q(f,0) \geq 0 $. Now if we assume that $ \inf \limits_{x \in \T}  Q(u(0),0)(x) = 0 $ then it might be very difficult to ensure that $(*)$ holds by using a continuity-in-time argument; so we come to the conclusion that the right assumption on the data is
$ \inf \limits_{x \in \T} Q(u(0),0)(x) > 0 $. So we see from Remark \ref{Rem:Qother} that the gauge transform has turned the pointwise
and $k-$ dependent condition \ref{Condition:Init} at $t=0$ into an average and $k-$ independent one (namely $ \tilde{\delta} (u(0)) > 0 $) , which is consistent with the fact that we want to find conditions as unrestricted as possible for which our theorem holds. \\
In Section \ref{Sec:EnergyDiffEst}, we prove gauged energy estimates for the difference of two solutions of the regularized problem with smoothed out data that is also regularized by operators of the form $J_{\epsilon,s}$ (this is the Bona-Smith approximation \cite{BonaS}). In Section \ref{Sec:ProofThm}, we prove Theorem \ref{Thm:MainDiff} by using a limit process. Throughout this paper, one has to perform a delicate analysis to make sure that the solution of the regularized problem and the solution of (\ref{Eqn:CauchyPb2}) belong to similar sets as the data. \\
\\
\textbf{Acknowledgments}: The author would like to thank Kotaro Tsugawa for interesting discussions related to this problem. \\
\\
\textbf{Funding}: This article was partially funded by an URB grant (ID: $3245$) from the American University of Beirut. This research was also partially conducted in Japan and
funded by a (JSPS) Kakenhi grant  [ 15K17570 to T.R.].

\section{Notation}
\label{Sec:Notation}

In this section we introduce some notation that we use throughout this paper. \\
\\
In this paper, unless otherwise specified, we do not mention for sake of notation simplicity the spaces to which some functions (or more generally distributions) belong in some propositions or lemmas: this exercise is left to the reader. \\
Let $f$ be a function on $ \T $ . If $n \in \N $ then $\p_{x}^{n} f$ denotes the  $n^{th}$ derivative of $f$.
Let $m \in \mathbb{R} $, $n \in \mathbb{Z}$, and $\hat{f}(n)$ be the $n^{th}-$ Fourier coefficient of $f$. Let $D^{m}$ be the operator such that
$\widehat{D^{m} f}(n) := (in)^{m} \hat{f}(n)$. It is well-known that if $m \in \N$ then $D^{m} = \p_{x}^{m}$. We define $\| f \|_{\dot{H}^{m}} :=  \| \{  (in)^{m} \hat{f}(n) \}_{n \in \mathbb{Z}} \|_{l^{2}}$. \\
\\
In this paper we use a gauge transform at a time $t$, in the spirit of \cite{Hayashi,HayashiOzawa1, HayashiOzawa2}. Let $k' \in \R $. If $ \delta(\vec{f},t) > 0 $
and $ \partial_{\omega_{3}} F (\vec{f},t ) \in \mathcal{C}^{0}$ then $\partial_{\omega_{3}} F (\vec{f},t) $ has a constant sign and we define the gauge function
at $t$

\EQQARRLAB{
\Phi_{k'}(f,t)(x) & := \left| \p_{\omega_3} F(\vec{f},t)(x) \right|^{\frac{2k'-15}{6}}
e^{\int_{0}^{x} \frac{1}{3} \left( \frac{P(f,t)}{_{\p_{\omega_3} F (\vec{f},t) }} - \left[ \frac{P(f,t)}{_{\p_{\omega_3}F (\vec{f},t)}} \right]_{ave}  \right) \; dx' } \cdot
\label{Eqn:DefPhip}
}
Observe that the term $ \left[ \frac{P(f,t)}{_{\p_{\omega_3} F (\vec{f},t)}} \right]_{ave}$ in (\ref{Eqn:DefPhip}) makes the function $\Phi_{k'}(f,t)$ periodic and, consequently, it is defined on $\T$. Observe also that

\EQQARRLAB{
\left( k' - \frac{15}{2} \right) \p_{x} \p_{\omega_3} F(\vec{f},t)
+ \left( 3 \Phi_{k'}(f,t) \p_{x} \Phi_{k'}^{-1}(f,t) \p_{\omega_3} F (\vec{f},t) + P(f,t) \right)  & = Q(f,t) \cdot
\label{Eqn:CrucEq}
}
This is a crucial observation that we will often use throughout this paper. We will often work with the function
$\Phi_{k'} \left( u(t),t \right) \partial_{x}^{6} u(t) $ and prove estimates that involve this function. \\
\\
We write $C := C (\alpha_{1},...,\alpha_{n}) $ if $C$ is a constant depending on the variables $\alpha_{1}$,...,$\alpha_{n}$.
If $ C := C( \alpha_{1}, \alpha_{2}) $  is a constant depending on two variables $\alpha_{1}$ and $\alpha_{2}$, then we may take the liberty
to write only $C$ if we do not want to emphasize the dependance on $\alpha_{1}$ and $\alpha_{2}$, or to write only $C := C(\alpha_{1})$ if we do
not want to emphasize the dependance on $\alpha_{2}$, in order to simplify the notation. This convention naturally extends to a constant depending on
several variables. We write $a \lesssim b $ if there exists a constant $C > 0$ such that $a \leq C b$. We write $ a \lesssim_{\alpha_{1},...,\alpha_{n}} b$ if there exists
$C := C( \alpha_{1},...,\alpha_{n})$ such that $a \leq C b$.  If $a \lesssim_{\alpha_{1},\alpha_{2}} b $ then we may take the liberty to write only
$a \lesssim b $ if we do not want to emphasize the dependance on $\alpha_{1}$ and $\alpha_{2}$ or to write only $ a \lesssim_{\alpha_{1}} b $ if we do not
want to emphasize the dependance on $\alpha_{2}$. This convention naturally extends to $a \lesssim_{\alpha_{1},...,\alpha{n}} b$. We write
$a \ll b$ if there exists a small constant $ c > 0$ such that $a \leq c b$. Similarly we write $a \ll_{\alpha_{1},...,\alpha_{n}} b $ if there
exists  $c := c \left( \alpha_{1},...,\alpha_{n} \right)$ such that $a \leq c b$. We use the same conventions for $c$ as for $C$.      \\
\\
Let $x+ = x + \epsilon$ and $x- = x - \epsilon$ with $ 0 <  \epsilon \ll 1$. Let $p \in \N$ and $q \in \mathbb{R}$. We define $\Combin{q}{p} := \frac{q(q-1)...(q-(p-1))}{p!}  $
for $p \neq 0$ and $\Combin{q}{0} := 1 $. \\
\\
We recall some lemmas. The first lemma shows that we can control the $H^{k'}-$ norm of a product of two functions $f$ and $g$:

\begin{lem}
Let $k' \geq 0 $. Let $f$ and $g$ be two functions. Then

\EQQARR{
\| f g \|_{H^{k'}} \lesssim \| f \|_{H^{k'}} \| g \|_{H^{\frac{1}{2}+}} + \| f \|_{H^{\frac{1}{2}+}} \| g \|_{H^{k'}} \cdot
}
\label{Lem:prod}
\end{lem}
For a proof see e.g \cite{Iorio} and references therein \footnote{The proof of Lemma \ref{Lem:prod} is  a slight modification of that of Theorem $3.200.$ in \cite{Iorio} }.  Given $\epsilon \in (0,1]$ and $s \geq 0$, let $J_{\epsilon,s}$ be the operator
defined by $ \widehat{J_{\epsilon,s} f}(n) := e^{- \epsilon \langle n \rangle^{s}} \hat{f}(n)$ for $n \in \mathbb{Z}$. The second lemma
shows that this operator have some smoothing properties:

\begin{lem}
Let $0 \leq j \leq s$ and $l \geq 0$. The following holds:

\EQQARR{
\| J_{\epsilon,s} f  - f \|_{H^{s-j}} \lesssim \epsilon^{\frac{j}{s}}  \| f \|_{H^{s}}, \; \text{and} \;
\| J_{\epsilon,s} f  \|_{H^{s+l}} \lesssim \epsilon^{-\frac{l}{s}} \| f \|_{H^{s}} \cdot
}
\label{Lem:MollEst}
\end{lem}
For a proof see e.g \cite{Iorio}. It will be often used in the propositions where we implement the Bona-Smith approximation argument
(\cite{BonaS}) \footnote{see also \cite{Iorio} for an exposition of the ideas of this argument}.

\section{Preliminary lemmas}
\label{Sec:PrelimLemma}

In this section we prove some lemmas.\\
\\
In the first lemma, we estimate an expression that involves two functions $f$ and $g$. The expression is written as a sum of two terms: the main term and the
remainder. Observe that in the remainder the excess of $m$ derivatives of $g$ is transferred on $f$.

\begin{lem}
Let $f$ and $g$ be two functions. Then

\begin{equation}
\begin{array}{l}
k' \geq 0, m \in \mathbb{N}^{*}: \, \left\| D^{k'} (f \p_x^{m} g) -  \sum \limits_{l=0}^{m-1} \binom{k'}{l} D^{l} f  D^{k'-l} \p_{x}^{m} g  \right\|_{L^{2}} \lesssim
\| f \|_{H^{k'+m}} \| g \|_{H^{\frac{1}{2}+}}  + \| f \|_{H^{ \left( m + \frac{1}{2} \right)+}} \| g \|_{H^{k'}} \cdot
\end{array}
\label{Eqn:LemD1}
\end{equation}







\label{lem:LemD}
\end{lem}

\begin{proof}

We can write $L.H.S \;  \text{of} \; (\ref{Eqn:LemD1}) = \left\| \sum \limits_{n_1 \in \mathbb{Z}} X(n,n_1) \hat{f}(n-n_1) \hat{g}(n_1) \right\|_{l^{2}}$ with

\EQQARR{
X(n,n_1) & := (in)^{k'} (in_1)^{m} - \sum \limits_{l=0}^{m-1}  \binom{k'}{l} (i (n-n_1))^{l} (in_1)^{k'-l+m} \cdot
}
We estimate $X(n,n_1)$. If $|n-n_1| \gtrsim |n_1| $ then clearly $|X(n,n_1)| \lesssim |n-n_1|^{k'+m} +
+ |n -n_{1}|^{m} | n_1 |^{k'} $. If $|n-n_1| \ll |n_1| $ then we factor out $n_1^{m}$ and
we apply the Taylor formula for $n_1 \neq 0$ to get $|X(n,n_1)| \lesssim |n_1|^{m} |n-n_1|^{m} |n_1|^{k'-m} \lesssim |n-n_1|^{m} |n_1|^{k'} $. If $n_1= 0$ then
the bound found for $X(n,n_1)$ clearly holds. The Young inequality and the Cauchy-Schwarz inequality yield

\EQQARRLAB{
 L.H.S \;  \text{of} \; (\ref{Eqn:LemD1})  & \lesssim
\| f \|_{H^{k'+m}} \| \hat{g}(n) \|_{l^{1}}  + \| n^{m} \hat{f}(n) \|_{l^{1}} \| g \|_{H^{k'}} \\
& \lesssim \| f \|_{H^{k'+m}}  \| g \|_{H^{\frac{1}{2}+}}
+ \| f \|_{H^{ \left( m+ \frac{1}{2} \right) +}} \| g \|_{H^{k'}}  \cdot
}

\end{proof}

In the next two lemmas we prove some estimates involving derivatives of powers of functions, quotient of functions, exponential of functions, differences of exponentials
of functions. We also prove nonlinear estimates. These estimates will be used to prove some estimates that involve the gauge function.

\begin{lem}
Let $\beta \in \mathbb{R}$. Let $\bar{\delta} > 0$ and $(k', K) \in (\mathbb{R}^{+})^{2}$. Then the following holds:

\begin{itemize}

\item Let $f$ be a function. Assume that $ \| f \|_{H^{\frac{3}{2}+}} \leq K $ and that $ \inf \limits_{x \in \T} |f(x)| \geq \bar{\delta} $. Then there exists
$C := C(\bar{\delta},K) > 0$ such that

\EQQARRLAB{
\| f^{\beta} \|_{H^{k'}} & \leq C \left( 1 + \| f \|_{H^{k'}} \right)
\label{Eqn:DerivPower}
}

\item Let $f$ and $g$ be two functions. Assume that $ \| f \|_{H^{\frac{1}{2}+}} \leq K $,  $ \| g \|_{H^{\frac{3}{2}+}} \leq K $, and   $ \inf \limits_{x \in \T} |g(x)| \geq \bar{\delta}$. Then there exists $C := C(\bar{\delta},K) > 0$ such that

\EQQARRLAB{
\left\| \frac{f}{g} \right\|_{H^{k'}} & \leq C \left( 1 + \| f \|_{H^{k'}} + \| g \|_{H^{k'}} \right)
\label{Eqn:DerivFrac}
}
There exists $C := C (\bar{\delta},K) > 0$ such that

\EQQARRLAB{
\| g \|_{H^{\max \left( k', \frac{3}{2}+ \right)}}  \leq K  \Longrightarrow \left\| \frac{f}{g} \right\|_{H^{k'}}  \leq
C \| f \|_{H^{\max{\left( k', \frac{1}{2}+ \right)}}}  \\
\label{Eqn:DerivFrac2}
}

\item Let $f$ be a function. Then

\EQQARRLAB{
k' \in [0,1]: & \left\| \int_{0}^{x} f(x') \; dx' \right\|_{H^{k'}} \lesssim \| f \|_{L^{2}}, \; \text{and} \\
& \\
k' \geq 1: &   \left\| \int_{0}^{x} f(x') \; dx' \right\|_{H^{k'}} \lesssim \| f \|_{H^{k'-1}} \cdot
\label{Eqn:ExpIntBasic}
}

\item Let $f$ be a function. Assume that $ \| f \|_{H^{\frac{1}{2}+}} \leq K $. Then there exists $C := C (K) > 0$ such that

\EQQARRLAB{
k' \in [0,1]: &  \left\| e^{\int_{0}^{x} f(x') \; dx'} \right\|_{H^{k'}} \leq C, \; \text{and}  \\
&  \\
k' \geq 1: & \left\| e^{\int_{0}^{x} f(x') \; dx'} \right\|_{H^{k'}}  \leq C \left( 1 + \| f \|_{H^{k' -1}} \right) \cdot
\label{Eqn:ExpInt}
}

\item Let $f$ and $g$ be two functions. Assume that $ \| f \|_{H^{\frac{1}{2}+}} \leq K $ and that $\| g \|_{H^{\frac{1}{2}+}} \leq K $. Assume also that
$\| f \|_{H^{k'-1}} \leq K$ and that $\| g \|_{H^{k'-1}} \leq K $ if $k' \geq 1$. Then there exists $C := C(K)$ such that

\EQQARRLAB{
k' \in [0,1]: & \left\| e^{\int_{0}^{x} f(x') \; d x'} - e^{\int_{0}^{x} g(x') \; d x'} \right\|_{H^{k'}}   \leq C \| f - g \|_{L^{2}}, \; \text{and}  \\
k' \geq 1: & \left\| e^{\int_{0}^{x} f(x') \; d x'} - e^{\int_{0}^{x} g(x') \; d x'}  \right\|_{H^{k'}}  \leq C \| f - g \|_{H^{k'-1}} \cdot
\label{Eqn:ExpIntDiff}
}

\end{itemize}
\label{lem:EstDerivOp}
\end{lem}

\begin{rem}
The proof shows that the constants $C$ depending only on $K$ can be chosen as continuous functions of $K$ that increase as $K$ increases. It also shows that
the constants $C$ depending on $\bar{\delta}$ and $K$ can be chosen as continuous functions of $(\bar{\delta},K)$ that increase as $K$ increases
and increase as $\bar{\delta}$ decreases.
\label{Rem:EstDerivOp}
\end{rem}

\begin{proof}

We prove (\ref{Eqn:DerivPower}) by an induction process. \\
Let $ 0 \leq k' \leq 1$. The H\"older inequality, the Sobolev embedding  $H^{\frac{1}{2}+} \hookrightarrow L^{\infty}$ applied to the cases
$\beta < 1$ and the estimate $\| f^{\beta -1} \|_{L^{\infty}} \lesssim 1 $ if $\beta \geq 1$ yield

\EQQARR{
\| f^{\beta} \|_{L^{2}} & \lesssim \langle \| f \|_{H^{\frac{1}{2}+}} \rangle^{\beta-1} \| f \|_{L^{2}} \; \text{and} \\
\| \p_{x} f^{\beta} \|_{L^{2}} & \lesssim  \langle \| f \| \rangle^{\beta - 1}_{H^{\frac{1}{2}+}} \| \p_{x} f \|_{L^{2}} \lesssim \| \p_{x} f \|_{L^{2}} \cdot
}
Hence (\ref{Eqn:DerivPower}) holds for this range of $k'$. Assume now that (\ref{Eqn:DerivPower}) holds for all $ 0 \leq k' \leq r$ with $r \in \mathbb{N}^{*}$. We see
from Lemma \ref{Lem:prod} that

\EQQARR{

\left\| \p_{x} f^{\beta} \right\|_{H^{k'-1}} &  \lesssim \langle \| f^{\beta-1} \|_{H^{\frac{1}{2}+}} \rangle \| \p_{x} f \|_{H^{k'-1}}
+ \langle \| f^{\beta -1} \|_{H^{k'-1}} \rangle \| \p_{x} f \|_{H^{\frac{1}{2}+}} \\
& \leq C \left( 1 + \| f \|_{H^{k'}}  \right) \cdot
}
Hence  (\ref{Eqn:DerivPower}) holds for $r \leq k' \leq r + 1$. \\
Next we prove (\ref{Eqn:DerivFrac}) and (\ref{Eqn:DerivFrac2}). We see from (\ref{Eqn:DerivPower}) that

\EQQARR{
\left\| \frac{f}{g} \right\|_{H^{k'}} & \lesssim  \| f \|_{H^{k'}} \| g^{-1} \|_{H^{\frac{1}{2}+}} + \| g^{-1} \|_{H^{k'}} \| f \|_{H^{\frac{1}{2}+}} \\
& \lesssim \left( 1 + \| g \|_{H^{\frac{1}{2}+}} \right) \| f \|_{H^{k'}} + \left( 1 + \| g \|_{H^{k'}} \right) \| f \|_{H^{\frac{1}{2}+}} \cdot
}
Hence (\ref{Eqn:DerivFrac}) and (\ref{Eqn:DerivFrac2}) hold. \\
Next we prove (\ref{Eqn:ExpIntBasic}). Let $I:= \int_{0}^{x} f(x') \; dx'$. Let $ 0 \leq k' \leq 1$ . The Minkowski inequality shows that
$\| I \|_{L^{2}} \lesssim \| f \|_{L^{2}} $; we also have $\| \p_{x} I  \|_{L^{2}} = \| f \|_{L^{2}} $; hence
(\ref{Eqn:ExpIntBasic}) follows from $\| I \|_{H^{k'}} \lesssim \| I \|^{1- k'}_{L^{2}} \| I \|^{k'}_{H^{1}}$.  Now assume that $k'> 1$. Then
we see from $\| \p_{x} I  \|_{H^{k'-1}} = \| f \|_{H^{k'-1}}$ that (\ref{Eqn:ExpIntBasic}) holds. \\
Next we also prove (\ref{Eqn:ExpInt}). Let $ 0 \leq k' \leq 1$. There exists $C' > 0$ such that

\EQQARRLAB{
\left\| e^{\int_{0}^{x} f(x') \; dx'} \right\|_{L^{2}} & \lesssim  \left\| e^{\int_{0}^{x} f(x') \; dx'} \right\|_{L^{\infty}} \lesssim  e^{C' \| f \|_{L^{2}}}
 \lesssim 1 \cdot
\label{Eqn:EstExpf}
}
We also have

\EQQARRLAB{
\left\| \p_x e^{\int_{0}^{x} f(x') \; dx'} \right\|_{L^{2}} & \lesssim \| f \|_{L^{2}} \left\| e^{\int_{0}^{x} f(x') \; dx'} \right\|_{L^{\infty}}   \\
& \lesssim 1 \cdot
\label{Eqn:ExpDerExpf}
}
Hence (\ref{Eqn:ExpInt}) holds. \\
Assume that (\ref{Eqn:ExpInt}) holds for all $ 0 \leq  k' \leq r $ with $r \in \mathbb{N}^{*}$. If $ r \leq k' \leq r +1 $ then

\EQQARRLAB{
\left\| \p_x  e^{\int_{0}^{x} f(x') \; dx'} \right\|_{H^{k' -1}} & \lesssim \left\| e^{\int_{0}^{x} f(x') \; dx'} \right\|_{H^{\frac{1}{2}+}}
\| f \|_{H^{k'-1}} + \left\| e^{\int_{0}^{x} f(x') \; dx'} \right\|_{H^{k'-1}} \| f \|_{H^{\frac{1}{2}+}} \\
& \lesssim  1 + \| f \|_{H^{k'-1}} \cdot
\label{Eqn:ExpDerkExp}
}
Hence (\ref{Eqn:ExpInt}) holds for $ r \leq k' \leq r + 1 $. \\
Next we prove (\ref{Eqn:ExpIntDiff}). Let $h(x'):= f(x') - g(x')$, $X := e^{\int_{0}^{x} h(x') \; d x'} - 1$, and $ Y := e^{\int_{0}^{x} g(x') \; d x'}$. \\
First we estimate $\| X \|_{H^{m}}$. We claim that

\EQQARRLAB{
m \in [0,1]: & \|  X  \|_{H^{m}}  \lesssim \| h \|_{L^{2}} \\
m  \geq 1: & \| X \|_{H^{m}} \lesssim \| h \|_{H^{m-1}}   \cdot
\label{Eqn:EstdiffHm}
}
Indeed, define $\bar{m}$ as follows: let $\bar{m}:= 0$ if $m \in [0,1]$ and let $\bar{m} := m-1$ if $m \geq 1$. If $\| h \|_{H^{\bar{m}}} \gtrsim 1$ then it
follows from (\ref{Eqn:ExpInt}) and $\| \mathbf{1} \|_{H^{m}} \lesssim 1$. If $\| h \|_{H^{\bar{m}}} \ll 1$ then we use again an induction process. First let us assume that
$\bar{m} \in [0,1]$. Observe that  $ \left| \int_{0}^{x} h(x') \: dx' \right| \lesssim \| h \|_{L^{2}} \ll 1$; since $e^{y} - 1 \approx y $ for $|y| \ll 1$
we get $ \| X \|_{L^{2}}  \lesssim  \| h \|_{L^{2}}$; proceeding as in (\ref{Eqn:ExpDerExpf}) we get
$ \| \p_{x} X \|_{L^{2}} \lesssim  \| h \|_{L^{2}} $ and consequently  (\ref{Eqn:EstdiffHm}) holds for $ \bar{m} \in [0,1]$. Now assuming that
(\ref{Eqn:EstdiffHm}) holds for $0 \leq m \leq r $,  we proceed in a similar way as (\ref{Eqn:ExpDerkExp}) in order to estimate
$\| X \|_{H^{m}}$ for $r \leq m \leq r + 1 $ and we find that (\ref{Eqn:EstdiffHm}) holds. \\
Hence using also (\ref{Eqn:ExpInt}) we get

\EQQARR{
\left\| e^{\int_{0}^{x} f(x') \; dx'}  - e^{\int_{0}^{x} g(x') \; dx'}   \right\|_{H^{k'}} & \lesssim  \| X \|_{H^{k'}}
\| Y \|_{H^{\frac{1}{2}+}} + \| X \|_{H^{\frac{1}{2}+}} \| Y \|_{H^{k'}} \\
& \lesssim R.H.S \; \; \text{of} \; \; (\ref{Eqn:ExpIntDiff}) \cdot
}

\end{proof}

\begin{lem}
Let $G :  \left( \mathbb{R}^{4} \times \T \right) \times \mathbb{R} \longrightarrow \mathbb{R}$ be a $\mathcal{C}^{\infty}-$ function. Let $I \subset \mathbb{R}$ be an interval such
that  $|I| \leq 1$. Let $(k',K) \in (\mathbb{R}^{+})^{2}$. The following holds:

\begin{itemize}

\item Let $f$ be a function. Assume that $ \| f \|_{H^{\frac{9}{2}+}} \leq K $. Then there exists $C: = C(K) > 0$ such that for all $t \in I$

\EQQARRLAB{
\left\| G (\vec{f},t) \right\|_{H^{k'}} & \leq C \left( 1 + \| f \|_{H^{k'+3}} \right)
\label{Eqn:EstFDeriv}
}
Let $\beta \in \mathbb{R}$. Assume that $ \| f \|_{H^{\frac{9}{2}+}} \leq K $, and that there exists $\bar{\delta} > 0$ such that
for all $t \in I$, $ \inf \limits_{x \in \T} \left| G (\vec{f},t)(x) \right| \geq \bar{\delta}$. Then there exists $C := C(\bar{\delta}, K) > 0$ such that for $t \in I$

\EQQARRLAB{
\left\| G^{\beta} (\vec{f},t) \right\|_{H^{k'}} & \leq C \left( 1 + \| f \|_{H^{k'+3}} \right)
\label{Eqn:EstFDeriv2}
}

\item Let $f$ and $g$ be two functions. Assume that $ \| f \|_{H^{\frac{9}{2}+}} \leq K $ , $\| g \|_{H^{\frac{9}{2}+}} \leq K$,
$\| f \|_{H^{k'+3}} \leq K$, and $\| g \|_{H^{k'+3}} \leq K $ . Then there exists $ C := C(K) > 0$ such that for all $t_{1},t_{2} \in I$

\EQQARRLAB{
\left\| G (\vec{f},t_{2}) -  G  (\vec{g},t_{1})  \right\|_{H^{k'}} & \leq C \left( \| f - g \|_{H^{\max \left( k' +3, \frac{7}{2}+ \right)}} + |t_{2} - t_{1}| \right) \cdot
\label{Eqn:DiffEstFDeriv}
}
Let $\beta \in \mathbb{R} $. Assume that $ \| f \|_{H^{\frac{9}{2}+}} \leq K $, $ \| g \|_{H^{\frac{9}{2}+}} \leq K $,  $\| f \|_{H^{k'+3}} \leq K $,
$\| g \|_{H^{k'+3}} \leq K$, and that there exists $\bar{\delta} > 0 $ such that for all $t \in I$ and for all
$\theta \in [0,1]$ we have  $\inf \limits_{x \in \T} \left| G \left( \overrightarrow{h_{\theta}} ,t  \right)(x) \right| \geq \bar{\delta}$ , with
$h_{\theta} := \theta f + (1- \theta) g $. Then there exists $C := C( \bar{\delta},K) > 0$ such that for all $t_{1},t_{2} \in I$

\EQQARRLAB{
\left\| G^{\beta}(\vec{f},t_{2}) - G^{\beta}(\vec{g},t_{1}) \right\|_{H^{k'}} & \leq C \left( \| f - g \|_{H^{\max \left( k' +3, \frac{7}{2}+ \right)}}
+ |t_{2} - t_{1}| \right)
\cdot
\label{Eqn:DiffEstFDeriv2}
}

\end{itemize}
\label{lem:EstDerivF}
\end{lem}

\begin{rem}
Regarding the constants in Lemma \ref{lem:EstDerivF}, we refer to Remark \ref{Rem:EstDerivOp}.
\label{Rem:EstDerivF}
\end{rem}

\begin{proof}
We prove Lemma \ref{lem:EstDerivF} by an induction process. \\
\\
First we prove (\ref{Eqn:EstFDeriv}). \\
Let $ 0 \leq k' \leq 1$. The H\"older inequality, the Sobolev embedding $H^{\frac{1}{2}+} \hookrightarrow L^{\infty}$, and the composition rule imply that

\EQQARRLAB{
\left\| G(\vec{f},t) \right\|_{L^{2}} & \lesssim  \sup \limits_{| (\vec{\omega},t) | \lesssim \langle K \rangle }
\left| G(\vec{\omega},t) \right| \lesssim 1, \; \text{and} \\
\left\| \p_x  \left( G(\vec{f},t) \right)  \right\|_{L^{2}} & \lesssim
 \sup \limits_{\substack{\gamma \in \N^{6}: |\gamma| \leq 1 \\ |(\vec{\omega},t)| \lesssim \langle K \rangle }}
 \left| \p^{\gamma} G(\vec{\omega},t) \right|   \langle \| f \|_{H^{4}} \rangle \lesssim 1 \cdot
\label{Eqn:GInducInit}
}
Hence (\ref{Eqn:EstFDeriv}) holds. \\
Let $r' \in \N^{*}$. Assume that (\ref{Eqn:EstFDeriv}) holds for all $ 0 \leq  k' \leq r $ and for all smooth functions $G$. The induction assumption and
Lemma \ref{Lem:prod} yield

\EQQARRLAB{
\left\| \p_x  \left( G(\vec{f},t) \right) \right\|_{H^{k'-1}} & \lesssim
\sup \limits_{\substack{\gamma \in \N^{6}: |\gamma| \leq 1 \\ |(\vec{\omega},t)| \lesssim \langle K \rangle }}
\| \p^{\gamma} G(\vec{\omega},t) \|_{H^{\frac{1}{2}+}}   \| f \|_{H^{k'+3}} \\
& + \sup \limits_{\substack{\gamma \in \N^{6}: |\gamma| \leq 1 \\ |(\vec{\omega},t)| \lesssim \langle K \rangle }}
\| \p^{\gamma} G(\vec{\omega},t) \|_{H^{k'-1}} \langle \| f \|_{H^{\frac{9}{2}+}} \rangle \\
& \leq C \left( 1 + \| f \|_{H^{k'+3}} \right)
\label{Eqn:GInduc}
}
Hence (\ref{Eqn:EstFDeriv}) holds for $ r \leq k' \leq r + 1 $. \\
\\
The proof of (\ref{Eqn:EstFDeriv2}) is very similar to that (\ref{Eqn:EstFDeriv}). Since
$ \inf \limits_{x \in \T} \left|  G(\vec{f},t)(x) \right| \geq \bar{\delta}$ this implies that $G$, $G^{\beta}$, and $G^{\beta-1}$ are $\mathcal{C}^{\infty}-$
functions. Hence (\ref{Eqn:EstFDeriv2}) holds for $0 \leq k' \leq 1$  (resp. $r \leq k' \leq r +1$) by slightly modifying (\ref{Eqn:GInducInit})
(resp. (\ref{Eqn:GInduc})). \\
\\
Next we prove (\ref{Eqn:DiffEstFDeriv}). Let $t_{\theta} := \theta t_{1} + (1- \theta) t_{2} $.  Write

\EQQARR{
G (\vec{f},t_{2})  - G (\vec{g},t_{1}) = \int_{0}^{1}
\left[
\begin{array}{l}
\sum \limits_{i = 0}^{3}
\p_{\omega_{i}} G \left( \overrightarrow{h_{\theta}}, t_{\theta}  \right)( \p_{x_i} f - \p_{x_i} g ) + \\
\p_{t} G \left( \overrightarrow{h_{\theta}} , t_{\theta} \right) (t_{2} - t_{1})
\end{array}
\right]
\; d \theta
}
We get from (\ref{Eqn:EstFDeriv}) and Lemma \ref{Lem:prod}

\EQQARR{
\left\| G (\vec{f},t_{2})  - G (\vec{g},t_{1}) \right\|_{H^{k'}} & \lesssim
\sup \limits_{\theta \in [0,1]}   \left\| \p_{t} G \left( \overrightarrow{h_{\theta}}, \ t_{\theta}  \right) \right\|_{H^{k'}} |t_{2} - t_{1}|
+ \sup  \limits_{\substack{ i \in \{0,..,3 \} \\  \theta \in [0,1] } } \left\| \p_{\omega_i} G \left(  \overrightarrow{h_{\theta}}, t_{\theta} \right) \right\|_{H^{k'}}
\| f - g \|_{H^{\frac{7}{2}+}} \\
& + \sup  \limits_{\substack{ i \in \{0,..,3 \} \\  \theta \in [0,1] } } \left\| \p_{\omega_i} G \left(  \overrightarrow{h_{\theta}},
t_{\theta} \right)
\right\|_{H^{\frac{1}{2}+}} \| f - g \|_{H^{k'+3}} \\
& \lesssim ( 1 + \| f \|_{H^{k'+3}} + \| g \|_{H^{k'+3}} ) \left( \| f - g \|_{H^{\frac{7}{2}+}} + |t_{2} - t_{1}| \right) \\
& + ( 1 + \| f \|_{H^{\frac{7}{2}+}} + \| g \|_{H^{\frac{7}{2}+}} ) \| f - g \|_{H^{k'+3}} \\
& \leq C \left( \| f - g \|_{H^{\max \left( k' +3, \frac{7}{2}+ \right)}} + |t_{2}- t_{1}| \right)
}
Finally we prove (\ref{Eqn:DiffEstFDeriv2}). We have

\EQQARR{
\left\| G^{\beta} ( \vec{f},t_{2})  - G^{\beta} (\vec{g},t_{1}) \right\|_{H^{k'}}  \\
\lesssim
\sup \limits_{\theta \in [0,1]}
\left(
\begin{array}{l}
\left\| G^{\beta -1}  \left( \overrightarrow{h_{\theta}} , t_{\theta}  \right) \right\|_{H^{\frac{1}{2}+}}
\left\| \p_{t} G \left(  \overrightarrow{h_{\theta}},   t_{\theta} \right) \right\|_{H^{k'}}
+  \left\| G^{\beta -1}  \left( \overrightarrow{h_{\theta}}, t_{\theta} \right) \right\|_{H^{k'}}
\left\| \p_{t} G \left(  \overrightarrow{h_{\theta}}, t_{\theta} \right) \right\|_{H^{\frac{1}{2}+}}
\end{array}
\right) |t_{2}- t_{1} | \\
+ \sup \limits_{\substack{ i \in \{-1,..,3 \} \\  \theta \in [0,1] } }
\left\| G^{\beta-1} \left( \overrightarrow{h_{\theta}}, t_{\theta} \right)  \right\|_{H^{\frac{1}{2}+}}
\left\| \p_{\omega_i} G \left( \overrightarrow{h_{\theta}},t_{\theta} \right) \right\|_{H^{\frac{1}{2}+}}
\| f - g \|_{H^{k'+3}} \\
+ \sup \limits_{\substack{ i \in \{-1,..,3 \} \\  \theta \in [0,1] } }
\left\| G^{\beta-1} \left( \overrightarrow{h_{\theta}},t_{\theta} \right)  \right\|_{H^{k'}}
\left\| \p_{\omega_i} G \left( \overrightarrow{h_{\theta}}, t_{\theta} \right) \right\|_{H^{\frac{1}{2}+}}
\| f - g \|_{H^{\frac{7}{2}+}} \\
+ \sup \limits_{\substack{ i \in \{-1,..,3 \} \\  \theta \in [0,1] } }
\left\| G^{\beta-1} \left( \overrightarrow{h_{\theta}}, t_{\theta}  \right)  \right\|_{H^{\frac{1}{2}+}}
\left\| \p_{\omega_i} G \left( \overrightarrow{h_{\theta}}, t_{\theta} \right) \right\|_{H^{k'}}
\| f - g \|_{H^{\frac{7}{2}+}} \\
\lesssim \left(  1 +  \| f \|_{H^{\max \left( k' +3, \frac{7}{2}+ \right)}} + \| g \|_{H^{\max \left( k' +3, \frac{7}{2}+ \right)}} \right)^{2}
\left( \| f - g \|_{H^{\max \left( k' +3, \frac{7}{2}+ \right)}} + |t_{2} - t_{1}| \right) \\
\lesssim \| f - g \|_{H^{\max \left( k' +3, \frac{7}{2}+ \right)}} + |t_{2} - t_{1}| \cdot
}

\end{proof}

The next lemma below allows to prove estimates involving the derivatives of $\Phi_{k'}(f,t)$. More precisely:

\begin{lem}
Let $I \subset \mathbb{R}$ be an interval such that $|I| \leq 1$. Let $(k',K) \in \mathbb{R}^{+} \times \mathbb{R}^{+}$. The following holds:

\begin{enumerate}

\item  Let $f$ be a function. Assume that $ \| f \|_{H^{\frac{9}{2}+}} \leq K $ and that for all $t \in I$ $ \inf \limits_{x \in \T} | \p_{\omega_3} F (\overrightarrow{f},t)(x) | \geq \bar{\delta} $. Then there exists $C := C (\bar{\delta}, K) > 0 $ such that for all $t \in I$

\EQQARRLAB{
\left\| \Phi_{k'}(f,t) \right\|_{H^{k'}} & \leq C \left( 1 + \| f \|_{H^{k'+3}} \right), \; \text{and} \; \| \Phi_{k'}^{-1} (f,t) \|_{H^{k'}} \leq C \left( 1 + \| f \|_{H^{k'+3}} \right) \cdot
\label{Eqn:Phifkpr}
}

\item  Define $\overline{k'}$ to be the following number:

\EQQARR{
k' \geq 3: & \overline{k'} := k' +3 \\
1 \leq k' < 3: & \overline{k'} := \max \left( \frac{9}{2}+, k'+3 \right) \\
0 \leq k' < 1: & \overline{k'} := \frac{9}{2}+
}
Let $f$ and $g$ be two functions. Assume that  $ \| f \|_{H^{\overline{k'}}} \leq K $  and $ \| g \|_{H^{\overline{k'}}} \; \leq K $. Assume also that
there exists $\bar{\delta} > 0 $ such that for all $t \in I$ and for all $\theta \in  [0,1]$ $\inf \limits_{x \in \T}
  \left| \p_{\omega_3} F \left( \overrightarrow{h_{\theta}},t \right)(x)  \right| \geq \bar{\delta} $ with $h_{\theta}:= \theta f + (1- \theta) g$. Then there exists $C := C (\bar{\delta}, K) $ such that for all $t_{1}, t_{2} \in I$ the following holds:

\EQQARRLAB{
\left\| \Phi_{k'}(f,t_{2}) - \Phi_{k'}(g,t_{1}) \right\|_{H^{k'}} & \leq  C \left( \| f - g \|_{H^{\overline{k'}}} + |t_{2} - t_{1}| \right), \; \text{and} \; \\
& \\
\left\| \Phi_{k'}^{-1}(f,t_{2}) - \Phi_{k'}^{-1}(g,t_{1}) \right\|_{H^{k'}} &  \leq C \left( \| f - g \|_{H^{\overline{k'}}} + |t_{2} - t_{1}| \right) \cdot
\label{Eqn:DiffPhifkpr}
}

\end{enumerate}

\label{lem:gauge}
\end{lem}

\begin{rem}
Regarding the constants in Lemma \ref{lem:gauge}, we refer to Remark \ref{Rem:EstDerivOp}.
\label{Rem:Gauge}
\end{rem}

\begin{proof}
We only prove the first estimate of (\ref{Eqn:Phifkpr}): the proof of the second estimate is similar and therefore omitted. Let $m \geq 0$. We have

\EQQARR{
\left\| \left[ \frac{P(f,t)}{\p_{\omega_3} F(\vec{f},t)} \right]_{ave} \right\|_{L^{2}} & \lesssim
\left\| \left[ \frac{P(f,t)}{\p_{\omega_3} F(\vec{f},t)} \right]_{ave} \right\|_{L^{\infty}} \lesssim  \left\|  \frac{P(f,t)}{\p_{\omega_3} F(\vec{f},t)}  \right\|_{L^{1}}
\lesssim  \left\|  \frac{P(f,t)}{\p_{\omega_3} F(\vec{f},t)}  \right\|_{L^{2}}
}
Hence in view of the Sobolev embedding  $H^{\frac{1}{2}+} \hookrightarrow L^{\infty}$ we have

\EQQARR{
\left\| \left[ \frac{P(f,t)}{\p_{\omega_3} F(\vec{f},t)} \right]_{ave} \right\|_{L^{2}} & \lesssim \sup \limits_{\substack{\gamma \in \N^{6}: |\gamma| \leq 2
\\ | (\vec{\omega},t)| \lesssim  \langle K \rangle}} \left| \p^{\gamma} F(\vec{\omega},t) \right|
\| f \|_{H^{4}} \lesssim 1
}
Hence, using also $\p_{x}^{n} \left[ \frac{P(f,t)}{\p_{\omega_3}F (\vec{f},t)} \right]_{ave} = 0$ for  $n \in \N^{*}$, we get

\EQQARRLAB{
 \left\| \left[ \frac{P(f,t)}{\p_{\omega_3} F (\vec{f},t)} \right]_{ave} \right\|_{H^{m}} \lesssim 1.
 \label{Eqn:EstFracAvPf}
 }
Lemmas \ref{Lem:prod} and \ref{lem:EstDerivF} yield

\EQQARRLAB{
\| P(f,t) \|_{H^{m}} & \lesssim \sup  \limits_{\gamma \in \N^{6}: |\gamma| \leq 2 }
 \left\| \p^{\gamma} F (\vec{f},t)  \right\|_{H^{m}} \langle \| f \|_{H^{\frac{9}{2}+}} \rangle  \\
 & + \sup  \limits_{\gamma \in \N^{6}: |\gamma| \leq 2 }
 \left\| \p^{\gamma} F (\vec{f},t)  \right\|_{H^{\frac{1}{2}+}} \| f \|_{H^{m+4}} \\
 & \lesssim 1 + \| f \|_{H^{m+4}} \cdot
 \label{Eqn:EstPf}
}
We get from Lemma \ref{lem:EstDerivOp}

\EQQARRLAB{
\left\| \frac{P(f,t)}{\p_{\omega_3} F (\vec{f},t) } \right\|_{H^{m}} &  \lesssim 1 + \| P(f,t) \|_{H^{m}} + \| \p_{\omega_3} F (\vec{f},t) \|_{H^{m}}  \\
 & \lesssim 1 + \| f \|_{H^{m+4}} \cdot
\label{Eqn:EstFracPf}
}
Let $X :=  e^{\int_{0}^{x} \frac{1}{3} \left( \frac{P(f,t)}{\p_{\omega_3} F(\vec{f},t)} - \left[ \frac{P (f,t)}{\p_{\omega_3} F (\vec{f},t)}  \right]_{ave} \right) \; dx'}$. We have

\EQQARR{
\| \Phi_{k'} (f,t) \|_{H^{k'}} & \lesssim \left\| (\p_{\omega_3} F)^{\frac{2k'-15}{6}} (\vec{f},t) \right\|_{H^{k'}}
\| X \|_{H^{\frac{1}{2}+}} +  \| X \|_{H^{k'}} \left\|  (\p_{\omega_3} F)^{\frac{2k'-15}{6}} (\vec{f},t)  \right\|_{H^{\frac{1}{2}+}} \\
& \lesssim  \left( 1 + \| f \|_{H^{k'+3}} \right) \| X \|_{H^{\frac{1}{2}+}} + \| X \|_{H^{k'}}  \\
& \lesssim 1 + \| f \|_{H^{k'+3}} \cdot
}
Hence (\ref{Eqn:Phifkpr}) holds. \\
\\
Next we prove the first estimate of (\ref{Eqn:DiffPhifkpr}): indeed the proof of the second estimate is similar and therefore left to the
reader. We have

\EQQARR{
 \| \Phi_{k'} (f,t_{2}) - \Phi_{k'} (g,t_{1}) \|_{H^{k'}} & \lesssim \| \triangle Y(t_{1},t_{2}) \|_{H^{k'}} \| Z(\vec{f},t_{2}) \|_{H^{\frac{1}{2}+}}
+ \| \triangle Y(t_{1},t_{2}) \|_{H^{\frac{1}{2}+}} \| Z(\vec{f},t_{2}) \|_{H^{k'}}  \\
& + \| Y(\vec{g},t_{1}) \|_{H^{k'}} \| \triangle Z(t_{1},t_{2}) \|_{H^{\frac{1}{2}+}} + \| Y(\vec{g},t_{1}) \|_{H^{\frac{1}{2}+}} \| \triangle Z(t) \|_{H^{k'}},
}
with  $ Y(\vec{f},t) := (\p_{\omega_3} F)^{\frac{2k-15}{6}}(\vec{f},t) $ , $ \triangle Y(t_{1},t_{2}) := Y(\vec{f},t_{2}) - Y(\vec{g},t_{1}) $, \\
$Z(\vec{f},t)$ defined by $ Z (\vec{f},t)(x) :=   e^{\int_{0}^{x} \frac{1}{3} \left( \frac{P(f,t)}{\p_{\omega_3} F(\vec{f},t)} -
\left[ \frac{P(f,t)}{\p_{\omega_3} F (\vec{f},t) }  \right]_{ave} \right) \; d x' } $, and
$\triangle Z(t_1,t_2) := Z(\vec{f},t_2) - Z(\vec{g},t_1) $. Let $m \in \left\{ \frac{1}{2}, k' \right\}$. Lemma \ref{lem:EstDerivF}
shows that $\left\| \triangle Y(t_{1},t_{2}) \right\|_{H^{k'}} \lesssim |t_{2} - t_{1}| + \| f - g \|_{H^{\max \left( k' + 3, \frac{7}{2}+ \right) }}$
and $\left\| \triangle Y(t_{1}, t_{2}) \right\|_{H^{\frac{1}{2}+}} \lesssim |t_{2} - t_{1}| +  \| f - g \|_{H^{\frac{7}{2}+}}$.
We also have $\left\| Y(\vec{g},t_{1}) \right\|_{H^{k'}} \lesssim 1 + \| g \|_{H^{k'+ 3}} $ and
$\left\| Y(\vec{g},t_{1}) \right\|_{H^{\frac{1}{2}+}} \lesssim 1 + \| g \|_{H^{\frac{7}{2}+}} $. We have
$\left\|  Z(\vec{f},t_{2})\right\|_{H^{\frac{1}{2}+}} \lesssim  1$ and
$\left\|  Z(\vec{f},t_{2}) \right\|_{H^{k'}} \lesssim  1$ ( resp.  \; $\left\|  Z(\vec{f},t_{2}) \right\|_{H^{k'}} \lesssim
1 + \| f \|_{H^{k'+3}}$ )
if $k' \leq 1 $ (resp. $k' \geq 1$). It remains to estimate $ \| \triangle  Z(t_{1},t_{2}) \|_{H^{m}}$ for $m \in \left\{ \frac{1}{2}, k' \right\}$.
To this end we prove the following claim: \\
\\
\underline{Claim}: Let $p \geq 0$. Assume that $\| f \|_{H^{ \max \left( \frac{9}{2}+, p+4 \right) }} \leq K$ and
$\| g \|_{H^{\max \left( \frac{9}{2}+, p +3 \right)}} \leq K$. Then

\EQQARRLAB{
\| P(f,t_{2}) - P(g,t_{1}) \|_{H^{p}} & \lesssim \| f - g \|_{H^{\max{ \left( p + 4, \frac{9}{2}+ \right)}}} + |t_{2} - t_{1}|
\label{Eqn:DiffEstP}
}

\begin{proof}

We have $ \| P(f,t_{2}) - P(g, t_{1}) \|_{H^{p}} \leq A + B  $ with

\EQQARR{
 A := \sup  \limits_{\gamma \in \N^{6}: |\gamma| \leq 2 } \left\| \p^{\gamma}  F ( \vec{f},t_{2} ) -  \p^{\gamma} F( \vec{g},t_{1} )  \right\|_{H^{p}}
\sup_{ 0 \leq j \leq 4} \| \p_{x}^{j} f \|_{H^{\frac{1}{2}+}} \\
+  \sup \limits_{\gamma \in \N^{6}: |\gamma| \leq 2 } \left\| \p^{\gamma}  F ( \vec{f},t_{2} ) -  \p^{\gamma} F( \vec{g},t_{1} )  \right\|_{H^{p}}
\sup_{ 0 \leq j \leq 4} \| \p_{x}^{j} f \|_{H^{p}}, ; \text{and}
}

\EQQARR{
B := \sup \limits_{ 0 \leq j \leq 4 } \| \p_{x}^{j} (f-g)  \|_{H^{p}} \sup  \limits_{\gamma \in \N^{6}: |\gamma| \leq 2 }
 \left\| \p^{\gamma}  F ( \vec{g},t_{1}) \right\|_{H^{\frac{1}{2}+}} \\
 + \sup \limits_{ 0 \leq j \leq 4 } \| \p_{x}^{j} (f-g)  \|_{H^{\frac{1}{2}+}}
 \sup  \limits_{\gamma \in \N^{6}: |\gamma| \leq 2 } \left\| \p^{\gamma}  F ( \vec{g},t_{1}) \right\|_{H^{p}} \cdot
}
Lemma \ref{lem:EstDerivF} implies that

\EQQARR{
A \lesssim \| f  - g \|_{H^{\max \left( p + 3, \frac{7}{2}+ \right)}}  \| f \|_{H^{\max \left( \frac{9}{2}+, p \right)} }  + |t_{2} - t_{1}| \\
\lesssim  \| f  - g \|_{H^{\max \left( p + 3, \frac{7}{2}+ \right)}} + | t_{2} - t_{1} | \cdot
}
We also have

\EQQARR{
B \lesssim \| f - g \|_{H^{p+4}}  \left( 1 +  \| g \|_{H^{\frac{7}{2}+}} \right)
+ \| f - g \|_{H^{\frac{9}{2}+}} \left( 1 + \| g \|_{H^{p+3}} \right) \\
\lesssim \| f - g \|_{H^{\max \left( p + 4, \frac{9}{2}+ \right)}} \cdot
}
Hence (\ref{Eqn:DiffEstP}) holds.

\end{proof}
Let $m \geq 1$. Lemma \ref{lem:EstDerivOp} shows that

\EQQARR{
\| \triangle Z(t_{1},t_{2}) \|_{H^{m}} & \lesssim \left\| \frac{P(f,t_{2})}{\p_{\omega_3} F(\vec{f},t_{2})} - \frac{P(g,t_{1})}{\p_{\omega_3} F(\vec{g},t_{1})} \right\|_{H^{m-1}} +
\left\| \left[ \frac{P(f,t_{2})}{\p_{\omega_3} F(\vec{f},t_{2})} - \frac{P(g,t_{1})}{\p_{\omega_3} F(\vec{g},t_{1})}  \right]_{ave} \right\|_{H^{m-1}} \\
& \lesssim \| \triangle Z_{1}(t_{1},t_{2}) \|_{H^{m-1}} + \| \triangle Z_{2}(t_{1},t_{2}) \|_{H^{m-1}},
}
assuming that $\| \triangle Z_1(t_{1},t_{2}) \|_{H^{ \max \left( m-1, \frac{1}{2}+ \right)}} \leq K$ and
$\| \triangle Z_2(t_{1},t_{2}) \|_{H^{ \max \left( m-1, \frac{1}{2}+ \right)}} \leq K $: this will be proved shortly. Observe from
(\ref{Eqn:EstFracPf}) and (\ref{Eqn:EstFracAvPf}) that $\| \triangle Z_1(t_1,t_2) \|_{H^{\frac{1}{2}+}} \lesssim K$ and that
$\| \triangle Z_2(t_1,t_2) \|_{H^{\frac{1}{2}+}} \lesssim K$.
We first estimate $ \| \triangle Z_{2}(t_1,t_2) \|_{H^{m-1}} $. Let $ A :=  \frac{P(f,t_2) - P(g,t_1)}{\p_{\omega_3} F (\vec{f},t_2)}$ and
$B := \frac{P(g,t_1) \left( \p_{\omega_3} F(\vec{g},t_1) - \p_{\omega_3} F(\vec{f},t_2) \right) }{\p_{\omega_3} F(\vec{f},t_2) \p_{\omega_3} F(\vec{g},t_1)} $. We have
(see proof of (\ref{Eqn:EstFracAvPf}))

\EQQARRLAB{
\left\| \triangle Z (t_1,t_2) \right\|_{H^{m-1}} & \lesssim \left\| \frac{P(f,t_2)}{\p_{\omega_3} F (\vec{f},t_2)} -  \frac{P(g,t_1)}{\p_{\omega_3} F (\vec{g},t_1)}    \right\|_{L^{2}}
\lesssim \| A \|_{L^{2}} + \| B \|_{L^{2}} \cdot
\label{Eqn:EstZ2}
}
The above claim implies

\EQQARRLAB{
\| A \|_{L^{2}} & \lesssim \| P(f,t_2) - P(g,t_1) \|_{L^{2}} \lesssim \| f - g \|_{H^{\frac{9}{2}+}} + |t_2- t_1| \cdot
\label{Eqn:EstAL2}
}
We see from  (\ref{Eqn:EstPf}) and Lemma \ref{lem:EstDerivF}

\EQQARRLAB{
\| B \|_{L^{2}} & \lesssim \| P(g,t_{1}) \|_{L^{\infty}}   \| \p_{\omega_3} F(\vec{g},t_{1}) - \p_{\omega_3} F(\vec{f},t_{2}) \|_{L^{2}} \\
& \lesssim  \|  P(g,t_{1}) \|_{H^{\frac{1}{2}+}}  \left( \| f - g \|_{H^{\frac{7}{2}+}} + |t_{2} - t_{1}| \right) \\
& \lesssim \| f - g \|_{H^{\frac{7}{2}+}} + |t_{2} - t_{1}| \cdot
\label{Eqn:EstBL2}
}
Hence $\| \triangle Z_2(t_{1},t_{2}) \|_{H^{m-1}} \lesssim \| f - g \|_{H^{\overline{k'}}}$. We now estimate
$\| \triangle Z_1 (t_{1},t_{2}) \|_{H^{m-1}}$. We have
$\| \triangle Z_1(t_{1},t_{2}) \|_{H^{m-1}} \lesssim \| A \|_{H^{m-1}} + \| B \|_{H^{m-1}} $. Assume that $m \geq 3$.
We use Lemma \ref{Lem:prod}, Lemma \ref{lem:EstDerivOp}, Lemma \ref{lem:EstDerivF}, (\ref{Eqn:EstPf}), and (\ref{Eqn:DiffEstP}).
We have $ \| A \|_{H^{m-1}}  \lesssim \| P(f,t_{2}) - P(g,t_{1}) \|_{H^{m-1}} \lesssim \| f - g \|_{H^{\overline{k'}}} + |t_{2} - t_{1}| $. We have

\EQQARR{
\| B \|_{H^{m-1}} & \lesssim  \left\| P(g,t_{1}) \left( \p_{\omega_3} F(\vec{g},t_{1}) - \p_{\omega_3} F(\vec{f},t_{2}) \right)  \right\|_{H^{m-1}} \\
& \lesssim \| P(g,t_{1}) \|_{H^{m-1}}  \left\| \p_{\omega_3} F(\vec{g},t_{1}) - \p_{\omega_3} F(\vec{f},t_{2}) \right\|_{H^{\frac{1}{2}+}} \\
& + \| P(g,t_{1}) \|_{H^{\frac{1}{2}+}}  \left\| \p_{\omega_3} F(\vec{g},t_{1}) - \p_{\omega_3} F(\vec{f},t_{2}) \right\|_{H^{m-1}} \\
& \lesssim \| f - g \|_{H^{\overline{k'}}} \cdot
}
Hence $ \| \triangle Z_{1}(t_{1},t_{2}) \|_{H^{m-1}} \lesssim \| f- g \|_{H^{\overline{k'}}} $. Assume now that $ 1 \leq m < 3$. We have

\EQQARRLAB{
\| A \|_{H^{m-1}} & \lesssim \| P(f,t_{2}) - P(g,t_{1}) \|_{H^{\max \left( m -1 , \frac{1}{2}+ \right)}} \lesssim \| f - g \|_{H^{\overline{k'}}}
+ |t_{2} - t_{1}| \cdot
\label{Eqn:EstA}
}
We also have

\EQQARRLAB{
\| B \|_{H^{m-1}} & \lesssim \left\| P(g,t_{1}) \left( \p_{\omega_3} F(\vec{g},t_{1}) - \p_{\omega_3} F(\vec{f},t_{2}) \right) \right\|_{H^{\max \left( m-1, \frac{1}{2}+ \right) }} \\
& \lesssim \| P(g,t_{1}) \|_{H^{\max \left( m-1, \frac{1}{2}+ \right) }} \left\| \p_{\omega_3} F(\vec{g},t_{1}) - \p_{\omega_3} F(\vec{f},t_{2}) \right\|_{H^{\frac{1}{2}+}}  \\
& + \| P(g,t_{1}) \|_{H^{\frac{1}{2}+}} \left\| \p_{\omega_3} F(\vec{g},t_{1}) - \p_{\omega_3} F(\vec{f},t_{2}) \right\|_{H^{\max \left( m-1, \frac{1}{2}+ \right) }}  \\
& \lesssim \| f - g \|_{H^{\overline{k'}}} + |t_{2} - t_{1}| \cdot
\label{Eqn:EstB}
}
Hence $\| \triangle Z_1(t_{1},t_{2}) \|_{H^{m}} \lesssim  \| f - g \|_{H^{\overline{k'}}} + |t_{2} - t_{1} |$.\\
Assume now that $0 \leq m \leq 1 $. Lemma \ref{lem:EstDerivOp} shows that

\EQQARR{
\| \triangle Z(t_{1},t_{2}) \|_{H^{m}} & \lesssim \| \triangle Z_1(t_{1},t_{2}) \|_{L^{2}} +  \left\| \triangle Z_2(t_{1},t_{2}) \right\|_{L^{2}},
}
assuming that $ \| \triangle Z_1 (t_{1},t_{2})\|_{H^{\frac{1}{2}+}} \lesssim K$ and $\| \triangle Z_2(t_{1},t_{2}) \|_{H^{\frac{1}{2}+}} \lesssim K$: this will be proved shortly.
From (\ref{Eqn:EstZ2}), (\ref{Eqn:EstAL2}), and (\ref{Eqn:EstBL2}), we see that
$ \| \triangle Z_2(t_{1},t_{2}) \|_{H^{\frac{1}{2}+}} \lesssim \| f - g \|_{H^{\overline{k'}}} $. A similar scheme used in (\ref{Eqn:EstA}) and
(\ref{Eqn:EstB}) shows that $ \| \triangle Z_1(t_{1},t_{2}) \|_{H^{\frac{1}{2}+}} \lesssim \| f - g \|_{H^{\overline{k'}}} $. Hence the assumptions
above are clearly satisfied and moreover

\EQQARR{
\| \triangle Z(t_{1},t_{2}) \|_{H^{m}} & \lesssim \| f - g \|_{H^{\overline{k'}}} + |t_{2} - t_{1}| \cdot
}

\end{proof}

\section{Regularized problem: local well-posedness}
\label{Sec:RegulLocWell}

We consider the following regularized problem (with $\epsilon \in (0,1]$)

\EQQARRLAB{
\p_{t} u + \epsilon \p_{x}^{4} u & = F(\vec{u},t) \cdot
\label{Eqn:Regul}
}
Let $A \in \mathbb{R}$. Let  $e^{- A \p_x^{4} f} $  be the operator such that
$ \widehat{e^{-A \p_x^{4}} f}(n) := e^{- A n^{4}} \hat{f}(n)$ for $n \in \mathbb{Z}$. The following proposition holds:

\begin{prop}
Let $k' \geq k_{0}^{'} >  \frac{9}{2}$. Let $\tilde{\phi} \in H^{k'}$. Then there exists $T_{\epsilon} := T_{\epsilon} ( \tilde{\phi} )  \in (0, \infty]$ and a unique solution
$u \in \mathcal{C} \left( \left[ 0,T_{\epsilon} \right) , H^{k'} \right)$ such that $u$ satisfies the integral equation for (\ref{Eqn:Regul}) with initial data
$u(0):= \tilde{\phi}$ on $[0, T_{\epsilon})$, i.e

\EQQARRLAB{
t \in [ 0, T_{\epsilon} ) : & u(t) = \mathcal{T}(u(t)) :=  e^{- \epsilon t \p_x^{4}} \tilde{\phi} + \int_{0}^{t} e^{- \epsilon (t-t') \p_{x}^{4}} F \left( \overrightarrow{u(t')} ,t' \right) \; dt',
\label{Eqn:Integral}
}
and such that either $(i): \;  \lim \inf_{t \rightarrow T_{\epsilon}} \| u(t) \|_{H^{k'_{0}}}  = \infty $ or
$ (ii): \; T_{\epsilon}  = \infty $ holds. Moreover let $\tilde{\phi}_{n} \in H^{k'}$ and $\tilde{\phi}_{\infty} \in H^{k'}$ be such that  $ \| \tilde{\phi}_{n} - \tilde{\phi}_{\infty} \|_{H^{k'}} \rightarrow 0 $ as $ n \rightarrow \infty $. Let $u_{n} \in \mathcal{C} \left( \left[ 0, \, T_{\epsilon} ( \tilde{\phi}_{n})  \right), H^{k'} \right)$
$\left( \text{resp.} \; u_{\infty} \; \in \mathcal{C} \left( \left[ 0, \, T_{\epsilon}( \tilde{\phi}_{\infty}  ) \right), H^{k'} \right) \right)$ be the solution of (\ref{Eqn:Integral}), replacing $\tilde{\phi}$ with $\tilde{\phi}_{n}$ (resp. $\tilde{\phi}_{\infty}$). Let
$ T \in \left[ 0, \, T_{\epsilon} ( \tilde{\phi}_{\infty} ) \right) $. Then $ \sup \limits_{t \in [0,T] } \| u_n(t) - u_{\infty}(t) \|_{L_{t}^{\infty} H^{k'} ([0,T])} \rightarrow 0$ as $n \rightarrow \infty$.
\label{Prop:Integral}
\end{prop}


\begin{rem}
The conclusion above implies that  $T_{\epsilon} \left( \tilde{\phi_{n}} \right) > T $ for $n$ large.
\end{rem}

\begin{proof}

The proof uses standard arguments. In the sequel let $C (x)$ be a nonnegative number depending on $x$ of which the value is allowed to change from one line to the other one and such that all the statements below are true. \\
\\
Let $T'_{\epsilon} := T'_{\epsilon} \left( \|  \tilde{\phi} \|_{H^{k'}} \right)$ be a positive time that is small enough such that all the statements below are true.
Let $X$ be the following Banach space

\EQQARR{
X := \left\{ v \in \mathcal{C} \left( [ 0,T'_{\epsilon}), H^{k'} \right):
\; \| v \|_{ L_{t}^{\infty}  H^{k'} \left( [ 0,T'_{\epsilon}) \right) } \leq 2 ( 1 + \| \tilde{\phi} \|_{H^{k'}} ) \right\} \cdot
}
Clearly $X$, endowed with the norm $ \| v \|_{X} := \| v \|_{L_{t}^{\infty} H^{k'} ([0,T'_{\epsilon})) } $, is a Banach space.  Let $t \in [0,T'_{\epsilon})$.
Let $ A_{v}(t) := F \left( \overrightarrow{v(t)},t \right) $. \\
We first claim that $ \mathcal{T}(X) \subset X$. The Minkowski inequality shows that

\EQQARR{
\| \mathcal{T}(v(t)) \|_{H^{k'}} & \leq \| \tilde{\phi} \|_{H^{k^{'}}} +
\int_{0}^{t} \left\| \langle n \rangle^{3} e^{- \epsilon (t-t') n^{4}} \langle n \rangle^{ k^{'} -3} \widehat{A_{v}(t')}(n) \right\|_{l^{2}} \; d t' \\
}
From (\ref{Eqn:EstFDeriv}) we see that

\EQQARRLAB{
\left\| \langle n \rangle^{ k' - 3} \widehat{A_{v}(t')}(n)  \right\|_{l^{2}} & \leq C \left( \| \tilde{\phi} \|_{H^{k_{0}^{'}}} \right)
\left( 1 + \| v(t') \|_{H^{k'}} \right)
\label{Eqn:Nonlin92}
}
Elementary considerations show that

\EQQARRLAB{
\sup_{n \in \mathbb{Z}} \langle n \rangle^{3} e^{- \epsilon (t-t') n^{4}} \lesssim
\epsilon^{-\frac{3}{4}} (t-t')^{- \frac{3}{4}} \cdot
\label{Eqn:EstElemExp}
}
Hence we get after integration

\EQQARRLAB{
\| \mathcal{T}(v) \|_{X} & \leq \| \tilde{\phi} \|_{H^{k'}} + C ( \|  \tilde{\phi} \|_{H^{k^{'}_{0}}} )
 \epsilon^{-\frac{3}{4}} \left( T^{'}_{\epsilon} \right)^{\frac{1}{4}} ( 1 + \| \tilde{\phi} \|_{H^{k'}} ) \\
& \leq \| \tilde{\phi} \|_{H^{k'}} +  C ( \|  \tilde{\phi} \|_{H^{k^{'}}} )
 \epsilon^{-\frac{3}{4}}  \left( T^{'}_{\epsilon} \right)^{\frac{1}{4}} ( 1 + \| \tilde{\phi} \|_{H^{k'}} ) \\
& \leq 2 ( 1 + \| \tilde{\phi} \|_{H^{k'}} ) \cdot
\label{Eqn:EstfixedOne}
}
One gets from (\ref{Eqn:DiffEstFDeriv})

\EQQARRLAB{
\left\| \langle n \rangle^{ k'-3 }
\left(  \widehat{A_{v_{1}}(t')}(n) - \widehat{A_{v_{2}}(t')}(n) \right)
\right\|_{l^{2}}
& \leq C \left( \| \tilde{\phi} \|_{H^{k'}} \right) \| v_{1}(t') -  v_{2}(t') \|_{H^{k'}}
\label{Eqn:EstDiffv1v2}
}
Hence integrating in time we see that

\EQQARR{
\left\| \mathcal{T}(v_1) - \mathcal{T}(v_2) \right\|_{X} & \lesssim C \left( \| \tilde{\phi} \|_{H^{k'}} \right) \epsilon^{- \frac{3}{4}} \left(T^{'}_{\epsilon} \right)^{\frac{1}{4}}
\| v_{1} -  v_{2} \|_{ L_{t}^{\infty} H^{k'} ([0,T^{'}_{\epsilon}])} \\
& \leq \frac{1}{2} \| v_1 - v_2 \|_{X}
}
Hence $\mathcal{T}$ is a contraction.  \\
\\
We then claim that $\mathcal{T} v \in \mathcal{C} \left( [ 0,T'_{\epsilon}), H^{k'} \right) $. Indeed the dominated convergence
theorem allows to prove that $ t \rightarrow e^{- \epsilon t \p_{x}^{4} } \tilde{\phi}$ in continuous
in $ H^{k'}$. Let $t \in [0, T^{'}_{\epsilon})$ and let $h > 0 $ be a small number. Let $ G(t) :=  \int_{0}^{t} e^{- \epsilon (t-t') \p_{x}^{4}} A_{v}(t') \; dt' $. Writing
$ G (t + h) - G(t) = H_{1} + H_{2} $ with \\
$H_{1} := \int_{t}^{t+h}  e^{- \epsilon (t + h -t') \p_{x}^{4}} A_{v}(t') \; dt'  $,
and  $ H_{2} := \int_{0}^{t} \left( e^{- \epsilon (t + h - t') \p_{x}^{4}} - e^{- \epsilon (t- t') \p_{x}^{4} } \right) A_{v}(t') \; dt'  $. (\ref{Eqn:EstElemExp}), (\ref{Eqn:EstfixedOne}), and the dominated convergence theorem
allow to prove that $H_{1} \rightarrow 0 $ as $h \rightarrow 0$. In order to prove that
$ H_{2} \rightarrow 0 $ we proceed as follows. The mean value theorem and elementary estimates show that
$ \left| e^{- \epsilon (t + h - t')n^{4}} - e^{- \epsilon (t- t') n^{4}} \right|  \lesssim e^{- \epsilon (t - t') n^{4}} n^{4} h \epsilon $. We also have
$ \left| e^{- \epsilon (t + h - t')n^{4}} - e^{- \epsilon (t- t') n^{4}} \right|  \lesssim e^{- \epsilon (t - t') n^{4}} $ . Hence
$\| H_{2} \|_{H^{k'}} \lesssim X + Y $ with
$ X := \int_{0}^{t}  \left\| \langle n \rangle^{3}  e^{- \epsilon (t-t') n^{4}} \langle n \rangle^{k'-3} \widehat{A_{v}(t')}(n) \right\|_{l^{2}(n^{4} h \geq 1)}  \; dt'  $ and
$ Y :=  \int_{0}^{t} \left\| \langle n \rangle^{3} n^{4} h \epsilon e^{- \epsilon (t-t') n^{4}} \langle n \rangle^{k'-3}  \widehat{A_{v}(t')}(n)  \right\|_{l^{2}(n^{4} h \leq 1)} \; dt' $. There exists $c > 0 $ such that

\EQQARR{
X & \lesssim  \int_{0}^{t} (t-t')^{-\frac{3}{4}} \sup \limits_{n' \geq \frac{\epsilon (t-t')}{h}} \left( (n^{'})^{\frac{3}{4}} e^{-n'} \right) \left\|
\langle n \rangle^{k'-3}  \widehat{A_{v}(t')}(n) \right\|_{l^{2}} \; dt' \\
& \lesssim ( 1 + \| \tilde{\phi} \|_{H^{k'}}) \int_{0}^{t} (t-t')^{-\frac{3}{4}} e^{- \frac{ c \epsilon (t-t')}{h}}  \; dt'.
}
Hence, by making the change of variable $ u= \frac{(t-t') c \epsilon}{h}, $ $A \rightarrow 0$ as $ h \rightarrow 0$. We have

\EQQARR{
Y & \lesssim  h^{0+}   \int_{0}^{t} \left\| \sup_{n^{4} h \leq 1 } \left( n^{3+}  e^{- \epsilon (t-t') n^{4}} \right) \langle n \rangle^{k'-3} \widehat{A_{v}(t')}(n) \right\|_{l^{2}(n^{4} h \leq 1)} \; dt' \\
& \lesssim  h^{0+} \int_{0}^{t}  (t-t')^{\left (-\frac{3}{4} - \right) } \; dt'  ( 1 + \| \tilde{\phi} \|_{H^{k'}}) \\
& \lesssim h^{0+} ( 1 + \| \tilde{\phi} \|_{H^{k'}})  \cdot
}
Hence $B \rightarrow 0$ as $h \rightarrow 0$. \\
Observe that the argument shows that one can choose $T^{'}_{\epsilon}$ as a continuous and decreasing function of $\| \tilde{\phi} \|_{H^{k'}}$. \\
\\
We consider now the union of all intervals $[0,T')$ such that there exists a (unique) solution
$u \in \mathcal{C} \left( [0,T'), H^{k'} \right)$ satisfying the integral equation for (\ref{Eqn:Regul}) on $[0,T')$. Clearly there exists
$T_{\epsilon} := T_{\epsilon} ( \tilde{\phi} ) \in (0, \infty]$ such that this union is equal to $[0,T_{\epsilon})$. Assume now that $T_{\epsilon} < \infty$ and that
$ \lim \inf_{t \rightarrow T_{\epsilon}} \| u(t) \|_{H^{k'_{0}}}  < \infty $. Hence we can find an $M  < \infty $  and
$ 0 \leq \tilde{t} < T_{\epsilon} $ close enough to $T_{\epsilon}$ such that the following properties holds:
$ \| u(\tilde{t}) \|_{H^{k_{0}^{'}}} \leq M $, and (by proceeding similarly as in (\ref{Eqn:EstfixedOne})) we get for all
$ t \in [\tilde{t}, T_{\epsilon}) $

\EQQARRLAB{
\| u \|_{L_{t}^{\infty} H^{k'_{0}} ([\tilde{t}, t ]) }  \leq M
+ C (M) \epsilon^{-\frac{3}{4}}  (t-\tilde{t})^{\frac{1}{4}} \left( \| u \|_{L_{t}^{\infty} H^{k'_{0}} ([\tilde{t},t])} + 1 \right)
\label{Eqn:EstfixedOnePlus}
}
Hence a continuity argument shows that $\| u \|_{L_{t}^{\infty} H^{k'_{0}} ([\tilde{t}, t ])} \lesssim  M +1 $. Since we also have
$ u \in \mathcal{C} \left( [0,\tilde{t}], H^{k'_{0}} \right)$ this implies that there exists a constant  (that we still denote by $M > 0$) such that $ \left\| u(t) \right\|_{H^{k_{0}^{'}}} \leq M $ for all $t \in [0,T_{\epsilon})$.
We then claim that there exists $M' > 0 $ such that

\EQQARRLAB{
t \in [0, T_{\epsilon}): \;  \left\| u(t) \right\|_{H^{k'}} \leq M'  \cdot
\label{Eqn:BoundMpri}
}
Indeed let  $ J := [a,t) \subset [0,T_{\epsilon})$; then we get

\EQQARR{
\| u \|_{L_{t}^{\infty} H^{k'}([a,t))} & \leq \| u(a) \|_{H^{k'}} + C (M) \epsilon^{-\frac{3}{4}}  (t-a)^{\frac{1}{4}} \left( 1  + \| u \|_{L_{t}^{\infty} H^{k'}([a,t))} \right);
}
hence a continuity argument shows that there exists $c := c(M, \epsilon) > 0$ such that if $ |J| \leq c $ then $ \| u \|_{L_{t}^{\infty} H^{k'} (J)} \leq
2 \langle \| u(a) \|_{H^{k'}} \rangle $; hence, dividing $[0,T_{\epsilon})$ into subintervals $J$ such that $ |J| = c $ except maybe the last one we see that (\ref{Eqn:BoundMpri}) holds. Let $ \bar{T} $ be close enough to $T_{\epsilon}$ from the left. Then  $ \| u(\bar{T}) \|_{H^{k'}} \leq M' $; moreover by using the arguments above to construct a unique solution  $ u \in \mathcal{C} \left( [0,T'_{\epsilon}], H^{k'} \right) $ satisfying (\ref{Eqn:Integral}) we can construct a unique solution $u \in
\mathcal{C} \left( [\bar{T}_{\epsilon}, \overline{\bar{T}}), H^{k'} \right) $ for some $ \overline{\bar{T}} > T_{\epsilon} $ that satisfies (\ref{Eqn:Integral}), replacing
``$\tilde{\phi}$'', ``$\int_{0}^{t}$'' with ``$u(\bar{T})$'', `` $ \int_{\bar{T}}^{\overline{\bar{T}}} $'' respectively. Hence there exists a unique solution
$u \in \mathcal{C} \left( [0,\overline{\bar{T}} ), H^{k'}  \right)$ of (\ref{Eqn:Integral}), which contradicts the definition of $T_{\epsilon}$. \\
\\
Let $\{ \tilde{\phi}_{n} \}$ be a sequence such that  $\| \tilde{\phi}_{n} - \tilde{\phi}_{\infty} \|_{H^{k'}} \rightarrow 0 $ as
$ n \rightarrow \infty $. Let $t \in \left[ 0, \frac{T_{\infty}^{'}}{2} \right) $, with
$T^{'}_{\infty} := T_{\epsilon}^{'} \left( \| \tilde{\phi}_{\infty} \|_{H^{k^{'}}} \right)$. The above observation shows that
$T^{'}_{\epsilon} \left( \| \tilde{\phi}_{n} \|_{H^{k^{'}}} \right) > \frac{T^{'}_{\infty}}{2} $ for $n$ large.
We have

\EQQARR{
 \left\| e^{-\epsilon t \p_{x}^{4}} ( \tilde{\phi}_{n} - \tilde{\phi}_{\infty} ) \right\|_{H^{k'}} & \lesssim \| \tilde{\phi}_{n} - \tilde{\phi}_{\infty} \|_{H^{k'}} \cdot
}
We see from (\ref{Eqn:EstDiffv1v2}) ( with $v_{1} := u_{n}$ and $ v_{2} := u_{\infty} $ )  and the estimate $ \| \tilde{\phi}_{n} \|_{H^{k'}} \lesssim \| \tilde{\phi}_{\infty} \|_{H^{k'}} $
we see that

\EQQARR{
\left\| \langle n \rangle^{ k'-3}
\left(  \widehat{A_{u_{n}}(t')}(n) - \widehat{A_{u_{\infty}}(t')}(n) \right)  \right\|_{l^{2}}
& \leq C \left( \| \tilde{\phi} \|_{H^{k'}} \right) \| u_{n}(t') -  u_{\infty}(t') \|_{H^{k'}}
}
Integrating in time we get $ \| u_{n} - u \|_{L_{t}^{\infty} H^{k'}  \left( \left[ 0 ,\frac{T^{'}_{\infty}}{2} \right) \right)  } \leq  \| \tilde{\phi}_{n} - \tilde{\phi} \|_{H^{k'}} + \frac{1}{2} \| u_{n} - u \|_{L_{t}^{\infty} H^{k'} \left( \left[ 0 ,\frac{T^{'}_{\infty}}{2} \right) \right)} $,
which implies that $\| u_{n} - u_{\infty} \|_{L_{t}^{\infty} H^{k'}  \left( \left[ 0 ,\frac{T^{'}_{\infty}}{2} \right) \right)} \rightarrow 0 $ as $ n \rightarrow \infty $. \\
\\
Let $ 0 \leq T <  T_{\epsilon}( \tilde{\phi}_{\infty}  )$. Observe that the continuity of $u_{\infty}$ on $[0,T]$ implies that there exists $ M  > 0 $ such that

\EQQARRLAB{
t \in [0,T]: \; \| u_{\infty}(t) \|_{H^{k'}} \leq M \cdot
\label{Eqn:BoundCompu}
}
Let $ T_{*} := \sup  \left\{ t \in [0,T]: \;  \sup \limits_{s \in [0,t]} \| u_{n}(s) - u_{\infty}(s) \|_{H^{k'}} \rightarrow 0 \; \text{as} \;
n \rightarrow \infty \right\} $. Since $ \| u_{n} - u_{\infty} \|_{L_{t}^{\infty} H^{k'} \left( \left[0, \frac{T^{'}_{\infty}}{2} \right] \right)} \rightarrow 0 $  as $ n \rightarrow \infty $, we have $T_{*} > 0$. Assume that $T_{*} < T $. Then choose  $ 0 \leq \tilde{t} < T_{*} $  that is close enough to $ T_{*} $ such that the
statements below hold. We have  $ \sup \limits_{t \in [0,\tilde{t}]} \| u_{n}(t) - u_{\infty}(t) \|_{H^{k'}} \rightarrow 0 $ as $ n \rightarrow \infty $, and
we also have $ \| u_{n} - u_{\infty} \|_{L_{t}^{\infty} H^{k'}  ([\tilde{t}, \tilde{t}^{'}])} = \sup \limits_{0 \leq h \leq \tilde{t}^{'} - \tilde{t} } \| u_{n}(\tilde{t} + h )  - u_{\infty} (\tilde{t} + h) \|_{H^{k'}}  \rightarrow 0 $ for some $ T \geq \tilde{t}^{'} >  T_{*} $  as $ n \rightarrow \infty $, by applying the above argument
to $ u_{n}(\tilde{t})$  ( resp. $u_{\infty}(\tilde{t})$) for $ t \in [0, \tilde{t}^{'} - \tilde{t}]$ ) instead of $ \tilde{\phi}_{n}$  (resp. $\tilde{\phi}_{\infty}$) for
$ t \in [0, T^{'}_{\epsilon} ] $ ), taking into account (\ref{Eqn:BoundCompu}). But then this would contradict the definition of $T_{*}$. Hence $ T_{*} = T $. We now claim that
$ \sup \limits_{t \in [0,T_{*}]} \| u_{n}(t) - u_{\infty}(t) \|_{H^{k'}} \rightarrow 0 $ as $ n \rightarrow \infty $. Indeed we can choose
$ 0 \leq \tilde{t} < T_{*} $  that is close enough to $ T_{*} $  such that $ \sup \limits_{t \in [0,\tilde{t}]} \| u_{n}(t) - u_{\infty}(t) \|_{H^{k'}} \rightarrow 0 $ as
$ n \rightarrow \infty $ and  $ \| u_{n} - u_{\infty} \|_{L_{t}^{\infty} H^{k'}  ([\tilde{t}, \tilde{t}^{'}])} \rightarrow 0 $ for some
$   T_{\epsilon}( \tilde{\phi}_{\infty} )  > \tilde{t}^{'} >  T_{*} $  as $ n \rightarrow \infty$. Hence $ \sup \limits_{t \in [0,T] } \| u_{n}(t) - u_{\infty}(t) \|_{H^{k'}} \rightarrow 0 $ as $ n \rightarrow \infty $.

\begin{rem}
We say that $u$ is a smooth solution of (\ref{Eqn:Integral}) obtained by Proposition \ref{Prop:Integral} with data
$ \tilde{\phi} $ if $\tilde{\phi}$ is smooth, i.e $ \tilde{\phi}  \in \bigcap \limits_{m \geq k'_{0}} H^{m} = \mathcal{C}^{\infty} $ \footnote{This equality follows for the well-known Sobolev embedding $H^{m} \hookrightarrow \mathcal{C}^{p}$ if $ m > p + \frac{1}{2}$}. Observe that  $ u \in  \mathcal{C}^{\infty} \left( [0,T], H^{m} \right) $ for all
$ m \geq k^{'}_{0} $ and that $u(t) \in \mathcal{C}^{\infty}$: this follows from a standard bootstrap argument
applied to $\p_{t} u = -\epsilon \p_{x}^{4} u + F(\vec{u},t)$, taking into
account Lemma \ref{lem:EstDerivF}. Let $\tilde{\phi} \in H^{k'}$. Let  $\{ \tilde{\phi}_{n} \} $ be a sequence of smooth
functions such that $\tilde{\phi}_{n} \rightarrow \tilde{\phi} $. Then Proposition \ref{Prop:Integral} allows to approximate the solution $u$
arbitrarily on any compact interval $[0,T] \subset [0,T_{\epsilon}) $ with the sequence of smooth solutions $ \{ u_{n} \}$. In the sequel we will
implicitly use $a/$: smooth solutions to justify some computations that allow to prove some estimates that involve these solutions
$b/$: the approximation of $u$ with smooth solutions to show that these estimates also hold for $u$.

\label{Rem:smooth}
\end{rem}

\end{proof}

\section{Gauged energy}
\label{Sec:EnergyDef}

Let $k^{'}  \geq 6$. Let $f$ and $g$ be two functions.
We define the gauged energy $E_{k'}(f,t)$ of $f$  (and more generally $E_{k'}(f,g,t)$) at a time $t$ in the following fashion:

\EQQARR{
E_{k'}(f,g,t) & := \left\| \Phi_{k'} (f,t) \p_{x}^{6} f  - \Phi_{k'} (g,t) \p_{x}^{6} g \right\|^{2}_{H^{k'-6}}  +
\| f - g \|^{2}_{H^{\max \left( k'-3,\frac{9}{2}+ \right)}}, \; \text{and}  \\
E_{k'}(f,t) & :=  E_{k'}(f,0,t) \cdot
}
The following proposition shows that the gauged energy $E_{k'}(f,t)$ of $f$ (resp. $E_{k'}(f,g,t)$) can be compared to the $H^{k'}-$ norm of $f$
(resp.  $\| f - g \|_{H^{k'}}$):

\begin{prop}

Let $I \subset \mathbb{R}$ be an interval such that $|I| \leq 1$. Let $\bar{\delta} > 0 $. Let $K \in \mathbb{R}^{+}$.  Let $k' > \frac{13}{2}$. The following holds:

\begin{enumerate}

\item  Assume that $\inf \limits_{x \in \T}  \left| \p_{\omega_3} F \left( \vec{f},t \right)(x)  \right| \geq \bar{\delta}$
for all $t \in I$. \\

\begin{itemize}

\item Assume also that $ \| f \|_{H^{\frac{13}{2}+}} \leq K $. Then there exists $C := C(K,\bar{\delta}) > 0$ such that for all $t \in I$

\EQQARRLAB{
E_{k'}(f,t) & \leq C \left( 1 + \| f \|^{2}_{H^{k'}} \right)
\label{Eqn:EHs}
}

\item Assume also that $E_{\frac{13}{2}+} (f,t) \leq K $. Then there exists  $C := C(K,\bar{\delta}) > 0$ such that for all $t \in I$

\EQQARRLAB{
\| f \|^{2}_{H^{k'}} & \leq C \left( 1 + E_{k'}(f,t)  \right) \cdot
\label{Eqn:HsE}
}

\end{itemize}

\item Assume that $ \| f \|_{H^{k'}} \leq K $ and $ \| g \|_{H^{k'}} \leq K $. Assume also that for all $t \in I$
and for all $\theta \in  [0,1]$ we have $\inf \limits_{x \in \T}  \left| \p_{\omega_3} F \left( \overrightarrow{h_{\theta}},t \right)(x)  \right| \geq \bar{\delta} $ with $h_{\theta}:= \theta f + (1- \theta) g$. Then there exists $C := C(K,\bar{\delta}) > 0$ such that for all $t \in I$

\EQQARR{
C^{-1} \| f - g \|^{2}_{H^{k'}} & \leq E_{k'}(f,g,t) \leq  C \| f - g \|^{2}_{H^{k'}}
}

\end{enumerate}
\label{prop:compEHnorm}

\end{prop}

\begin{rem}
The proof shows that one can choose the constants depending on $K$ to be continuous functions that increase as $K$ increases.
\end{rem}

\begin{proof}

Lemma \ref{Lem:prod} and Lemma \ref{lem:gauge} imply that

\EQQARR{
\left\| \Phi_{k'} (f,t) \p_{x}^{6} f \right\|_{H^{k'-6}} & \lesssim \| \Phi_{k'} (f,t) \|_{H^{\frac{1}{2}+}} \|  \p_{x}^{6} f \|_{H^{k'-6}}
+ \| \p_{x}^{6} f \|_{H^{\frac{1}{2}+}} \left\| \Phi_{k'} (f,t) \right\|_{H^{k'-6}} \\
& \lesssim \left( 1 + \| f \|_{H^{\frac{7}{2}+}} \right) \| f \|_{H^{k'}}
+ \| f \|_{H^{\frac{13}{2}+}} \left( 1 + \| f \|_{H^{k'-3}} \right) \\
& \lesssim 1 + \| f \|_{H^{k'}} \cdot
}
Hence (\ref{Eqn:EHs}) holds. We have

\EQQARR{
\| D^{k'} f \|^{2}_{L^{2}} & \lesssim \left\| \Phi^{-1}_{k'}(f,t) \Phi_{k'}(f,t) \p_{x}^{6} f \right\|^{2}_{H^{k'-6}} \\
&  \lesssim \left\| \Phi^{-1}_{k'}(f,t) \right\|^{2}_{H^{\frac{1}{2}+}}
\left\| \Phi_{k'} (f,t) \p_{x}^{6} f \right\|^{2}_{H^{k'-6}}  + \left\| \Phi^{-1}_{k'}(f,t) \right\|^{2}_{H^{k'-6}} \left\| \Phi_{k'}(f,t) \p_{x}^{6} f \right\|^{2}_{H^{\frac{1}{2}+}} \\
& \lesssim \left\| \Phi_{k'}(f,t) \p_{x}^{6} f \right\|^{2}_{H^{k'-6}} \left( 1 + \| f \|^{2}_{H^{\frac{7}{2}+}} \right)^{2}
+  \left( 1 + \| f \|^{2}_{H^{k'-3}} \right) E_{\frac{13}{2}+}(f,t) \\
& \lesssim E_{k'}(f,t)
}
Hence (\ref{Eqn:HsE}) holds. We have

\EQQARR{
\left\| \Phi_{k'} (f,t) \p_{x}^{6} f - \Phi_{k'} (g,t) \p_{x}^{6} g  \right\|_{H^{k'-6}} \\
\lesssim \left\| \Phi_{k'}(f,t) - \Phi_{k'}(g,t) \right\|_{H^{k'-6}} \|  \p_{x}^{6} f \|_{H^{\frac{1}{2}+}}
+ \left\| \Phi_{k'} (f,t) - \Phi_{k'} (g,t) \right\|_{H^{\frac{1}{2}+}} \|  \p_{x}^{6} f \|_{H^{k'-6}} \\
+ \left\| \Phi_{k'} (g,t) \right\|_{H^{k'-6}} \left\| \p_{x}^{6} (f-g) \right\|_{H^{\frac{1}{2}+}}
+ \left\| \Phi_{k'} (g,t) \right\|_{ H^{\frac{1}{2}+} } \left\| \p_{x}^{6} (f-g) \right\|_{H^{k'-6}} \\
\lesssim \| f - g \|_{H^{k'}}
}
It remains to show that $\| D^{k'} (f - g) \|^{2}_{L^{2}} \lesssim E_{k'}(f,g,t)$. We have (taking into account
that $k' - 6 > \frac{1}{2}$)

\EQQARR{
\| D^{k'}(f-g) \|^{2}_{L^{2}} & \lesssim \left\| \Phi^{-1}_{k'}(f,t) \Phi_{k'}(f,t) \p_{x}^{6} (f-g)  \right\|^{2}_{H^{k'-6}} \\
& \lesssim  \left\| \Phi^{-1}_{k'}(f,t) \right\|^{2}_{H^{k'-6}} \left\| \Phi_{k'}(f,t) \p_{x}^{6} (f-g)  \right\|^{2}_{H^{k'-6}} \\
& \lesssim  \left\| \Phi_{k'}(f,t) \p_{x}^{6} f - \Phi_{k'} (g,t) \p_{x}^{6} g \right\|^{2}_{H^{k'-6}} + \left\| \Phi_{k'} (g,t) - \Phi_{k'} (f,t) \right\|^{2}_{H^{k'-6}}  \| \p_{x}^{6} g \|^{2}_{H^{k'-6}} \\
& \lesssim \left\| \Phi_{k'} (f,t) \p_{x}^{6} f - \Phi_{k'} (f,t) \p_{x}^{6} g \right\|^{2}_{H^{k'-6}} + \| f -g  \|^{2}_{H^{\max ( k' - 3, \frac{9}{2}+ )}} \\
& \lesssim E_{k'}(f,g,t)
}
\end{proof}

\section{Gauged energy estimates}
\label{Sec:EnergyEst}

In this section we prove some gauged energy estimates. To this end we first prove estimates for a function $v$ that satisfies the equation
(\ref{Eqn:Eqvkpr6}): see Proposition \ref{prop:linearv}. Then we apply these estimates to
$ v(t) := \Phi_{k'}(u(t),t) \partial_{x}^{6} u(t)$  (with $u$ solution of (\ref{Eqn:Regul})) to get the gauged energy estimates.

\subsection{General estimates}

We prove the following proposition:

\begin{prop}
Let $ 0 < \epsilon \ll 1$. Let $k' \geq 6$. Assume that $v$ is a function on an interval $[0,T']$ that satisfies

\EQQARRLAB{
\p_{t} v + \epsilon \p_{x}^{4} v & = a_{3} \p_{x}^{3} v + a_{2} \p_{x}^{2} v + a_{1} \p_{x} v +  a_{0} \\
& + \epsilon b_3 \p_{x}^{3} v + \epsilon b_2 \p_{x}^{2} v + \epsilon b_1 \p_x v + \epsilon b_{0}  + \epsilon \p_{x}^{4} c v \cdot
\label{Eqn:Eqvkpr6}
}
There exists $C > 0$ such that for all $t \in [0,T']$

\EQQARRLAB{
\frac{d \left( \| v(t) \|^{2}_{\dot{H}^{k'-6}} \right)}{dt} + 2 \epsilon \| D^{k'-4} v(t) \|^{2}_{L^{2}}
+ 2 \int_{0}^{T} \int_{\T} \left(  \left( k' -  \frac{15}{2}   \right) \p_{x} a_{3}(t) + a_{2}(t) \right)
 \left( \p_{x} D^{k'-6} v(t) \right)^{2} \; dx \\
\leq C
\left(
\begin{array}{l}
\| a_0 (t) \|^{2}_{H^{k'-6}}  + \| v(t) \|^{2}_{H^{k'-6}}
\left(
1 +  \sum \limits_{m=1}^{3} \| a_m(t) \|_{H^{\left( m + \frac{1}{2} \right)+}} \right) \\
+ \sum \limits_{m=1}^{3} \| a_m (t) \|^{2}_{H^{k'-6 + m}}  \| v(t) \|^{2}_{H^{\frac{1}{2}+}}   \\
+ \epsilon \left( \| b_3(t) \|_{H^{\frac{3}{2}+}} + \| b_2(t) \|_{H^{\frac{1}{2}+}} \right) \| \p_{x} D^{k'-6} v(t) \|^{2}_{L^{2}}  \\
+ \epsilon^{2} \| b_0(t) \|^{2}_{H^{k'-6}} +  \epsilon^{2}
\left(   \sum \limits_{m=1}^{3} \| b_{m}(t) \|^{2}_{H^{k'-6+m}}  + \| c(t) \|^{2}_{H^{k'-2}} \right) \| v(t) \|^{2}_{H^{\frac{1}{2}+}} \\
+ \epsilon \left(
\sum \limits_{m=1}^{3}  \| b_m(t) \|_{H^{ \left( m + \frac{1}{2} \right)+}}
+ \| c(t) \|_{H^{\frac{9}{2}+}}
\right)
\| v(t) \|^{2}_{H^{k'-6}}
\end{array}
\right)
\label{Eqn:EstHv}
}

\label{prop:linearv}
\end{prop}


\begin{proof}

Elementary considerations show that

\EQQARRLAB{
\langle \partial_{t} v(t) + \epsilon \p_{x}^{4} v(t), D^{2 (k'-6)} v(t) \rangle  = \\
\frac{(-1)^{k'-6}}{2} \frac{d \left( \| v(t) \|^{2}_{\dot{H}^{k'-6}} \right)}{dt}  + (-1)^{k'-6} \epsilon \left\| D^{k' - 4} v(t) \right\|^{2}_{L^{2}} \cdot
\label{Eqn:Meth1}
}
On the other hand

\EQQARRLAB{
\langle \p_{t} v(t) + \epsilon \p_{x}^{4} v(t), D^{2 (k'-6)} v(t) \rangle  = \\
(-1)^{k'-6} \left( A_{3}(t) + ... + A_{0}(t) + B_3(t)+ ... + B_0(t) + C(t) \right), \; \text{with}
\label{Eqn:Meth2}
}

\EQQARR{
A_0(t) := \int D^{k'-6} \left( a_0(t) \right) D^{k'-6} v(t) \; dx, \, B_0(t)  := \epsilon \int D^{k'-6}  \left( b_0(t) \right) D^{k'-6} v(t) \; dx, \\
A_m(t) := \int D^{k'-6} \left( a_m(t) \p_{x}^{m} v(t) \right) D^{k'-6} v(t) \; dx, \, \text{and} \,  B_m (t) := \epsilon \int D^{k'-6} \left( b_m(t) \p_{x}^{m} v(t) \right) D^{k'-6} v(t) \; dx \cdot
}
Here $m \in \{ 1,2,3 \}$. We first estimate $A_m(t)$. \\
We see from Lemma \ref{lem:LemD} that

\EQQARR{
A_3 (t) &  = A_3^{a}(t) + A_3^{b}(t) + A_3^{c}(t) \\
&  + O \left( \| v(t) \|_{H^{k'-6}}  \left(  \| a_{3}(t) \|_{H^{k'-3}} \| v(t) \|_{H^{\frac{1}{2}+}} + \| a_{3}(t) \|_{H^{\frac{7}{2}+}} \| v(t) \|_{H^{k'-6}} \right) \right),
}
with

\EQQARR{
A_{3}^{a}(t) & :=  \int a_3(t) D^{k'-6} \p_{x}^{3} v (t) D^{k'-6} v (t) \; dx, \\
A_{3}^{b}(t)  & :=  (k'-6) \int D a_3(t) D^{k'-7} \p_{x}^{3} v(t) D^{k'-6} v (t) \; dx, \; \text{and} \\
A_{3}^{c}(t)  & := \frac{ (k'-6) (k'-7)}{2} \int D^{2} a_3 (t) D^{k'-8} \p_{x}^{3} v(t)  D^{k'-6} v (t) \; dx \cdot
}
Integrations by parts show that

\EQQARR{
A_3^{a}(t) & = - \int \p_{x} a_3 (t) \p_{x}^{2} D^{k'-6} v (t) D^{k'-6} v (t) \; dx
-  \int a_3 (t) \p_{x}^{2} D^{k'-6} v(t) \p_{x} D^{k'-6} v(t) \; dx \\
& =  \int \p_{x}^{2} a_3(t) \p_x D^{k'-6} v(t) D^{k'-6} v(t) \; dx
+ \frac{3}{2} \int \p_{x} a_3(t) \left( \p_x D^{k'-6} v(t) \right)^{2} \; dx  \\
& = - \frac{1}{2} \int \p_{x}^{3} a_3 (t) \left( D^{k'-6} v(t) \right)^{2} \; dx
+ \frac{3}{2} \int \p_{x} a_3(t) \left( \p_x D^{k'-6} v(t) \right)^{2} \; dx \cdot
}
We have

\EQQARR{
A_3^{b}(t) & =  (k'-6) \int \p_{x} a_{3}(t) \p_{x}^{2} D^{k'-6} v(t) D^{k'-6} v(t) \; dx \\
& = -  (k'-6) \int \p_{x}^{2} a_3(t)  \p_{x} D^{k'-6} v(t) D^{k'-6} v(t) \; dx
-  (k'-6) \int \p_{x} a_3(t) \left( \p_x D^{k'-6} v(t) \right)^{2} \; dx  \\
& = \frac{k'-6}{2} \int \p_{x}^{3} a_3(t) \left( D^{k'-6} v(t) \right)^{2} \; dx
-  (k'-6) \int \p_{x} a_3(t) \left( \p_x D^{k'-6} v(t) \right)^{2} \; dx
}
We have

\EQQARR{
A_3^{c}(t) & =  \frac{(k'-6)(k'-7)}{2} \int \p_{x}^{2} a_{3}(t) \p_{x} D^{k'-6} v(t) D^{k'-6} v(t) \; dx \\
& = - \frac{(k'-6)(k'-7)}{4} \int \p_{x}^{3} a_3(t)  \left( D^{k'-6} v(t) \right)^{2} \; dx \cdot
}
We can also write

\EQQARR{
A_{2}(t) & = A_{2}^{a}(t) + A_{2}^{b}(t) \\
& + O \left( \| a_2(t) \|_{H^{k'-4}} \| v(t) \|_{H^{k'-6}} \| v(t) \|_{H^{\frac{1}{2}+}} \right)
+ O \left( \| a_2(t) \|_{H^{\frac{5}{2}+}} \| v(t) \|_{H^{k'-6}}^{2} \right),
}
with

\EQQARR{
A_{2}^{a}(t) & :=  \int a_2(t) D^{k'-6} \p_{x}^{2} v(t) D^{k'-6} v(t) \; dx, \; \text{and} \\
A_{2}^{b}(t) & :=  (k'-6) \int D a_2(t) D^{k'-7} \p_{x}^{2} v(t) D^{k'-6} v(t) \; dx \cdot
}
We have

\EQQARR{
A_{2}^{a}(t) & = -  \int \p_{x} a_2(t)  \p_{x} D^{k'-6} v(t) D^{k'-6} v(t) \; dx -  \int a_2(t)   \left( \p_{x} D^{k'-6} v(t) \right)^{2} \; dx \\
          & = - \int a_2(t) \left( \p_x D^{k'-6} v(t) \right)^{2} \; dx + O \left( \| \p_x^{2} a_2(t) \|_{L^{\infty}} \| v(t) \|^{2}_{H^{k'-6}} \right), \; \text{and}
}

\EQQARR{
A_2^{b}(t) & = - \frac{k'-6}{2} \int \p_{x}^{2} a_2(t) \left( D^{k'-6} v(t) \right)^{2} \; dx = O \left( \| \p_x^{2} a_2 (t) \|_{L^{\infty}}
\| v(t) \|^{2}_{H^{k'-6}} \right) \cdot
 }
We have

\EQQARR{
A_1(t) & =  \int a_1(t) D^{k'-6} \p_x v(t) D^{k'-6} v(t) \; dx + O \left( \| a_1(t) \|_{H^{k'-5}} \| v(t) \|_{H^{\frac{1}{2}+}}
\| v(t) \|_{H^{k'-6}} \right) \\
& + O \left( \| a_1(t) \|_{H^{\frac{3}{2}+}} \| v(t) \|^{2}_{H^{k'-6}} \right)  \\
& = - \frac{1}{2}  \int \p_{x} a_1(t) \left( D^{k'-6} v(t) \right)^{2} \; dx  + O (...) + O (...)  \\
& = O \left( \| \p_{x} a_1(t) \|_{L^{\infty}} \| v(t) \|_{H^{k'-6}}^{2} \right) + O (...) + O (...) \cdot
}
We have

\EQQARR{
| A_0(t)|  \lesssim \left\| D^{k'-6} \left( a_0 (t) \right) \right\|_{L^{2}}  \| D^{k'-6} v(t) \|_{L^{2}} \lesssim
\| a_0 (t) \|_{H^{k'-6}}  \|  v(t) \|_{H^{k'-6}} \cdot
}
We then estimate $B_{m}(t)$ by a similar process to estimate $A_{m}(t)$. We get

\EQQARR{
 \epsilon^{-1} B_{3}(t) & =  \frac{(k'-7)(8-k')}{4} \int \p_{x}^{3} b_3 (t) \left( D^{k'-6} v(t) \right)^{2} \; dx
+ \left( \frac{15}{2} - k' \right) \int \p_{x} b_3(t) \left( \p_{x} D^{k'-6} v(t) \right)^{2} \; dx  \\
& + O \left( \| v(t) \|_{H^{k'-6}}  \left(  \| b_{3}(t) \|_{H^{k'-3}} \| v(t) \|_{H^{\frac{1}{2}+}} + \| b_{3}(t) \|_{H^{\frac{7}{2}+}} \| v(t) \|_{H^{k'-6}} \right) \right),
}

\EQQARR{
 \epsilon^{-1} B_{2}(t) & = - \int b_2(t)   \left( \p_{x} D^{k'-6} v(t) \right)^{2} \; dx
 + O \left( \| b_{2}(t) \|_{H^{k'-4}} \| v(t) \|_{H^{k'-6}}  \| v(t) \|_{H^{\frac{1}{2}+}} \right) \\
 &  + O \left( \| b_2 (t) \|_{H^{\frac{5}{2}+}} \| v(t) \|^{2}_{H^{k'-6}}  \right) + O \left( \| \p_x^{2} b_2(t) \|_{L^{\infty}} \| v(t) \|^{2}_{H^{k'-6}} \right), \; \text{and}
}

\EQQARR{
\epsilon^{-1} B_{1}(t) & =  O \left( \| \p_{x} b_1 (t) \|_{L^{\infty}} \| v(t) \|^{2}_{H^{k'-6}} \right) + O \left( \| b_1(t) \|_{H^{k' - 5}} \| v(t) \|_{H^{\frac{1}{2}+}}
\| v(t) \|_{H^{k'-6}} \right) \\
& + O \left( \| b_1(t) \|_{\frac{3}{2}+} \| v(t) \|^{2}_{H^{k'-6}} \right), \; \text{and}
}

\EQQARR{
\epsilon^{-1} | B_0(t)| & \lesssim \| b_0(t) \|_{H^{k'-6}} \| v(t) \|_{H^{k'-6}} \cdot
}
We get from Lemma \ref{Lem:prod}

\EQQARR{
\epsilon^{-1} |C(t)| & \lesssim  \left( \| c(t) \|_{H^{k'-2}} \| v(t) \|_{H^{\frac{1}{2}+}} + \| c(t) \|_{H^{\frac{9}{2}+}} \| v(t) \|_{H^{k'-6}}
\right) \| v(t) \|_{H^{k'-6}} \cdot
}
Next we divide (\ref{Eqn:Meth1}) and (\ref{Eqn:Meth2}) by $(-1)^{k'-6}$. We see that we can bound from above
$\frac{d  \left( \| v(t) \|^{2}_{\dot{H}^{k-6}} \right)}{dt}x   + 2 \epsilon \| D^{k'-4} v(t) \|^{2}_{L^{2}} + 2 \int \left( \left( k' - \frac{15}{2} \right) \p_{x} a_{3}(t) + a_2(t) \right)
\left( \p_{x} D^{k'-6} v(t) \right)^{2} \; dx  $ by the RHS of (\ref{Eqn:EstHv})
by using the above estimates, the Sobolev embedding $H^{\frac{1}{2}+} \hookrightarrow L^{\infty}$,
and the Young inequality $ab \leq \frac{a^{2}}{2} + \frac{b^{2}}{2}$.

\end{proof}

We also prove the proposition below:

\begin{prop}
Let $ 0 < \epsilon \ll 1$. The following hold:

\begin{enumerate}

\item Let $k' \geq 3$, $K \geq 0$, and $j \in  \left\{ \frac{9}{2}+, k' -3 \right\} $.
Let $u$ be a solution of (\ref{Eqn:Integral}) obtained by Proposition \ref{Prop:Integral} on an interval $[0,T']$.
Assume that $ \sup_{t \in [0,T']} \| u(t) \|_{H^{\frac{9}{2}+}} \leq K $. Then there exists $C := C(K) > 0$ such that

\EQQARRLAB{
\frac{d \left( \| u(t) \|^{2}_{H^{j}} \right)}{dt}  & \leq C \left(  1 + \| u(t) \|^{2}_{H^{j+3}} \right) \cdot
\label{Eqn:EqvL20}
}

\item Assume that $v$ is a function on an interval $[0,T]$ that satisfies

\EQQARRLAB{
\p_{t} v + \epsilon \p_{x}^{4} v & = a_{3} \p_{x}^{3} v + a_{2} \p_{x}^{2} v + a_{1} \p_{x} v + a_{0} \\
& + \epsilon b_3 \p_{x}^{3} v + \epsilon b_2 \p_{x}^{2} v + \epsilon b_1 \p_x v  + \epsilon b_0 +  \epsilon \p_{x}^{4} c v \cdot
\label{Eqn:EqvL2}
}
There exists $C > 0$ such that for all $t \in [0,T']$

\EQQARRLAB{
\frac{d \left( \| v(t) \|^{2}_{L^{2}} \right)}{dt} + 2 \epsilon \| \p_{x}^{2} v(t) \|^{2}_{L^{2}}  \\
\leq C
\left(
\begin{array}{l}
 \| a_0 (t) \|^{2}_{L^{2}} + \left( 1 + \sum \limits_{m=1}^{3} \| a_m(t) \|_{H^{ \left(m + \frac{1}{2} \right)+}} \right)
\| v(t) \|^{2}_{L^{2}}  \\
+ \sum \limits_{m=1}^{3} \| a_m (t) \|^{2}_{H^{m}}  \| v(t) \|^{2}_{H^{\frac{1}{2}+}}  \\
+ \left( \| a_3(t) \|_{H^{\frac{3}{2}+}} + \| a_2(t) \|_{H^{\frac{1}{2}+}} \right) \| v(t) \|^{2}_{H^{1}} \\
+ \epsilon \left( \| b_3(t) \|_{H^{\frac{3}{2}+}} + \| b_2(t) \|_{H^{\frac{1}{2}+}} \right) \| \p_{x} v(t) \|^{2}_{L^{2}}  \\
+ \epsilon^{2} \| b_0(t) \|^{2}_{L^{2}} +  \epsilon^{2}
\left(   \sum \limits_{m=1}^{3} \| b_{m}(t) \|^{2}_{H^{m}}  + \| c(t) \|^{2}_{H^{4}} \right) \| v(t) \|^{2}_{H^{\frac{1}{2}+}} \\
+ \epsilon \left(
\sum \limits_{m=1}^{3}  \| b_m(t) \|_{H^{ \left( m + \frac{1}{2} \right)+}}
+ \| c(t) \|_{H^{\frac{9}{2}+}}
\right)
\| v(t) \|^{2}_{L^{2}}
\end{array}
\right)
\label{Eqn:EstL2v}
}

\end{enumerate}
\label{prop:linearL2v}
\end{prop}

\begin{rem}
Regarding the constant $C:=C(K)$ in (\ref{Eqn:EqvL20}) of Proposition \ref{prop:linearL2v}, we refer to Remark \ref{Rem:EstDerivOp}.
\label{Rem:linearL2v}
\end{rem}

\begin{proof}

Let $m \in  \left\{0, \frac{9}{2}+, k'-3 \right\}$. Elementary considerations show that

\EQQARR{
\langle \p_{t} u(t) + \epsilon \p_{x}^{4} u(t),  D^{2m} u(t) \rangle & = \frac{(-1)^{m}}{2}
\frac{d \left(  \| D^{m} u(t) \|^{2}_{L^{2}} \right)}{dt} + (-1)^{m} \epsilon
\| D^{m+2} u(t) \|^{2}_{L^{2}}
}
On the other hand $ \langle F(\overrightarrow{u(t)}), D^{2m} u(t) \rangle
= (-1)^{m} \langle D^{m} \left( F(\overrightarrow{u(t)}) \right), D^{m} u(t) \rangle $.
Hence we see from Lemma \ref{lem:EstDerivF}, the Cauchy-Schwartz inequality, and elementary estimates that

\EQQARR{
\frac{d \left( \| u(t) \|^{2}_{\dot{H}^{m}} \right)}{dt}  & \lesssim  1 + \| u(t) \|^{2}_{H^{m+3}} \cdot
}
Hence (\ref{Eqn:EqvL2}) holds. \\
Mimicking the proof of Proposition \ref{prop:linearv} (with $k'$ replaced with $6$) we see that $ (*): \frac{d \left(  \| v(t) \|^{2}_{L^{2}} \right) }{dt}   + 2 \epsilon \| \p_{x}^{2} v(t) \|^{2}_{L^{2}} + 2 \int \left( -\frac{3}{2}  \p_{x} a_{3}(t) + a_2(t) \right) \left( \p_{x} v(t) \right)^{2} \; dx $ is bounded by the sum of the RHS of (\ref{Eqn:EstL2v}) (with again $k'$ replaced with $6$).
Observe  that the second term of $(*)$ is bounded by a constant multiplied by  $ \left( \| \partial_{x} a_{3}(t) \|_{L^{\infty}} + \| a_{2}(t) \|_{L^{\infty}} \right) \| v(t) \|^{2}_{H^{1}} $. Hence, using also
the Sobolev embedding $H^{\frac{1}{2}+} \hookrightarrow L^{\infty} $, we see that (\ref{Eqn:EstL2v}) holds.

\end{proof}

\subsection{Gauged energy estimates: statement}

In this subsection we prove the following gauged energy estimates:

\begin{prop}
Let $ k' \geq k'_{0} > \frac{17}{2} $. Let  $ \tilde{\phi} \in \widetilde{\mathcal{P}}_{+,k'} $  be a smooth function. Then there exist $1 \geq T' := T' \left( \| \tilde{\phi} \|_{H^{k'_{0}}}, \tilde{\delta}( \overrightarrow{\tilde{\phi}} ), \tilde{\delta}^{'}( \tilde{\phi}) \right) > 0 $ and
$ C :=  \left( \| \tilde{\phi} \|_{H^{k'_{0}}}, \tilde{\delta}( \overrightarrow{\tilde{\phi}} ) \right) $ such that if $u$ is the solution of (\ref{Eqn:Integral}) obtained by Proposition \ref{Prop:Integral} on $[0,T_{\epsilon})$ then

\EQQARRLAB{
T_{\epsilon} > T'; \; \sup_{t \in [0,T']} \left( E_{k'}(u(t),t) + \int_{0}^{t} \int_{\T} Q \left( u(t^{'}),t' \right) \left( \partial_{x} D^{k'-6} v(t') \right)^{2} dx \; dt' \right)
\leq 2  \left( 1 + E_{k'}(\tilde{\phi},0) \right), \\
t \in [0,T']: \; \frac{d E_{k'} (u(t),t )}{dt}  + \int_{\T} Q (u(t),t) \left( \partial_{x} D^{k'-6} v(t) \right)^{2} \; dx
\leq C \left( 1 + E_{k'}(u(t),t) \right), \\
u(t) \in \mathcal{P}_{+,k'}(t), \; \delta  \left( \overrightarrow{u(t)},t \right) \gtrsim \tilde{\delta} \left( \overrightarrow{\tilde{\phi}} \right), \; \text{and} \;
\delta^{'} \left( u(t),t \right) \gtrsim \tilde{\delta}^{'} (\tilde{\phi}),
\label{Eqn:BoundNrjVisc}
}
with $v(t):= \Phi_{k'} \left( u(t),t \right) \partial_{x}^{6} u(t) $.


 \label{Prop:OneEnergyEst}
\end{prop}


\begin{rem}
The proof of Proposition \ref{Prop:OneEnergyEst} shows that $T'$ (resp. $C$) can be chosen as a continuous function that decreases (resp. increases) as $\tilde{\delta} \left( \overrightarrow{\tilde{\phi}} \right)$ decreases, decreases (resp. increases) as $\| \tilde{\phi} \|_{H^{k_{0}^{'}}}$ increases. It also shows that
$T'$ can be chosen as a continuous function that decreases as $ \tilde{\delta}^{'}(\tilde{\phi})$ decreases.
\label{Rem:PropOneEnergyEst}
\end{rem}


\begin{proof}

The proof relies upon the lemma below:

\begin{lem}
Let $K \geq 0$. Let $u$ be a solution of (\ref{Eqn:Integral}) on an interval $[0,T^{''}]$ with $0 < T^{''} \leq 1$. Let $\bar{\delta} > 0$. Assume that
that  $\delta \left( \overrightarrow{u(t)}, t \right) \geq \bar{\delta}$ for all $t \in [0,T^{''}]$. Assume also that
$\sup_{t \in [0,T^{''}]} E_{\frac{17}{2}+} \left(  u(t),t \right) \leq K$. Then there exists $C:= C(K,\bar{\delta}) > 0$ such that

\EQQARRLAB{
t \in [0,T^{''}]: \; \frac{d E_{k'} (u(t),t)}{dt} + \int_{\T} Q (u(t),t) \left( \partial_{x} D^{k'-6} v(t) \right)^{2} \; dx
\leq C \left( 1 + E_{k'} (u(t),t)  \right) \cdot
\label{Eqn:EstDerE}
}

\label{lem:OneEnergyEst}
\end{lem}

\begin{rem}
The proof of Lemma \ref{lem:OneEnergyEst} shows that $C$ can be chosen as a continuous function that increases as $K$ increases,
and that increases as $\bar{\delta}$ decreases.
\label{Rem:lemOneEnergyEst}
\end{rem}


We postpone the proof of Lemma \ref{lem:OneEnergyEst} to Section \ref{Sec:ProofLemmas}. \\
\\
Let

\EQQARR{
T_{\epsilon}^{*} & :=  \inf \left\{ t \geq 0: \; E_{k'_{0}} \left( u(t),t \right) > 2 \left( 1 + E_{k'_{0}} (\tilde{\phi},0) \right)  \right\} \cdot
}
We see from Proposition \ref{Prop:Integral} and Proposition \ref{prop:compEHnorm} that   $ 0 < T_{\epsilon}^{*} < T_{\epsilon}$. Moreover
 $ E_{k'_{0}} \left( u(t),t \right) \leq 2 \left( 1 + E_{k'_{0}} (\tilde{\phi}, 0) \right) $
for all $t \in [0,T_{\epsilon}^{*} ]$. We have $ \sup_{ t \in [0,T_{\epsilon}^{*}] } E_{k'_{0}} \left( u(t), t \right) \lesssim_{\| \tilde{\phi} \|_{H^{k'_{0}}}} 1 $ and
$ \sup_{ t \in [0,T_{\epsilon}^{*}] } \| u(t) \|^{2}_{H^{k'_{0}}} \lesssim 1 + \| \tilde{\phi} \|^{2}_{H^{k'_{0}}}  $. Let
$T' :=   T' \left( \tilde{\delta} \left( \overrightarrow{\tilde{\phi}} \right), \tilde{\delta}^{'}(\tilde{\phi}), \| \tilde{\phi} \|_{H^{k_{0}^{'}}} \right) > 0 $ be small enough such that all the statements and estimates below hold. \\
\\
Assume that $ 0 <  T_{\epsilon}^{*} \leq T' $. Let $t \in [0,T_{\epsilon}^{*}]$. \\
\\
\underline{Claim}: $  \delta \left( \overrightarrow{u(t)}, t \right) \gtrsim  \tilde{\delta} \left( \overrightarrow{\tilde{\phi}} \right) $ \\
\\
Indeed we see from Appendix  that  $ \left\| X ( u(t), \tilde{\phi} ,t,0) \right\|_{L^{\infty}}
\lesssim_{\| \tilde{\phi} \|_{H^{k_{0}^{'}}}}  \| u(t) - \tilde{\phi} \|_{H^{\frac{7}{2}+}} + t $. We have $t \ll \tilde{\delta} \left( \overrightarrow{\tilde{\phi}} \right) $.
So it suffices to estimate  $ \| u(t) - \tilde{\phi} \|_{H^{\frac{7}{2}+}} $. In fact we will estimate $ \| u(t) - \tilde{\phi} \|_{H^{\frac{9}{2}+}} $,
since we will also use this estimate at the end of the proof.  Let $u_{l}(t) := e^{-\epsilon t \p_{x}^{4}} \tilde{\phi}$ and
$ u_{nl}(t) :=  \int_{0}^{t} e^{- \epsilon (t - t') \p_{x}^{4}} F \left( \overrightarrow{u(t')}, t' \right) \; dt' $.  Elementary estimates show that
$ \left|  e^{-x}  -  1 \right| \lesssim \min ( x, 1 ) $ if $ x \geq 0 $.  Hence, dividing the region of summation into $|n| \leq N$ and $|n| \geq N $ with
$N \approx \frac{1}{(\epsilon t)^{\frac{1}{4}}}$ we get

\EQQARR{
\| u_{l}(t) - \tilde{\phi} \|^{2}_{H^{\frac{9}{2}+}} & \lesssim \sum \limits_{n \in \Z} \langle  n \rangle^{2 \left( \frac{9}{2}+ \right) }
\left| e^{-\epsilon t (in)^{4} } - 1  \right|^{2} | \widehat{\tilde{\phi}}(n) |^{2} \\
& \lesssim t^{0+} \| \tilde{\phi} \|^{2}_{H^{k'_{0}}} \\
& \ll \tilde{\delta} \left( \overrightarrow{\tilde{\phi}} \right) \cdot
}
The Minkowski inequality and Lemma \ref{lem:EstDerivF} show that

\EQQARR{
\| u_{nl}(t) \|_{H^{\frac{9}{2}+}} & \lesssim \int_{0}^{t}
\| e^{- \epsilon (t-t') \partial_{x}^{4}} F \left( \overrightarrow{u(t')}, t' \right) \|_{H^{\frac{9}{2}+}} \; d t' \\
& \lesssim_{ \| \tilde{\phi} \|_{H^{k'_{0}}} } t  \left( 1 + \sup_{t' \in [0, t]} \| u(t') \|_{H^{\frac{15}{2}+}} \right) \\
& \lesssim_{ \| \tilde{\phi} \|_{H^{k'_{0}}} } t  \\
& \ll \tilde{\delta} \left( \overrightarrow{\tilde{\phi}} \right) \cdot
}
Hence we see from the above estimates and the triangle inequality that the claim holds.\\
\\
\underline{Claim}:  $u(t) \in \mathcal{P}_{+,k'}(t)$ and $ \delta^{'}(u(t),t) \gtrsim \tilde{\delta}^{'} (\tilde{\phi}) $  \\
\\
From the claim above and similar arguments as those in Remark \ref{Rem:Qother} we see that $\partial_{\omega_{3}} F \left( \overrightarrow{u(t)}, t \right) $ has a constant sign. Hence $Q (u(t),t) =   \left| \partial_{\omega_{3}} F ( \overrightarrow{u(t)}, t ) \right|
\left[ \frac{P( u(t),t)}{\left| \partial_{\omega_{3}} F ( \overrightarrow{u(t)}, t) \right|}  \right]_{ave} $. \\
We see from Appendix and the above estimates that $ \left\| Y( u(t),\tilde{\phi},t,0 ) \right\|_{L^{\infty}}
\lesssim_{ \tilde{\delta} \left( \overrightarrow{\tilde{\phi}} \right), \| \tilde{\phi} \|_{H^{k'_{0}}} }
\left\| u(t) - \tilde{\phi} \right\|_{H^{\frac{9}{2}+}} + t
\lesssim_{ \tilde{\delta} \left( \overrightarrow{\tilde{\phi}} \right), \| \tilde{\phi} \|_{H^{k'_{0}}}} \left( t +  \epsilon^{0+} t^{0+} \right)
\ll \tilde{\delta}^{'}(\tilde{\phi}) $.  Hence an application of the triangle inequality yields the claim.   \\
\\
Hence we can apply Lemma \ref{lem:OneEnergyEst} and Gronwall inequality to get
$ (*): \; E_{k_{0}^{'}} (u(t),t) \leq e^{C^{'} t}  \left( E_{k_{0}^{'}} \left( \tilde{\phi},0  \right) + C^{'} t \right)
< \frac{3}{2} \left( 1 + E_{k_{0}^{'}} ( \tilde{\phi},0) \right)$
for some $C^{'} := C^{'} \left( \| \tilde{\phi} \|_{H^{k'_{0}}}, \tilde{\delta} \left( \overrightarrow{\tilde{\phi}} \right) \right)$   \\

Hence by letting $ t= T^{*}_{\epsilon}$ in the above inequality, we see that it contradicts cannot the definition of $ T_{\epsilon}^{*} $ . Hence $T_{\epsilon}^{*} > T' $.
\\
We then mimick the proof from ``\underline{Claim}: $  \delta (\overrightarrow{u(t)},t) \gtrsim  \tilde{\delta}( \overrightarrow{\tilde{\phi}} ) $''  to
``yields the claim.'' to conclude that for $t \in [0,T^{'}]$ we have $u(t) \in \mathcal{P}_{+,k'}(t) $, $ \delta ( \overrightarrow{u(t)}, t ) \gtrsim  \tilde{\delta} ( \overrightarrow{\tilde{\phi}} ) $, and $ \delta^{'}(u(t),t) \gtrsim \tilde{\delta}^{'} (\tilde{\phi}) $. Hence we can apply again  Lemma \ref{lem:OneEnergyEst} and
Gronwall inequality to get (\ref{Eqn:BoundNrjVisc}).

\end{proof}

\section{Gauged energy estimates for the difference of two solutions}
\label{Sec:EnergyDiffEst}

In this section we prove gauged energy estimates for the difference of two solutions of an equation of the form (\ref{Eqn:Integral}). To this end
we first prove estimates for a function $\bar{v}: = v_{1} - v_{1}$ with $v_{1}$, $v_{2}$ that satisfy (\ref{Eqn:vjEq}): see Proposition
\ref{prop:linearbarv} below. Then we apply these estimates to $\bar{v}(t) := \Phi_{k'}(u_{1}(t),t) \partial_{x}^{6} u_{1}(t) - \Phi_{k'}(u_{2}(t),t) \partial_{x}^{6} u_{2}(t) $
( with $u_{1}$, $u_{2}$ solutions of (\ref{Eqn:Regul})) in order to get the gauged energy estimates for the difference of two solutions.

\subsection{General estimates}

We prove the following proposition:

\begin{prop}
Let $j \in \{1,2\}$ and $ 0 < \epsilon_1 \leq \epsilon_2 \ll 1$. Let $k' \geq 6$. Assume that $v_j$ is a function that satisfies
on an interval $[0,T']$

\EQQARRLAB{
\p_{t} v_j + \epsilon_{j} \p_{x}^{4} v_j & = a_{3,j} \p_{x}^{3} v_j + a_{2,j} \p_{x}^{2} v_j + a_{1,j} \p_{x} v_j  + a_{0,j} \\
& + \epsilon_j b_{3,j} \p_{x}^{3} v_j + \epsilon_j b_{2,j} \p_{x}^{2} v_j + \epsilon_j b_{1,j} \p_x v_j + \epsilon_{j} b_{0,j} + \epsilon_{j} \p_{x}^{4} c_j v_j \cdot
\label{Eqn:vjEq}
}
Let $\bar{v} := v_1 - v_2$. Then there exists $C > 0$ such that for all time $t \in [0,T']$

\EQQARRLAB{
\frac{d  \left( \| \bar{v}(t) \|^{2}_{\dot{H}^{k'-6}} \right)}{dt} + 2 \epsilon_1 \| D^{k'-4} \bar{v} (t) \|^{2}_{L^{2}}  + 2 \int_{\T}  \left(  \left( k' - \frac{15}{2}  \right) \p_{x} a_{3,1}(t)  + a_{2,1}(t) \right) \left( \p_{x} D^{k'-6} \bar{v}(t) \right)^{2} \; dx \\
\leq C
\left(
\begin{array}{l}
\left(
1 + \sum \limits_{m=1}^{3} \| a_{m,1}(t) \|_{H^{ \left( m + \frac{1}{2} \right)+ }} \right)
\| \bar{v}(t) \|^{2}_{H^{k'-6}} + \sum \limits_{m=1}^{3}  \| a_{m,1}(t) \|^{2}_{H^{k'-6+m}}  \| \bar{v}(t) \|^{2}_{H^{\frac{1}{2}+}}  \\
+ \epsilon_1  \left( \| b_{3,1}(t)\|_{H^{\frac{3}{2}+}} + \| b_{2,1}(t) \|_{H^{\frac{1}{2}+}} \right)
\| \p_{x} D^{k'-6} \bar{v}(t) \|^{2}_{L^{2}} \\
+ \epsilon_{1}^{2} \| b_{0,1}(t) \|^{2}_{H^{k'-6}} + \epsilon^{2}_{2} \| b_{0,2}(t) \|^{2}_{H^{k'-6}} \\
+ \epsilon_1^{2} \left( \sum \limits_{m=1}^{3} \| b_{m,1}(t) \|^{2}_{H^{k'-6+m}} + \| c_1(t) \|^{2}_{H^{k'-2}} \right) \| \bar{v}(t) \|^{2}_{H^{\frac{1}{2}+}} \\
+ \epsilon_{1}^{2} \left( \| c_2(t) - c_1(t) \|^{2}_{H^{k'-2}}  \| v_2(t) \|^{2}_{H^{\frac{1}{2}+}}
+ \| c_2(t) - c_1(t) \|^{2}_{H^{\frac{9}{2}+}} \| v_2(t) \|^{2}_{H^{k'-2}} \right) \\
+ \epsilon_1  \left( \sum \limits_{m=1}^{3} \| b_{m,1}(t) \|_{H^{ \left( m + \frac{1}{2} \right)+}} + \| c_1(t) \|_{H^{\frac{9}{2}+}}  \right) \| \bar{v}(t) \|^{2}_{H^{k'-6}} \\
+ \epsilon_2^{2}
\left(
\begin{array}{l}
\sum \limits_{j=1}^{2} \sum \limits_{m=1}^{3}
\left(
\| b_{m,j}(t) \|^{2}_{H^{k'-6}} \| v_2(t) \|^{2}_{H^{ \left(m + \frac{1}{2} \right)+}} + \| b_{m,j}(t) \|^{2}_{H^{\frac{1}{2}+}} \| v_2(t) \|^{2}_{H^{k'-6+m}}
\right) \\
+ \| c_2(t) \|^{2}_{H^{k'-2}} \| v_2(t) \|^{2}_{H^{\frac{1}{2}+}}  + \| c_2(t) \|^{2}_{H^{\frac{9}{2}+}} \| v_2(t) \|^{2}_{H^{k'-6}} + \| v_2(t) \|^{2}_{H^{k'-2}}
\end{array}
\right) \\
+ \| a_{0,1}(t) - a_{0,2}(t) \|^{2}_{H^{k'- 6}} \\
+ \sum \limits_{m=1}^{3}
\left(
\begin{array}{l}
\| a_{m,1}(t) - a_{m,2}(t) \|^{2}_{H^{k'-6}} \| v_2(t) \|^{2}_{H^{ \left( m+ \frac{1}{2} \right)+}} \\
\\
+ \| a_{m,1}(t) - a_{m,2}(t) \|^{2}_{H^{\frac{1}{2}+}} \| v_{2}(t) \|^{2}_{H^{m+k'-6}}
\end{array}
\right)
\end{array}
\right)
\label{Eqn:Estbarv}
}

\label{prop:linearbarv}
\end{prop}


\begin{proof}
We have

\EQQARR{
\p_{t} \bar{v} + \epsilon_1 \p_{x}^{4} \bar{v}  & = (\epsilon_2 - \epsilon_1) \p_{x}^{4} v_2 + (\epsilon_1 - \epsilon_2) \p_x^{4} c_2 v_2  - \epsilon_1 \p_{x}^{4} (c_2 - c_1) v_2
+ \epsilon_1 \p_{x}^{4} c_1 \bar{v} \\
& + \sum \limits_{m=1}^{3} a_{m,1} \p_{x}^{m} \bar{v} + \sum \limits_{m=1}^{3} (a_{m,1} - a_{m,2}) \p_{x}^{m} v_2  + a_{0,1} - a_{0,2} \\
&  + \epsilon_1 \sum \limits_{m=1}^{3} b_{m,1} \p_{x}^{m} \bar{v} + \epsilon_1 \sum \limits_{m=1}^{3} b_{m,1} \p_{x}^{m} v_2
 - \epsilon_2 \sum \limits_{m=1}^{3} b_{m,2} \p_{x}^{m} v_2 + \epsilon_{1} b_{0,1} - \epsilon_{2} b_{0,2}
}
Elementary considerations  show that

\EQQARRLAB{
\langle \p_{t} \bar{v}(t) + \epsilon_1 \p_{x}^{4} \bar{v}(t),  D^{2 (k'-6)} \bar{v}(t) \rangle & =
(- 1)^{k'-6} \left( \frac{1}{2} \frac{d \left( \| \bar{v}(t) \|^{2}_{\dot{H}^{k'-6}} \right)}{dt}
+ \epsilon_1 \| D^{k'-4} \bar{v}(t) \|^{2}_{L^{2}} \right)  \cdot
\label{Eqn:Beglinearbarv}
}
The proof of Proposition \ref{prop:linearv}, Lemma \ref{lem:LemD}, and the Sobolev embedding $ H^{\frac{1}{2}+} \hookrightarrow L^{\infty} $ show that

\EQQARR{
\sum \limits_{m=1}^{3} \langle a_{m,1}(t) \p_{x}^{m} \bar{v}(t), D^{2(k'-6)} \bar{v}(t) \rangle \\
 =
(-1)^{k'-6}
\left\langle \left( \frac{15}{2} - k' \right) \p_{x} a_{3,1}(t) - a_{2,1}(t),  \left( \p_{x} D^{k'-6} \bar{v}(t) \right)^{2} \right\rangle \\
+ O
\left(
\| \bar{v}(t) \|_{H^{k'-6}}
\left(
\begin{array}{l}
\sum \limits_{m=1}^{3}  \| a_{m,1}(t) \|_{H^{ \left( m+ \frac{1}{2} \right) +}} \| \bar{v}(t) \|_{H^{k'-6}} \\
+ \sum  \limits_{m=1}^{3} \| a_{m,1}(t) \|_{H^{k'-6+ m}} \| \bar{v}(t) \|_{H^{\frac{1}{2}+}}
\end{array}
\right)
\right),
}

\EQQARR{
\left| \langle a_{0,1}(t) - a_{0,2}(t), D^{2(k'-6)} \bar{v}(t) \rangle \right|  & \lesssim \| a_{0,1}(t) - a_{0,2}(t) \|_{H^{k'-6}} \| \bar{v}(t) \|_{H^{k'-6}}
}

\EQQARR{
\epsilon_{1} \sum \limits_{m=1}^{3} \langle b_{m,1}(t) \p_{x}^{m} \bar{v}(t), D^{2 (k'-6)} \bar{v}(t) \rangle &
= (-1)^{k'-6} \left( X(t) + Y(t) + Z(t) \right),
}
with

\EQQARR{
 \epsilon_1^{-1}   X(t) & = \frac{(k'-7)(8-k')}{4} \int \p_{x}^{3} b_{3,1}(t) \left( D^{k'-6} \bar{v}(t) \right)^{2} \; dx
+ \left( \frac{15}{2} - k' \right) \int \p_{x} b_{3,1}(t) \left( \p_{x} D^{k'-6} \bar{v}(t) \right)^{2} \; dx  \\
& + O \left( \| \bar{v}(t) \|_{H^{k'-6}} \left(  \| b_{3,1}(t) \|_{H^{k'-3}} \| \bar{v}(t) \|_{H^{\frac{1}{2}+}}
+ \| b_{3,1}(t) \|_{H^{\frac{7}{2}+}} \| \bar{v}(t) \|_{H^{k'-6}} \right)   \right),
}

\EQQARR{
\epsilon_1^{-1} Y(t) & = -  \int b_{2,1}(t) \left( \p_{x} D^{k'-6} \bar{v}(t) \right)^{2} \; dx  \\
& + O \left( \| \bar{v}(t) \|_{H^{k'-6}}
\left(
\| b_{2,1}(t) \|_{H^{k'-4}} \| \bar{v}(t) \|_{H^{\frac{1}{2}+}}  +  \| b_{2,1}(t) \|_{H^{\frac{5}{2}+}}
\| \bar{v}(t) \|_{H^{k'-6}}
\right)
\right)\ \; \text{and}
}

\EQQARR{
\epsilon_1^{-1} Z(t) & = O \left( \| b_{1,1}(t) \|_{H^{\frac{3}{2}+}} \| \bar{v}(t) \|^{2}_{H^{k'-6}}
+  \|  b_{1,1}(t) \|_{H^{k'-5}} \| \bar{v}(t) \|_{H^{\frac{1}{2}+}} \| \bar{v}(t)  \|_{H^{k'-6}}
\right) \cdot
}
We also have

\EQQARR{
\left| \epsilon_1 \langle \p_{x}^{4} c_1(t) \bar{v}(t), D^{2(k'-6)} \bar{v}(t) \rangle  \right|
& = \left| (-1)^{k'-6} \epsilon_1 \langle D^{k'-6} \left( \p_{x}^{4} c_1(t) \bar{v}(t) \right), D^{k'-6} \bar{v}(t) \rangle \right| \\
& \lesssim \epsilon_1 \| \bar{v}(t) \|_{H^{k'-6}}
\left(
\begin{array}{l}
\| c_1(t) \|_{H^{\frac{9}{2}+}} \| \bar{v}(t) \|_{H^{k'-6}}
+ \| c_1(t) \|_{H^{k'-2}} \| \bar{v}(t) \|_{H^{\frac{1}{2}+}}
\end{array}
\right) \cdot
}
We also have

\EQQARR{
\left| \langle \epsilon_1 b_{0,1}(t) - \epsilon_2 b_{0,2}(t), D^{2(k'-6)} \bar{v}(t)  \rangle \right|
& \lesssim \epsilon_1 \left| \langle b_{0,1}(t), D^{2(k'-6)} \bar{v}(t) \rangle \right|
+ \epsilon_2 \left| \langle b_{0,1}(t), D^{2(k'-6)} \bar{v}(t) \rangle  \right| \\
& \lesssim
\left(
 \epsilon_1 \| b_{0,1}(t) \|_{H^{k'-6}} + \epsilon_2 \| b_{0,2}(t) \|_{H^{k'-6}}
\right)  \| \bar{v}(t) \|_{H^{k'-6}} \cdot
}
Lemma \ref{Lem:prod} shows that

\EQQARR{
\left| \langle  \left( a_{m,1}(t) - a_{m,2}(t) \right) \p_{x}^{m} v_2(t), D^{2(k'-6)} \bar{v}(t) \rangle  \right| \\
\\
= \left| (-1)^{k'-6} \langle D^{k'-6} \left( \left( a_{m,1}(t) - a_{m,2}(t) \right) \p_{x}^{m} v_2(t) \right), D^{k'-6} \bar{v}(t)  \rangle \right| \\
\\
\lesssim
\left(
\begin{array}{l}
\left\| a_{m,1}(t) - a_{m,2} (t) \right\|_{H^{k'-6}} \left\| v_2(t) \right\|_{H^{ \left( m  + \frac{1}{2} \right)+}} \\
+ \left\| a_{m,1}(t) - a_{m,2} (t) \right\|_{H^{\frac{1}{2}+}} \left\| v_2(t) \right\|_{H^{m + k' - 6}}
\end{array}
\right)
\| \bar{v}(t) \|_{H^{k'-6}} \cdot
}
If $j \in \{1,2\} $ then

\EQQARR{
\epsilon_j \left| \langle b_{m,j}(t) \p_{x}^{m} v_2(t), D^{2(k'-6)} \bar{v}(t)  \rangle \right| &
 \lesssim \epsilon_j
 \left(
 \begin{array}{l}
 \| b_{m,j}(t) \|_{H^{k'-6}} \| v_{2}(t) \|_{H^{ \left( m + \frac{1}{2} \right)+}}  \\
 + \| b_{m,j}(t) \|_{H^{\frac{1}{2}+}} \| v_{2}(t) \|_{H^{k'-6 +m}}
\end{array}
\right)
\| \bar{v}(t) \|_{H^{k'-6}} \cdot
}
We have

\EQQARRLAB{
(\epsilon_2 - \epsilon_1) \left| \langle \p_{x}^{4} c_2(t) v_2(t), D^{2(k'-6)} \bar{v}(t) \rangle \right| \\
= (\epsilon_2 - \epsilon_1) \left| (-1)^{k'-6} \langle D^{k'-6} \left( \p_{x}^{4} c_2(t) v_2(t) \right), D^{k'-6} \bar{v}(t) \rangle \right| \\
\lesssim \epsilon_2 \left(  \| c_2(t) \|_{H^{k'-2}} \| v_2(t) \|_{H^{\frac{1}{2}+}} + \| c_2(t) \|_{H^{\frac{9}{2}+}}  \| v_2(t) \|_{H^{k'-6}} \right)
\| \bar{v}(t) \|_{H^{k'-6}},
\label{Eqn:c2est}
}

\EQQARR{
(\epsilon_2 - \epsilon_1) \left| \langle \p_{x}^{4} v_2(t), D^{2(k'-6)} \bar{v}(t)  \rangle \right|
& = (\epsilon_2 - \epsilon_1) \left| (-1)^{k'-6} \langle D^{k'-6} \p_{x}^{4} v_2(t), D^{k'-6} \bar{v}(t) \rangle  \right| \\
& \lesssim \epsilon_{2}  \| v_2(t) \|_{H^{k'-2}}  \| \bar{v}(t) \|_{H^{k'-6}}, \, \text{and}
\label{Eqn:v2est}
}

\EQQARRLAB{
\epsilon_1 \left| \langle \p_{x}^{4} \left( c_1(t) - c_2(t) \right)  v_2(t), D^{2(k'-6)} \bar{v}(t) \rangle \right| \\
= \epsilon_1 \left| (-1)^{k'-6}  \langle D^{k'-6} \left( \p_{x}^{4} \left(  c_1(t)-c_2(t) \right) v_2(t) \right), D^{k'-6} \bar{v}(t) \rangle \right| \\
\lesssim \epsilon_1  \left( \| c_2(t) - c_1(t) \|_{H^{k'-2}} \| v_2(t) \|_{H^{\frac{1}{2}+}} + \| c_2(t) - c_1(t) \|_{H^{\frac{9}{2}+}}
\| v_2(t) \|_{H^{k'-2}} \right) \| \bar{v}(t) \|_{H^{k'-6}} \cdot
\label{Eqn:c1c2est}
}
The Young inequality $ab \leq \frac{a^{2}}{2} + \frac{b^{2}}{2}$ and the above estimates show that the LHS of (\ref{Eqn:Estbarv})
is bounded by the RHS of (\ref{Eqn:Estbarv}).

\end{proof}

We then prove the proposition below:

\begin{prop}
The following hold:

\begin{enumerate}

\item Let $k' \geq 3$ and $j  \in \{ 1, 2 \}$. Let $u_j$ be a solution of (\ref{Eqn:Integral}) with $\epsilon:= \epsilon_{j}$ on an interval $[0,T']$. Let $K \geq 0$.
Let $ m \in \left\{ \frac{9}{2}+,  k'-3 \right\}$. Assume that $\sup_{t \in [0,T']} \| u_j (t) \|_{H^{m+3}}  \leq K $. Let $\bar{u} := u_1 - u_2$. Then
there exists $C:= C(K) > 0$ such that

\EQQARRLAB{s
\frac{d \left( \| \bar{u}(t) \|^{2}_{H^{m}} \right) }{dt}  & \leq C \left( \| \bar{u}(t) \|^{2}_{H^{m+3}} +  \epsilon_2^{2} \| u_2(t) \|^{2}_{H^{m+4}} \right)  \cdot
\label{Eqn:EstbaruL2}
}

\item Let $j \in \{ 1,2 \}$. Let $0 < \epsilon_{1}  \leq \epsilon_2 \ll 1$. Assume that $v_j$ is a function that satisfies on an interval $[0,T']$

\EQQARRLAB{
\p_{t} v_j + \epsilon_{j} \p_{x}^{4} v_j & = a_{3,j} \p_{x}^{3} v_j + a_{2,j} \p_{x}^{2} v_j + a_{1,j} \p_{x} v_j +  a_{0,j} \\
& + \epsilon_j b_{3,j} \p_{x}^{3} v_j + \epsilon_j b_{2,j} \p_{x}^{2} v_j + \epsilon_j b_{1,j} \p_x v_j + \epsilon_{j} b_{0,j} + \epsilon_{j} \p_{x}^{4} c_j v_j \cdot
\label{Eqn:EqvjL2}
}
Let $\bar{v} := v_1 - v_2$. Then there exists $C > 0$ such that for all $t \in [0,T']$

\EQQARRLAB{
\frac{d \left( \| \bar{v}(t) \|^{2}_{L^{2}} \right) }{dt} + 2 \epsilon_1 \| \p_{x}^{2} \bar{v}(t) \|^{2}_{L^{2}}  \\
\leq  C
\left(
\begin{array}{l}
\left( 1 + \sum \limits_{m=2}^{3} \| a_{m,1}(t) \|_{H^{ \left( m -  \frac{3}{2} \right)+}} \right) \| \bar{v}(t) \|^{2}_{H^{1}} \\
+ \left( 1 + \sum \limits_{m=1}^{3} \| a_{m,1}(t) \|^{2}_{H^{\left( m + \frac{1}{2} \right)+}} \right)  \| \bar{v}(t) \|^{2}_{L^{2}} + \sum \limits_{m=1}^{3} \| a_{m,1}(t) \|^{2}_{H^{m}} \| \bar{v}(t) \|^{2}_{H^{\frac{1}{2}+}} \\
+ \epsilon_1  \left( \| b_{3,1}(t)\|_{H^{\frac{3}{2}+}} + \| b_{2,1}(t) \|_{H^{\frac{1}{2}+}} \right) \| \p_{x}  \bar{v}(t) \|^{2}_{L^{2}} \\
+ \epsilon_{1}^{2} \| b_{0,1}(t) \|^{2}_{L^{2}} + \epsilon_{2}^{2} \| b_{0,2}(t) \|^{2}_{L^{2}}
+ \epsilon_1^{2} \left( \sum \limits_{m=1}^{3} \| b_{m,1}(t) \|^{2}_{H^{m}}
+ \| c_1(t) \|^{2}_{H^{4}} \right) \| \bar{v}(t) \|^{2}_{H^{\frac{1}{2}+}} \\
+ \epsilon_{1}^{2} \left( \| c_2(t) - c_1(t) \|^{2}_{H^{4}} \| v_2(t) \|^{2}_{H^{\frac{1}{2}+}} + \| c_2(t) - c_1(t) \|^{2}_{H^{\frac{9}{2}+}}
\| v_2(t) \|^{2}_{H^{4}} \right) \\
+ \epsilon_1  \left( \sum \limits_{m=1}^{3} \| b_{m,1}(t) \|_{H^{ \left( m + \frac{1}{2} \right)+}} + \| c_1(t) \|_{H^{\frac{9}{2}+}}  \right)
\|\bar{v}(t) \|^{2}_{L^{2}} \\
+ \epsilon_2^{2}
\left(
\begin{array}{l}
\sum \limits_{j=1}^{2} \sum \limits_{m=1}^{3}
\left(
\| b_{m,j}(t) \|^{2}_{L^{2}} \| v_2(t) \|^{2}_{H^{ \left(m + \frac{1}{2} \right)+}} + \| b_{m,j}(t) \|^{2}_{H^{\frac{1}{2}+}} \| v_2(t) \|^{2}_{H^{m}}
\right) \\
+ \| c_2(t) \|^{2}_{H^{4}} \| v_2(t) \|^{2}_{H^{\frac{1}{2}+}}  + \| c_2(t) \|^{2}_{H^{\frac{9}{2}+}} \| v_2(t) \|^{2}_{L^{2}} + \| v_2(t) \|^{2}_{H^{4}}
\end{array}
\right) \\
+ \| a_{0,1}(t) - a_{0,2}(t) \|^{2}_{L^{2}} \\
+ \sum \limits_{m=1}^{3}
\left(
\begin{array}{l}
\| a_{m,1}(t) - a_{m,2}(t) \|^{2}_{L^{2}} \| v_2(t) \|^{2}_{H^{ \left( m+ \frac{1}{2} \right)+}} \\
+ \| a_{m,1}(t) - a_{m,2}(t) \|^{2}_{H^{\frac{1}{2}+}} \| v_{2}(t) \|^{2}_{H^{m}}
\end{array}
\right)

\end{array}
\right)
\label{Eqn:EstbarvL2}
}

\end{enumerate}
\label{prop:linearbarvL2}

\end{prop}

\begin{rem}
The proof shows that the constant $C$ depending on $K$ in Proposition \ref{prop:linearbarvL2} can be chosen as a function of $K$ that increases as
$K$ increases.
\label{rem:linearbarvL2}
\end{rem}


\begin{proof}
Let $\bar{u} := u_1 - u_2$. Let $m \in \left\{ 0, \frac{9}{2}+, k'-3 \right\}$. We have

\EQQARR{
\p_{t} \bar{u} + \epsilon_1 \p_{x}^{4} \bar{u} & = (\epsilon_2 - \epsilon_1) \p_{x}^{4} u_{2}
+ F (\overrightarrow{u_1(t)},t) - F (\overrightarrow{u_2(t)},t)
}
Elementary considerations show that

\EQQARR{
\langle \p_{t} \bar{u}(t) + \epsilon_1 \p_{x}^{4} \bar{u}(t), D^{2m} \bar{u}(t) \rangle & =
\frac{(-1)^{m}}{2} \frac{d}{dt} \left( \| D^{m} \bar{u}(t) \|^{2}_{L^{2}}  \right)
+ \frac{(-1)^{m}}{2} \epsilon_1 \frac{d}{dt} \left(  \| D^{m+2} \bar{u}(t) \|^{2}_{L^{2}} \right) \cdot
}
On the other hand the Young inequality $ab \leq \frac{a^{2}}{2} + \frac{b^{2}}{2}$ yields

\EQQARR{
(\epsilon_2 - \epsilon_1)  \left| \langle \p_{x}^{4} u_{2}(t), D^{2m} \bar{u}(t) \rangle \right|
& = (\epsilon_2 - \epsilon_1) \left| (-1)^{m} \langle D^{4+m} u_2(t), D^{m} \bar{u}(t) \rangle \right| \\
& \lesssim (\epsilon_2 - \epsilon_1) \| u_2(t) \|_{H^{m+4}}  \|  \bar{u}(t) \|_{H^{m}} \\
& \lesssim  \epsilon_2^{2} \| u_2(t) \|^{2}_{H^{m+4}} + \| \bar{u}(t) \|^{2}_{H^{m}}, \, \text{and}
}
$\langle F(\overrightarrow{u_1(t)},t) - F(\overrightarrow{u_2(t)},t), D^{2m} \bar{u}(t) \rangle
= (-1)^{m} \langle \, D^{m} ( F(\overrightarrow{u_1(t)},t) - F(\overrightarrow{u_2(t)},t) ), D^{m} \bar{u}(t) \, \rangle$. Hence
we see from Lemma \ref{lem:EstDerivF} that $ \frac{d \left( \| \bar{u}(t) \|^{2}_{\dot{H}^{m}} \right)}{dt} \lesssim \| \bar{u}(t) \|^{2}_{H^{m+3}}
+ \epsilon_{2}^{2}  \left\| u_{2}(t) \right\|^{2}_{H^{m+4}} $ if
$m \in \left\{ \frac{9}{2}+, k'-3 \right\}$ and $ \frac{ d \left(  \| \bar{u}(t) \|^{2}_{L^{2}} \right)} {dt} \lesssim
\| \bar{u}(t) \|^{2}_{H^{\frac{7}{2}+}} + \epsilon_{2}^{2} \| \bar{u}(t) \|^{2}_{H^{4}} $.
Hence (\ref{Eqn:EstbaruL2}) holds.\\
The Sobolev embedding $ H^{\frac{1}{2}+} \hookrightarrow L^{\infty} $ shows that

\EQQARR{
 \left| 2 \int \left( - \frac{3}{2} \p_{x} a_{3,1}(t) + a_{2,1}(t) \right) ( \p_{x} \bar{v}(t) )^{2} \; dx \right|
 & \lesssim \left( \| \p_{x} a_{3,1}(t) \|_{L^{\infty}} + \| a_{2,1}(t) \|_{L^{\infty}} \right) \| \p_{x} \bar{v}(t) \|^{2}_{L^{2}}  \\
 & \lesssim  \left( \| a_{3,1}(t) \|_{H^{\frac{3}{2}+}} + \| a_{2,1}(t) \|_{H^{\frac{1}{2}+}} \right)  \| \p_{x} \bar{v}(t) \|^{2}_{L^{2}}
\label{Eqn:BoundResL2}
}
We now mimick the proof of Proposition \ref{prop:linearbarv}, replacing $k'$ with $6$. We get (\ref{Eqn:EstbarvL2}).

\end{proof}

\subsection{Gauged energy estimates for difference of solutions: statement}

We prove the following proposition:

\begin{prop}

Let $k' > \frac{19}{2}$, $ k' \geq k'_{0} > \frac{17}{2} $, and $ \tilde{\phi} \in \widetilde{\mathcal{P}}_{+,k'} $.
Let $q \in \{ 1,2 \}$. There exists $ 0 < c := c \left( \tilde{\delta}(\overrightarrow{\tilde{\phi}}), \tilde{\delta}^{'}(\tilde{\phi}), \| \tilde{\phi} \|_{H^{k'}} \right) \ll 1  $ such that
if $ 0 < \epsilon_{1} \leq \epsilon_{2} \leq c $, then the following holds: if $u_{q}$ is the solution of
(\ref{Eqn:Integral}) obtained by Proposition \ref{Prop:Integral}  with $ \epsilon  := \epsilon_{q} $, replacing $\tilde{\phi}$  with
$\tilde{\phi}_{q, k'} :=  J_{\epsilon_{q},k'} \tilde{\phi} $ on $[0, T_{\epsilon_{q}}]$, then there exist $1 \geq
T' := T' \left( \| \tilde{\phi} \|_{H^{k'_{0}}}, \tilde{\delta}(\overrightarrow{\tilde{\phi}}), \tilde{\delta}^{'}(\tilde{\phi}) \right) > 0 $ and
$ C := C \left( \| \tilde{\phi} \|_{H^{k'}}, \tilde{\delta}(\overrightarrow{\tilde{\phi}}) \right) > 0 $ such that
\EQQARRLAB{
T_{\epsilon_{1}}, \, T_{\epsilon_{2}} > T'; \; (\ref{Eqn:BoundNrjVisc}) \; \text{holds, replacing} \; u \; \text{with} \; u_{q}; \;  \text{and} \\
\sup_{t \in [0,T']} \left( E_{k'}  \left( u_{1}(t), u_{2}(t),t \right) + \int_{0}^{t} \int_{\T} Q ( u_{1}(t'),t')
\left( \partial_{x} D^{k'-6}  \bar{v}(t') \right)^{2} \; dx \; dt' \right)   \\
\leq C \left(  \| u_{2}(0) - u_{1}(0) \|^{2}_{H^{k'}} + \epsilon_{2}^{0+} \right),
\label{Eqn:DiffEstNrj}
}
with $\bar{v}(t) := \Phi_{k'}(u_{1}(t),t) \partial_{x}^{6} u_{1}(t) - \Phi_{k'}(u_{2}(t),t) \partial_{x}^{6} u_{2}(t) $.

\label{prop:EnergyDiffEst}
\end{prop}

\begin{rem}
The proof of Proposition \ref{prop:EnergyDiffEst} shows that one can choose $T^{'}$  (resp. $C$) to be a continuous function that decreases
(resp. increases) as $\tilde{\delta}(\overrightarrow{\tilde{\phi}})$ decreases, and that decreases (resp. increases) as $\| \tilde{\phi} \|_{H^{k'_{0}}}$ increases.
It also shows that one can choose $T'$ as a continuous function that decreases as $\tilde{\delta}^{'}(\tilde{\phi})$ decreases.
\label{Rem:PropEnergyDiffEst}
\end{rem}

\begin{proof}

The proof relies upon the following lemma:

\begin{lem}
Let  $0 < \epsilon_{1} \leq \epsilon_{2} \ll 1 $. Let $K \geq 0$ and  $\bar{\delta} > 0$.
Consider the solutions $u_{q}$ on an interval $[0, T^{''}]$ with $ 0 < T^{''} \leq  1$. We define
$h_{\theta}(t):= \theta u_{1}(t) + (1- \theta) u_{2}(t) $ for $\theta \in [0,1]$ and for $t \in [0,T^{''}]$. Assume that \\
$t \in [0,T^{''}]: \; \max \left( \sup_{t \in [0,T^{''}]} E_{k'} (u_{1}(t),t),  \sup_{t \in [0,T^{''}]} E_{k'} (u_{2}(t),t) \right) \leq K $ . Assume also that
$ \delta \left( \overrightarrow{h_{\theta}(t)}, t \right) \geq \bar{\delta} $ for all
$(t, \theta) \in [0,T^{''}] \times [0,1] $. Then there exists $ C := C \left( \bar{\delta}, K  \right) > 0  $ such that for all $ t \in [0,T^{''}] $

\EQQARRLAB{
\frac{d E_{k'} (u_1(t),u_2(t),t)}{dt} + \int_{\T} Q (u_{1}(t),t) \left( \partial_{x} D^{k'-6} \bar{v}(t) \right)^{2} \; dx \; \leq C  \\
\\
\left(
\begin{array}{l}
E_{k'} \left( u_{1}(t), u_{2}(t),t \right) +   \left\| u_{1}(t)-  u_{2}(t) \right\|^{2}_{H^{\frac{7}{2}+}}  E_{k'+3} \left( u_{2}(t), t \right)
+ \left\| u_{1}(t) - u_{2}(t) \right\|^{2}_{H^{\frac{9}{2}+}} E_{k'+2} \left( u_{2}(t),t \right) \\
\\
+ \left\| u_{1}(t) -u_{2}(t) \right\|^{2}_{H^{\frac{13}{2}+}} E_{k'+1} \left(  u_{2} (t),t \right)
+ \left\| u_{1}(t)- u_{2}(t) \right\|^{2}_{H^{k'}} E_{\frac{19}{2}+} \left(  u_{2}(t),t \right) \\
\\
+ \epsilon_{2}^{2} \left( 1 +  E_{k'+4} \left( u_{2}(t),t \right) \right) + \epsilon_{1}^{2}
\left( 1 +
E_{k' + 2} \left( u_{1}(t),t \right) + E_{k'+4} \left( u_{2}(t),t \right)
\right)  \\
\\
+ \epsilon_{1}^{2} \left\| u_{1}(t) - u_{2}(t) \right\|^{2}_{H^{\frac{13}{2}+}} \left( 1 + E_{k'+2} \left( u_{1}(t),t \right) \right)
\end{array}
\right)
\cdots
\label{Eqn:EstNrjDiff}
}

\label{lem:EnergyDiffEst}
\end{lem}

\begin{rem}
The proof of Lemma \ref{lem:EnergyDiffEst} shows that one can choose $C$ to be a continuous function that increases as
$\bar{\delta}$ decreases, and increases as $K$ increases.
\label{Rem:LemEnergyDiffEst}
\end{rem}

We postpone the proof of Lemma \ref{lem:EnergyDiffEst} to Section \ref{Sec:ProofLemmas}. \\
\\
Let $ 0 < c := c \left( \tilde{\delta}(\overrightarrow{\tilde{\phi}}), \tilde{\delta}^{'}(\tilde{\phi}), \| \tilde{\phi} \|_{H^{k'}} \right) \ll 1 $ such that if $ 0 < \epsilon_{1} \leq \epsilon_{2} \leq c $, then the estimates  below hold. \\
We see from Appendix  and Lemma \ref{Lem:MollEst} that \\
$ \left| X ( \tilde{\phi}_{q,k'}, \tilde{\phi} ,0,0) \right| \lesssim \| \tilde{\phi}_{q, k'} - \tilde{\phi} \|_{H^{\frac{7}{2}+}} \lesssim
\epsilon_{q}^{ \frac{k' - \left( \frac{7}{2}+ \right) }{k'}} \| \tilde{\phi} \|_{H^{k'}}  $.
From Remark \ref{Rem:Qother} we see that $\partial_{\omega_{3}} F \left( \overrightarrow{\tilde{\phi}},0 \right)$ has a constant sign. Hence
$Q (\tilde{\phi},0) =   \left| \partial_{\omega_{3}} F ( \overrightarrow{\tilde{\phi}},0 ) \right|
\left[ \frac{P( \overrightarrow{\tilde{\phi}},0 )}{ \left| \partial_{\omega_{3}} F (\overrightarrow{\tilde{\phi}} ,0 ) \right|}  \right]_{ave} $.
We see from Appendix and the above estimates that  $ \left| Y( \tilde{\phi}_{q,k'}, \tilde{\phi} ,0,0) \right|
\lesssim_{\tilde{\delta}(\tilde{\phi}), \| \tilde{\phi} \|_{H^{k_{0}^{'}}} }
\| \tilde{\phi}_{q, k'} - \tilde{\phi} \|_{H^{\frac{9}{2}+}} \lesssim_{\tilde{\delta}(\tilde{\phi}),\| \tilde{\phi} \|_{H^{k_{0}^{'}}}} \epsilon_{q}^{0+} \| \tilde{\phi} \|_{H^{k'}} $. Hence the triangle inequality shows that

\EQQARRLAB{
\tilde{\delta}^{'} \left( \tilde{\phi}_{q,k'}  \right) \geq \frac{3 \tilde{\delta}^{'}(\tilde{\phi})}{4}, \;
\tilde{\delta} \left( \overrightarrow{ \tilde{\phi}_{q,k'} } \right)  \geq \frac{3 \tilde{\delta}(\overrightarrow{\tilde{\phi}})}{4} \; \text{and} \;
\tilde{\phi}_{q,k'} \in \widetilde{\mathcal{P}}_{+,k'} \cdot
\label{Eqn:EstPFInit}
}
Let $  m \geq  k_{0}^{'} > \frac{15}{2}$. Proposition \ref{Prop:OneEnergyEst} and Proposition \ref{prop:compEHnorm} show that there exists $1 \geq T'
:= T' \left( \tilde{\delta}(\overrightarrow{\tilde{\phi}}), \tilde{\delta}^{'}(\tilde{\phi}), \| \tilde{\phi} \|_{H^{k'_{0}}} \right) > 0$ such that $ T_{\epsilon_{q}} > T' $ and $ C := C \left( \tilde{\delta}(\overrightarrow{\tilde{\phi}}), \| \tilde{\phi} \|_{H^{k'_{0}}} \right) > 0 $ such that (\ref{Eqn:BoundNrjVisc}) holds, replacing $T_{\epsilon}$, $u$ with $T_{\epsilon_{q}}$, $u_{q}$ respectively and such that

\EQQARRLAB{
\sup_{t \in [0,T']}   \| u_{q}(t) \|^{2}_{H^{m}} & \lesssim 1 +
\sup_{t \in [0,T']}  E_{m}(u_{q}(t) , t) \\
& \lesssim  1 + E_{m} \left( u_{q} (0), 0 \right) \\
& \lesssim 1 + \| \tilde{\phi} \|^{2}_{H^{m}}   \cdot
\label{Eqn:UpBoundEkpri}
}
If it is necessary then we may WLOG reduce the size of $T'$ (with $T'$ still a function of $\delta(\overrightarrow{\tilde{\phi}})$, $\delta^{'}(\tilde{\phi})$, and $\| \tilde{\phi} \|_{H^{k_{0}'}}$) so that the estimates below hold. We see from Appendix that  for $t \in [0,T']$

\EQQARR{
\left\| X \left( h_{\theta}(t),u_{1}(t) ,t,t \right) \right\|_{L^{\infty}} & \lesssim_{\| \tilde{\phi} \|_{H^{k_{0}^{'}}}} \| h_{\theta}(t) - u_{1}(t) \|_{H^{\frac{7}{2}+}} \\
& \lesssim_{\| \tilde{\phi} \|_{H^{k_{0}^{'}}}}  \| u_{2}(t) - u_{1}(t) \|_{H^{\frac{7}{2}+}} \\
& \ll \tilde{\delta}(\overrightarrow{\tilde{\phi}}),
\label{Eqn:Omega3FDiff}
}
where at the last line we use $ u_{2}(t) - u_{1} (t) =  u_{2}(t) - \tilde{\phi}_{2,k'} + \tilde{\phi}_{2,k'} - \tilde{\phi}_{1,k'}
+ \tilde{\phi}_{1,k'} -  u_{1}(t) $, the estimates below that can be derived from the arguments in the proof of Proposition  \ref{Prop:OneEnergyEst} (see the text below
` \underline{Claim}: $\delta ( \overrightarrow{u(t)}, t ) \gtrsim \tilde{\delta} ( \overrightarrow{\tilde{\phi}} ) $ '), taking into account (\ref{Eqn:UpBoundEkpri}):

\EQQARRLAB{
 \| u_{q}(t) - \tilde{\phi}_{q,k'} \|_{H^{\frac{9}{2}+}}  & \lesssim_{ \| \tilde{\phi} \|_{H^{k_{0}^{'}}}} \epsilon_{q}^{0+} t^{0+} + t,
 \label{Eqn:uqminJ}
}
and the following estimate  that follows from Lemma \ref{Lem:MollEst}:

\EQQARR{
\| \tilde{\phi}_{2,k'} - \tilde{\phi}_{1,k'} \|_{H^{\frac{7}{2}+}} & \lesssim
\| \tilde{\phi}_{2,k'} - \tilde{\phi} \|_{H^{\frac{7}{2}+}}  + \| \tilde{\phi}_{1,k'} - \tilde{\phi} \|_{H^{\frac{7}{2}+}} \lesssim
\epsilon_{2}^{\frac{k'-  \left( \frac{7}{2}+ \right)}{k'}} \| \tilde{\phi} \|_{H^{k'}} \cdot
}
Hence the assumptions of Lemma \ref{lem:EnergyDiffEst} hold. \\
\\
We then prove that  $ u_{q}(t) \in \mathcal{P}_{+,k'}(t)$, $t \in [0,T']$. Indeed by following the arguments to prove the claim `$ \delta( \overrightarrow{u(t)},t) \gtrsim \tilde{\delta} ( \overrightarrow{\tilde{\phi}} )$, $t \in [0,T']$ ' and those below `\underline{Claim}:  $u(t) \in \mathcal{P}_{+,k'}(t)$ and $ \delta^{'}(u(t),t) \gtrsim \tilde{\delta}^{'} (\tilde{\phi}), t \in [0,T'] $ ' in the proof
of Proposition \ref{Prop:OneEnergyEst}, and by (\ref{Eqn:uqminJ}), we see that the claim holds.   \\
\\
Let $ m \in \mathbb{N}$. We get from Lemma \ref{Lem:MollEst} and Proposition \ref{prop:compEHnorm}

\EQQARRLAB{
E_{k'+m} \left( u_{q}(0),0 \right)  \lesssim 1 + \| u_{q}(0) \|^{2}_{H^{k'+m}}   \lesssim  \epsilon_{q}^{-\frac{2m}{k'}}
\left( 1 + \| \tilde{\phi} \|^{2}_{H^{k'}}  \right) \cdot
\label{Eqn:EstNrjInit}
}
If $ 0 \leq  m \leq k' $ then

\EQQARRLAB{
E_{k'-m} \left( u_{1}(0), u_{2}(0), 0 \right) & \lesssim \| u_{1}(0) - u_{2}(0) \|^{2}_{H^{k'-m}}  \lesssim \| u_{1}(0) - \tilde{\phi} \|^{2}_{H^{k'-m}} + \| u_{2}(0) - \tilde{\phi} \|^{2}_{H^{k'-m}}  \\
& \lesssim \epsilon_{2}^{\frac{2m}{k'}} \| \tilde{\phi} \|^{2}_{H^{k'}}  \cdot
\label{Eqn:EstDiffNrjInit2}
}
Let $t \in [0,T']$. In the sequel we use Lemma \ref{lem:EnergyDiffEst} combined with Proposition \ref{prop:compEHnorm}, Proposition \ref{Prop:OneEnergyEst}, (\ref{Eqn:EstNrjInit}), and
(\ref{Eqn:EstDiffNrjInit2}). We have

\EQQARRLAB{
\begin{array}{l}
\frac{d E_{(k'-3)-} \left( u_{1}(t),u_{2}(t), t \right)}{dt}  \\
\lesssim E_{(k'-3)-} \left( u_{1}(t), u_{2}(t), t \right)
\left( 1 +  E_{k'} \left( u_{2}(0) , 0 \right) \right)  + \epsilon_{2}^{2} \left( 1  + E_{k'+1} \left( u_{2}(0), 0 \right) \right) \\
+ \epsilon_{1}^{2} \left(1 +  E_{k'} \left( u_{1}(0),0 \right) + E_{k'+1} \left( u_{2}(0),0 \right)  \right)
+ \epsilon_{1}^{2} \left( 1 + E_{k'} \left( u_{1}(0),0 \right) \right) E_{(k'-3)-} \left( u_{1}(t),u_{2}(t), t \right) \\
\lesssim E_{(k'-3)-} \left( u_{1}(t),u_{2}(t),t \right) + \epsilon_{2}^{\frac{2(k'-1)}{k'}} \cdot
\end{array}
}
The Gronwall inequality yields

\EQQARRLAB{
 E_{(k'-3)-} \left( u_{1}(t),u_{2}(t),t \right)  \lesssim   E_{(k'-3)-} \left( u_{1}(0),u_{2}(0),0 \right) + \epsilon_{2}^{\frac{2(k'-1)}{k'}} \lesssim
 \epsilon_{2}^{\frac{6}{k'}+} \cdot
\label{Eqn:Ekminthree}
}
Hence

\begin{equation}
\begin{array}{l}
\frac{d E_{k'} (u_1(t),u_2(t), t )}{dt} + \int_{\T} Q (u_{1}(t),t) \left( \partial_{x} D^{k'-6} \bar{v}(t) \right)^{2} \; dx  \\
\\
\lesssim  E_{k'} \left( u_{1}(t), u_{2}(t),t \right) + \epsilon_{2}^{\frac{6}{k'}+}  \left( 1 + E_{k'+3} ( u_{2}(0), 0) \right)   \\
+ \epsilon_{2}^{2} \left( 1 + E_{k'+4} \left( u_{2}(0), 0 \right) \right)
+ \epsilon_{1}^{2} \left( 1 + E_{k'+2} \left( u_{1}(0) , 0  \right) + E_{k'+4} \left( u_{2}(0), 0 \right) \right) \\
+ \epsilon_{1}^{2} \left( 1 + E_{k'+2} \left( u_{1}(0), 0 \right) \right) E_{k'} \left( u_{1}(t),u_{2}(t),t \right) \\
\\
\lesssim E_{k'} \left( u_{1}(t), u_{2}(t), t \right) + \epsilon_{2}^{0+} \cdot
\end{array}
\nonumber
\end{equation}
Hence (\ref{Eqn:DiffEstNrj}) holds.

\end{proof}

\section{Proof of Theorem}
\label{Sec:ProofThm}

In this section $\alpha$ denotes a function of which the values is allowed to change from one line to the other one and such that $ \lim \limits_{x \rightarrow 0} \alpha ( x ) = 0 $. Let $m \in \mathbb{N}^{*}$. \\
\\
First we prove the local existence and the parabolic smoothing effect. \\
Let $ T := T \left( \tilde{\delta} (\vec{\phi}), \tilde{\delta}^{'}(\phi), \| \phi \|_{H^{k_{0}}} \right) > 0 $ be
small enough such that all the statements below are true. Let $ u_{m} $ be the solution of (\ref{Eqn:Integral})  with $ \epsilon := \frac{1}{m} $, replacing $\phi$ with
$ \phi_{m} :=  J_{\frac{1}{m},k} \phi $ on $ \left[ 0, T_{\frac{1}{m}} \right]$. The dominated convergence shows that
$\lim \limits_{m \rightarrow \infty} \|  \phi_{m} - \phi \|_{H^{k}} = 0 $. Hence we see from Proposition \ref{prop:EnergyDiffEst}, Proposition \ref{prop:compEHnorm},
and the triangle inequality that for $m$ large and for $ p \in \mathbb{N}$,  $ T_{\frac{1}{m}}, T_{\frac{1}{m+p}} > T  $ and

\EQQARRLAB{
t \in [0,T]: \; u_{m+p}(t) \in \mathcal{P}_{+,k}(t), \; \delta(\overrightarrow{u_{m+p}(t)}, t) \gtrsim \tilde{\delta}(\vec{\phi}), \;  \delta^{'}(u_{m+p}(t),t) \gtrsim \tilde{\delta}^{'}(\phi), \; \text{and} \\
\| u_{m+p} - u_{m} \|_{L_{t}^{\infty} H^{k} ([0,T])}  \lesssim  \left( \frac{1}{m} \right)^{0+} + \alpha \left( \frac{1}{m}  \right) \cdot
\label{Eqn:umpDiffEst}
}
Hence $\{ u_{m} \}$ is a Cauchy sequence so there exists $ u \in \mathcal{C} \left( [0,T], H^{k} \right) $ such that $u_{m} \rightarrow u $.
as $ m \rightarrow \infty $. Letting $p \rightarrow \infty$ and choosing $m$ large enough in (\ref{Eqn:umpDiffEst})
we see from Appendix and Remark  \ref{Rem:Qother} with $f:= u_{m}(t)$ and $ g := u$  that $ \delta(\overrightarrow{u(t)}, t) \gtrsim \tilde{\delta}(\vec{\phi}) $, $ \delta^{'}(u(t),t)  \gtrsim \tilde{\delta}^{'}(\phi) $, and $u(t) \in \mathcal{P}_{+,k}(t)$ for $ t \in [0,T] $. \\\
We then claim that $  u(t) = u_{0} + \int_{0}^{t} F \left(  \overrightarrow{u(t')}, t' \right) \; dt' $ \footnote{Observe from Lemma \ref{lem:EstDerivF}
that $ \int_{0}^{T} \|  F ( \overrightarrow{v(t')}, t') \|_{H^{k-3}} \; d t' < \infty $, so the integral makes sense as a Bochner integral}. To this end we write  $u_{m} = u_{l,m} + u_{nl,m}$ with $u_{l,m} (t) := e^{- \frac{1}{m} t \p_{x}^{4}} \phi_{m} $ and $ u_{nl,m}(t) := - \int_{0}^{t} e^{ - \frac{1}{m} (t - t') \p_{x}^{4} } F \left( \overrightarrow{u_{m}(t')}, t' \right) \; d t' $ for $t \in [0,T]$. Indeed

\EQQARR{
\| u_{l,m} (t) - \phi \|_{H^{k}} & \lesssim X_{1} + X_{2},
}
with $ X_{1} :=  \left\|  ( e^{- \frac{1}{m} t \p_{x}^{4}} - 1 ) \phi \right\|_{H^{k}} $,  and
$ X_{2} := \| \phi_{m} - \phi \|_{H^{k}}$ . Clearly $ X_{2} \rightarrow 0 $ as $ m \rightarrow \infty$.
The dominated convergence theorem shows that $ X_{1} \rightarrow 0 $ as $ m \rightarrow \infty$.
Hence $ \| u_{l,m} (t) - \phi \|_{H^{k}} \rightarrow 0 $ as $ m \rightarrow \infty $. The Minkowski inequality shows that

\EQQARR{
\left\| u_{nl,m} (t) - \int_{0}^{t} F ( \overrightarrow{u(t')}, t' ) \; d t' \right\|_{H^{k-3}} & \lesssim Y_{1} + Y_{2},
}
with $ Y_{1} :=  \int_{0}^{t}  \left\| ( e^{- \frac{1}{m} (t-t') \p_{x}^{4}} - 1 ) F \left( \overrightarrow{u(t')} ,t' \right) \right\|_{H^{k-3}} \; d t' $ and \\
$ Y_{2} := \int_{0}^{t} \left\| F ( \overrightarrow{u_{m}(t')}, t') -  F \left( \overrightarrow{u(t')}, t' \right) \right\|_{H^{k-3}} \; dt' $. Lemma \ref{lem:EstDerivF} and
the dominated convergence theorem show that $Y_{1} \rightarrow 0$ as $ m \rightarrow \infty$. We also have $Y_{2} \lesssim \sup_{t' \in [0,T]}
\left\|  u(t') - u_{m}(t') \right\|_{H^{k}} $. Hence $Y_{2} \rightarrow 0 $ as $ m \rightarrow \infty $. \\
We then prove the parabolic smoothing effect. Let $ T > \nu > 0 $. We see from Proposition \ref{prop:EnergyDiffEst} that

\EQQARR{
\int_{0}^{\nu} \int_{\T} Q(u_{m}(t),t) \left( \partial_{x} D^{k-6} \left( \Phi_{k} \left( u_{m}(t), t \right) \partial_{x}^{6}  u_{m}(t)
-  \Phi_{k} \left(  u_{m+p}(t),t \right)  \partial_{x}^{6} u_{m+p}(t) \right)  \right)^{2} \; dx \; dt \\
\lesssim \left( \frac{1}{m} \right)^{0+} + \alpha \left( \frac{1}{m} \right) \cdot
}
We see from Appendix and Remark \ref{Rem:Qother}  with $f : =u_{m}(t)$ and $ g:= u(t)$  and from (\ref{Eqn:umpDiffEst}) that there exists a positive constant $c := c \left( \tilde{\delta}(\vec{\phi}), \tilde{\delta}^{'}(\phi) \right)$  such that $ Q(u_{m}(t),t) \geq c $. Hence, using also Proposition \ref{prop:compEHnorm} we see that $ \{ u_{m} \} $ is a Cauchy sequence in
$L_{t}^{2} H^{k+1} ([0,\nu]) $ so there exists $w \in L_{t}^{2} H^{k+1} ([0,\nu]) \subset L_{t}^{2} H^{k} ([0,\nu])  $ such that $u_{m} \rightarrow w $ as $m \rightarrow \infty$. Since it also converges in $ L_{t}^{\infty} H^{k} ([0,\nu]) \subset L_{t}^{2} H^{k} ([0,\nu])$, $ w = u $. Hence we can find $ \bar{t}_{1} \in [0, \nu]$
such that $\| u(\bar{t}_{1}) \|_{H^{k+1}} < \infty $. Recall that there exists $c > 0$ such that  $ \delta \left( \overrightarrow{u(\bar{t}_{1})},\bar{t}_{1} \right) \geq c  \tilde{\delta}(\vec{\phi}) $ and
$ \delta' \left( u(\bar{t}_{1}), \bar{t}_{1} \right) \geq c  \tilde{\delta}^{'}(\phi) $. Hence by adapting slightly the arguments in the proof of the local existence part with the new data $u(\bar{t}_{1})$, we see that $ u \in \mathcal{C} \left( [\bar{t}_{1}, T], H^{k+1} \right) $.
We then apply the procedure again to find $\bar{t}_{2}$ such that $ \bar{t}_{1} \leq \bar{t}_{2} \leq \nu  $ and such that
$ u \in \mathcal{C} \left( [\bar{t}_{2}, T], H^{k+2} \right) $, taking into account that
$ \delta \left( \overrightarrow{u(\bar{t}_{2})}, \bar{t}_{2} \right) \geq  c  \tilde{\delta}(\vec{\phi}) $  and
$ \delta'\left( u(\bar{t}_{2}), \bar{t}_{2} \right) \geq
c  \tilde{\delta}^{'}(\phi)$. Iterating $m-$ times there exists $ 0 \leq  \bar{t}_{m} \leq  \nu $ such that
$ u \in \mathcal{C} \left( [\bar{t}_{m}, T], H^{k+m} \right) $. So
$ u \in \mathcal{C} \left( [\nu, T], H^{p} \right)$ for all $p \in \mathbb{N}$. A bootstrap argument (that also uses
Lemma \ref{lem:EstDerivF}) applied to $ \partial_{t} u = F(\vec{u},t) $  shows that $u \in \mathcal{C}^{\infty} \left(
[\nu, T] , H^{p} \right) $ for all $p \in \mathbb{N}$. Hence $ u \in \mathcal{C}^{\infty} \left( [\nu, T] \times \T \right)$. In other words
$ u \in \mathcal{C}^{\infty} \left( (0, T] \times \T \right)$ . \\
\\
We then prove the uniqueness. \\
In the proof we may assume WLOG that $ u_{q}(t) \in \mathcal{P}_{+,k}(t)$ for $t \in [0,\breve{T}] $. Let $ T^{*} :=
\max \left\{ t \in [0,\breve{T}]: \; u_{1}(t) = u_{2}(t) \right\}$. Assume that $T^{*} < \breve{T} $. We have
$u_{1} (T^{*}) = u_{2} (T^{*})$. Let $K \in \mathbb{R}$ be such that
$ \| u_{1}(t) \|_{H^{k}} + \| u_{2}(t) \|_{H^{k}} \leq K $ for all $t \in [0, \breve{T}]$. Recall that

\EQQARR{
u_{q}(t)  = u_{q}(T^{*}) + \int_{T^{*}}^{t}  F \left( \overrightarrow{u_{q}(t')}, t' \right) \; dt' \cdot
}
Then we claim that we can find $ \breve{T} - T^{*} > T^{'} := T^{'} \left( \delta \left( \overrightarrow{u_{q}(T^{*})}, T^{*} \right),
\delta^{'} \left( u_{q}(T^{*}), T^{*} \right), \| u(T^{*}) \|_{H^{k'_{0}}} \right) > 0 $ small enough such that
$\delta \left( \overrightarrow{h_{\theta}(t)}, t \right) \gtrsim \delta \left( \overrightarrow{u_{q}(T^{*})}, T^{*} \right) $ for $ t \in [ T^{*}, T^{*} + T^{'} ] $ with
$ h_{\theta}(t) :=  \theta u_{1}(t) + (1 - \theta) u_{2}(t)$. Indeed we see from
Appendix, the integral equation above, and Lemma \ref{lem:EstDerivF} that if $ t \in [ T^{*}, T^{*} + T^{'} ] $  then

\EQQARR{
\left\|  X \left( h_{\theta}(t), u_{1}(t), t,t \right) \right\|_{L^{\infty}} & \lesssim  \left\| h_{\theta}(t) - u_{1}(t) \right\|_{H^{\frac{7}{2}+}} \\
&  \lesssim  \left\| u_{2}(t) - u_{1}(t) \right\|_{H^{\frac{7}{2}+}} \\
&  \lesssim  \left\| u_{2}(t) - u_{q}(T^{*}) \right\|_{H^{\frac{7}{2}+}} + \| u_{1}(t) - u_{q}(T^{*}) \|_{H^{\frac{7}{2}+}} \\
& \lesssim_{K} (t - T^{*}) \\
& \ll \delta \left( \overrightarrow{u_{q}(T^{*})}, T^{*} \right) \cdot
}
Hence we can apply Lemma  \ref{lem:EnergyDiffEst} with $\epsilon_{1} = \epsilon_{2} = 0$, replacing `$k'$'  with `$k-3$'
\footnote{The proof of Lemma \ref{lem:EnergyDiffEst} with $k' := k-3$  relies upon Proposition \ref{prop:linearbarv} and Proposition \ref{prop:linearbarvL2} with  $v_{q}(t) := \Phi_{k'} \left( u_{q}(t),t \right) \partial_{x}^{6} u_{q}(t)$. The proof of each proposition has to be changed slightly in the case where $\epsilon_{1}= \epsilon_{2} =0$. Indeed observe that one cannot use an approximation of $u_{q}$ with smooth solutions to justify some computations unlike the case $\epsilon_{1} \neq 0$, $\epsilon_{2} \neq 0$ (see Remark \ref{Rem:smooth}). In order to overcome this issue, one should replace in the proof  every expression of the form ``$ \langle f , D^{2p} \bar{v}(t) \rangle $'' with
``$ (-1)^{p} \langle D^{p}  f , D^{p} \bar{v}(t) \rangle $'' if $ p \in \{ k'-3, k'-6\ \}$.}
. Hence letting $\epsilon_{1} = \epsilon_{2} = 0$ in (\ref{Eqn:EstNrjDiff}) and replacing `$k'$' with `$k-3$' we see that
$\frac{d E_{k-3} \left( u_{1}(t),u_{2}(t),t \right)}{dt}  \lesssim  E_{k- 3} \left( u_{1}(t),u_{2}(t),t \right)$ for
$t \in [ T^{*}, T^{*} + T' ] $. The Gronwall inequality yields $ E_{k-3} \left( u_{1}(t), u_{2}(t), t \right) = 0 $ and consequently $u_{1}(t)= u_{2}(t)$ for
$t \in [T^{*}, T^{*} + T']$.\\
\\
We then prove the continuous dependence. \\
Let $T^{*} := \sup \left\{ t \in [0,\breve{T}]: \; \sup_{t' \in [0,t] } \left\| u_{n}(t') - u_{\infty}(t') \right\|_{H^{k}} \; \rightarrow 0 \; \text{as}
\, n \rightarrow \infty  \right\} $. \\
\underline{Claim}: $T^{*} = \breve{T}$. Indeed assume that $T^{*} < \breve{T}$. Then, by slightly adapting the arguments in Appendix we infer that
$ \delta : [0,\breve{T}] \rightarrow \mathbb{R}: t \rightarrow \delta(\overrightarrow{u_{\infty}(t)}, t) $ and that $ \delta' : [0,\breve{T}] \rightarrow \mathbb{R}:
t \rightarrow \delta'(u_{\infty}(t),t) $ are continuous: hence there exists a $\bar{\delta} > 0$ such that $ \min \left( \,
\delta \left( \overrightarrow{u_{\infty}(t)}, t \right),
\delta' (u_{\infty}(t), t ) \, \right) \geq \bar{\delta} $ for all $t \in [0, \breve{T} ]$. We also infer that for all $t  \in [0, T^{*})$,
$\left( \delta \left( \overrightarrow{u_{n}(t)} , t \right), \delta' \left( u_{n}(t), t  \right) \right) \rightarrow
\left( \delta \left( \overrightarrow{u_{\infty}(t)} , t \right) , \delta^{'} \left( u_{\infty}(t), t \right)\right) $
as $n \rightarrow \infty$, since $ \left\| u_{n}(t)- u_{\infty}(t) \right\|_{H^{k}} \rightarrow 0$ as $n \rightarrow \infty$. Let $q \in \{ n, \infty \}$.
Let  $ 0 \leq \tilde{t} \leq T^{*} $ be close enough to $ T^{*} $ such that all the statements below are true. Unless otherwise stated, let $n$,$m$ be large enough such that
all the statements below are true for all $p \in \mathbb{N}$.
Let  $u_{n,m}$ (resp.  $u_{\infty,m}$ ) be the solution of  (\ref{Eqn:Integral}) with $\epsilon := \frac{1}{m}$, replacing
``$ e^{-\epsilon t \partial_{x}^{4}}\tilde{\phi}$''
and ``$\int_{0}^{t}$''  with  ``$  e^{-\epsilon (t- \tilde{t}) \partial_{x}^{4}}  u_{n,m}(\tilde{t})$ '' (resp.  
``$ e^{-\epsilon (t- \tilde{t}) \partial_{x}^{4}} u_{\infty,m} (\tilde{t})  $'' ) and ``$ \int_{\tilde{t}}^{t} $'' respectively. Here 
$ u_{n,m}(\tilde{t}) := J_{\frac{1}{m},k} u_{n}(\tilde{t})$ (resp. $u_{\infty,m}(\tilde{t}) := J_{\frac{1}{m},k} u_{\infty}(\tilde{t})$). Let $q \in \{ n, \infty \} $. 
Let $K := \sup \limits_{t \in [0, \breve{T}]} \| u_{\infty}(t) \|_{H^{k_{0}}} $.
We see from Proposition \ref{prop:compEHnorm} and Proposition \ref{prop:EnergyDiffEst} (applied to $w_{q}$, with $w_{q}(t- \tilde{t}):= u_{q}(t)$) that there exists
$T' := T' \left( \bar{\delta}, K  \right) > 0 $  such that $ \breve{T}  > \tilde{t} + T' > T^{*} $ and

\EQQARRLAB{
t \in [\tilde{t}, \tilde{t} + T']: \; \| u_{q,m+p} \|_{L_{t}^{\infty} H^{k}([\tilde{t}, \tilde{t} + T'])} \lesssim 1, \\
t \in [\tilde{t}, \tilde{t} + T']: \; u_{q,m}(t) \in \mathcal{P}_{+,k}(t) \; \text{holds}, \; \delta \left( \overrightarrow{u_{q,m}(t)}, t \right)
\gtrsim \bar{\delta}, \; \delta^{'} \left( u_{q,m}(t), t \right) \gtrsim \bar{\delta} , \; \text{and} \\
\| u_{q,m+p} - u_{q,m} \|_{L_{t}^{\infty} H^{k} ([\tilde{t},\tilde{t} + T'])} \lesssim  \left( \frac{1}{m} \right)^{0+} + \alpha \left( \frac{1}{m}  \right) \cdot
\label{Eqn:EstDiffuq}
}
Since $\{ u_{q,m} \}$ is a Cauchy sequence there exists $ v_{q} \in \mathcal{C} \left( [\tilde{t}, \tilde{t} + T'], H^{k} \right) $ such that
$ u_{q,m} \rightarrow v_{q} $. We now claim that $v_{q}= u_{q}$. Indeed by using similar arguments  as those that appear in the proof
of the local existence part of the theorem, we see that $v_{q}(t) = u_{q}(\tilde{t}) + \int_{\tilde{t}}^{t} F \left( \overrightarrow{v_{q}(t')}, t' \right) \; dt'$
for $t \in [\tilde{t}, \tilde{t} + T']$. Moreover letting $ p \rightarrow \infty$ in (\ref{Eqn:EstDiffuq}) we get

\EQQARRLAB{
\sup \limits_{t \in [\tilde{t}, \tilde{t} + T']} \| v_{q}(t) - u_{q,m} (t) \|_{H^{k}} & \lesssim \left( \frac{1}{m} \right)^{0+} + \alpha \left( \frac{1}{m}  \right)
\label{Eqn:EstDiffuq2limit}
}
Hence, we infer from Appendix that $ v_{q}(t) \in \mathcal{P}_{+,k}(t)$ for $t \in [\tilde{t}, \tilde{t} + T'] $. Hence $v_{q} = u_{q}$ by using the uniqueness part of the theorem. \\
We also see from Proposition \ref{Prop:Integral} that $ \lim \limits_{m \rightarrow \infty}
\| u_{n,m} - u_{\infty,m} \|_{L_{t}^{\infty} H^{k'} ([\tilde{t}, \tilde{t} + T'])} = 0$
as $ m \rightarrow \infty $. Let $\epsilon' > 0$. We have

\EQQARR{
\sup_{t \in [\tilde{t},\tilde{t} + T']} \| u_{n}(t) - u_{\infty}(t) \|_{H^{k}}
& \lesssim  \sup_{t \in [\tilde{t}, \tilde{t} + T']} \| u_{n}(t) - u_{n,m}(t) \|_{H^{k}} \\
& + \sup_{t \in [\tilde{t}, \tilde{t} + T']}  \| u_{n,m} (t) -  u_{\infty,m}  (t) \|_{H^{k}}   \\
& + \sup_{t \in [\tilde{t}, \tilde{t}+T']}  \| u_{\infty,m} (t) - u_{\infty} (t) \|_{H^{k}}   \\
&  \leq \epsilon' \cdot
}
Hence $ \lim \limits_{n \rightarrow \infty} \left\| u_{n}(t) - u_{\infty}(t) \right\|_{H^{k}} = 0 $ for $ \tilde{t} \leq  t \leq \tilde{t} + T' $. This is a contradiction. Hence $T^{*} = \breve{T}$. \\
Moreover by adapting slightly the argument above we see that $ \lim \limits_{n \rightarrow \infty} \left\| u_{n}(\breve{T}) - u_{\infty}(\breve{T}) \right\|_{H^{k}} = 0 $.  \\

\section{Proof of Corollary \ref{Cor:MainDiff}}

We only prove the corollary for $\phi \in \widetilde{\mathcal{P}}_{+,k}$, since a straightforward modification of the arguments below would show that the conclusion of the corollary also holds if  $\phi \in \widetilde{\mathcal{P}}_{-,k}$. \\
Let $ -T <  \tilde{t} < 0 $ small enough such that all the statements below are correct. Since $u \in \mathcal{C} \left( [-T,0], H^{k} \right)$, then
we see from Appendix  with  $ (f, g):=  \left( u(\tilde{t}), \phi \right)$ that
$ \delta \left( \overrightarrow{u(\tilde{t})}, \tilde{t} \right) \gtrsim \tilde{\delta}(\phi) $ and $ \inf \limits_{x \in \T} Q \left(u(\tilde{t}), \tilde{t} \right)(x) > 0 $. Now consider the problem
$ w(t) = u(\tilde{t}) + \int_{0}^{t} F \left( \overrightarrow{w(t')}, t' +  \tilde{t} \right) \; dt'$. We see from Theorem \ref{Thm:MainDiff} and Remark
\ref{Rem:Exist} that $w(t) = u(t + \tilde{t}) $ for $t \in [0,\tilde{t}]$ and that there exists $ T'  > |\tilde{t}| $ such that
$ w \in \mathcal{C}^{\infty} \left( (0,T'] \times \T \right) $. In particular $ \phi \in \mathcal{C}^{\infty} (\T)$, which is a
contradiction.

\section{Proof of Lemmas}
\label{Sec:ProofLemmas}

In his section we prove Lemma \ref{lem:OneEnergyEst} and  Lemma \ref{lem:EnergyDiffEst}.

\subsection{Proof of Lemma  \ref{lem:OneEnergyEst} }

An application of the Leibnitz rule shows that

\EQQARRLAB{
\p_{x}^{6} \left( F ( \overrightarrow{u(t)},t ) \right) & = \p_{\omega_{3}} F \left( \overrightarrow{u(t)}, t \right) \p_{x}^{9} u(t)
+ P \left( u(t),t \right) \p_{x}^{8} u(t) + \check{Q} \left( u(t),t \right) \p_{x}^{7} u(t) + R \left( u(t),t \right),
\label{Eqn:ComputDerivF6}
}
with $\check{Q} (u(t),t)$ polynomial in $\p_{x}^{5} u(t)$,..., $u(t)$, and derivatives of $F$ evaluated at
$\left( \overrightarrow{u(t)},t \right)$ (resp. $R (u(t),t)$ polynomial in
$\p_{x}^{6} u(t)$,..., $u(t)$, and derivatives of $F$ evaluated at $\left( \overrightarrow{u(t)},t \right)$ ). \\
\\
Recall that $v(t) := \Phi_{k'} \left( u(t),t \right) \p_{x}^{6} u(t)$. Combining the above equality with (\ref{Eqn:Regul}) we see that

\EQQARR{
\p_{t} v + \epsilon \p_{x}^{4} v & =  a_3 \p_{x}^{3} v  + a_2 \p_{x}^{2} v + a_1 \p_{x} v  + a_{\diamond} + \epsilon \left( b_3 \p_{x}^{3} v
+ b_2 \p_{x}^{2} v +  \bar{b}_{1} \p_x v + \p_{x}^{4} c v \right),
}
with $c(u(t),t) := - \Phi_{k'} \left( u(t),t \right)$ and $a_{\diamond} (u(t),t) :=
\bar{a}_{0}(u(t),t) + \partial_{t}  \Phi_{k'} \Phi_{k'}^{-1} (u(t),t) v(t)$. Here

\EQQARR{
\bar{a}_{0}(u(t),t) := \\
\Phi_{k'} \left( u(t),t \right) R \left( u(t),t \right) + \\
\left(
\begin{array}{l}
\Phi_{k'} \left( u(t),t \right) \left(
\begin{array}{l}
\p_{x}^{3} \Phi^{-1}_{k'} \left( u(t),t \right)
\p_{\omega_3} F \left( \overrightarrow{u(t)},t \right)
+ \p_{x}^{2} \Phi^{-1}_{k'} \left( u(t),t \right) P \left( u(t),t \right) \\
+ \p_{x} \Phi^{-1}_{k'} \left( u(t),t \right) \check{Q} \left( u(t),t  \right)
\end{array}
\right)
\end{array}
\right) v(t),
}

\EQQARR{
\left( a_3(u(t),t) , b_3(u(t),t) \right) := \left(  \p_{\omega_3} F \left( \overrightarrow{u(t)},t \right), - 4 \Phi_{k'} \left( u(t),t \right) \p_{x} \Phi^{-1}_{k'} \left( u(t),t \right) \right), \\
\left( a_2(u(t),t) , b_2(u(t),t) \right) :=
\left(
\begin{array}{l}
3 \Phi_{k'} \left( u(t),t \right) \p_{x} \Phi^{-1}_{k'} \left( u(t),t \right) \p_{\omega_3} F \left( \overrightarrow{u(t)},t \right) + P \left( u(t),t \right), \\
 -6 \Phi_{k'} \left( u(t),t \right) \p_{x}^{2} \Phi_{k'}^{-1} \left(  u(t),t \right)
\end{array}
\right), \, \text{and} \\
\left( a_1(u(t),t), \bar{b}_1(u(t),t) \right)   :=
\left(
\begin{array}{l}
\Phi_{k'} \left( u(t),t \right)
\left(
3 \p_{x}^{2} \Phi_{k'}^{-1} \left( u(t),t \right) \p_{\omega_3} F  \left( \overrightarrow{u(t)},t \right)
+ 2 \p_{x} \Phi_{k'}^{-1} \left( u(t),t \right)  P \left( u(t),t \right)
\right) \\
+ \check{Q} \left( u(t),t \right), - 4 \Phi_{k'} \left( u(t),t \right) \p_{x}^{3} \Phi^{-1}_{k'} \left( u(t),t \right)
\end{array}
\right) \cdot
}
Let $\bar{P} \left( u(t) \right) :=  \frac{P \left( u(t),t \right) }{_{\p_{\omega_3} F( \overrightarrow{u(t)},t)}} -
\left[ \frac{ P \left( u(t),t \right) }{_{\p_{\omega_3} F(\overrightarrow{u(t)},t)}}  \right]_{ave}  $.
We now rewrite $ \p_{t} \Phi_{k'} \left( u(t),t \right) \Phi^{-1}_{k'} \left( u(t),t \right) v(t) $ in such a way that we can apply Proposition \ref{prop:linearv}. \\
\\
\underline{Claim}: Let $S \left( u(t), t \right)$ be a polynomial in $u(t)$, $\p_{x} u(t)$,...,$\p_{x}^{5} u(t)$, $\p_{x}^{6} u(t)$, $\p_{t} u(t)$, $\p_{t} \p_{x} u(t)$, $\p_{t} \p_{x}^{2} u(t)$, derivatives of $F$ evaluated at $ \left( \overrightarrow{u(t)}, t \right)$, and  $ \left( \p_{\omega_3} F \left( \overrightarrow{u(t)},t \right) \right)^{-1} $. Let $m \geq 0$. Then

\EQQARRLAB{
\| S (u(t),t)  \|_{H^{m}} & \lesssim \langle \| u(t) \|_{H^{m+6}} \rangle \cdot
\label{Eqn:EstSut}
}
Let $ S \left( \overrightarrow{u(t)} , t \right)$ (resp. $\bar{S} (u(t),t)$)  be a polynomial in $u(t)$,..., $\p_{x}^{3}u(t)$
(resp. $u(t)$,..., $\p_{x}^{5} u(t)$), derivatives of $F$ evaluated at
$\left( \overrightarrow{u(t)}, t \right)$, and  $ \left( \p_{\omega_3} F ( \overrightarrow{u(t)} , t) \right)^{-1} $ . Let $m \geq 0$. Then

\EQQARRLAB{
\| S \left( \overrightarrow{u(t)}, t \right) \|_{H^{m}} & \lesssim \langle \| u(t) \|_{H^{m+3}} \rangle \cdot
\label{Eqn:EstSubart}
}
(resp. $ \| \bar{S} (u(t),t) \|_{H^{m}} \lesssim \langle \| u(t) \|_{H^{m+5}} \rangle $. )

\begin{proof}

We only prove (\ref{Eqn:EstSut}). The proof of (\ref{Eqn:EstSubart}) follows from that of (\ref{Eqn:EstSut}): therefore, it is left to the reader. \\
One can find $I_0$,...,$I_4$, $I_5$, $I_6$, $J_0$, $J_1$, $J_2$ finite subsets of $\{1,2,... \} \times ... \times \{1,2,... \}$ such that for all
$ \vec{i} := (i_0,...,i_6) \in I_0 \times... \times I_6 $ and for all $\vec{j} := (j_0,j_1,j_2) \in J_0 \times J_1 \times J_2 $, one can find
$a_{\vec{i},\vec{j}} \in \R$ such that $a_{\vec{i},\vec{j}} $ is a constant multiplied by a power of $ \left( \partial_{\omega_{3}} F \left( \overrightarrow{u(t)} ,t \right) \right)^{-1} $ that is multiplied by derivatives of $F$ evaluated at $\left( \overrightarrow{u(t)},t \right) $ such that

\EQQARR{
S \left( u(t), t \right) & = \sum \limits_{ \substack{\vec{i} \in I
_{0} \times ... \times I_{6} \\ \vec{j} \in J_0 \times J_1 \times J_2  } }
a_{\vec{i},\vec{j}}  (\p_{x}^{6} u(t))^{i_6}... (u(t))^{i_0}
( \p_{t} \p_{x}^{2} u(t) )^{j_2} ( \p_{t} \p_{x} u(t) )^{j_1} (\p_{t} u(t))^{j_0}
}
If $p \in \{ 0,1,2 \} $  then
$\p_{t} \p_{x}^{p} u(t) = - \epsilon \p_{x}^{4+p} u(t) + \p_{x}^{p} \left( F  \left( \overrightarrow{u(t)}, t \right) \right)$. An application of the
Leibnitz rule shows that $\p_{x}^{p} \left( F(\vec{f} ,t ) \right)$  is a  polynomial in
$\p_{x}^{6}f$,...,$f$, and derivatives of $F$ evaluated at
$(\vec{f},t)$. Applying (several times) Lemma \ref{Lem:prod} and \ref{lem:EstDerivF},  and using Proposition \ref{prop:compEHnorm}, we see that there
exists  $ \alpha \in \N$ such that

\EQQARR{
\| S \left( u(t),t \right) \|_{H^{m}} & \lesssim  \langle \| u(t) \|_{H^{m+6}} \rangle \langle \| u(t) \|_{H^{\frac{13}{2}+}} \rangle^{\alpha} \lesssim \langle \| u(t) \|_{H^{m+6}} \rangle \cdot
}

\end{proof}

Throughout the sequel of the proof, let $S \left( u(t), t \right)$ (resp. $S \left( \overrightarrow{u(t)},t \right)$) denotes a function that is a polynomial in $u(t)$, $\p_{x} u(t)$,...,$\p_{x}^{6} u(t)$, $\p_{t} \p_{x}^{2} u(t)$, $\p_{t} \p_{x} u(t)$ \\
$\left( \, \text{resp.} \, u(t), \p_{x} u(t),...,\p_{x}^{3} u(t) \, \right)$, derivatives of $F$ evaluated at $\left( \overrightarrow{u(t)},t \right)$, and  $ \left(\p_{\omega_3} F \left( \overrightarrow{u(t)},t \right) \right)^{-1}$. Here $u$ is a  function depending on the variables $x$ and $t$. For sake of simplicity we allow the value of $S \left( u(t),t \right)$  $ \left( \, \text{resp.} \, S \left( \overrightarrow{u(t)},t \right) \, \right) $  to change from one line to the other one. If several functions share the same properties as those stated just above appear within the same equation then $S_{1} \left( u(t),t \right)$, $S_{2} \left( u(t),t \right)$,... $\left( \, \text{resp.} \, S_{1} \left( \overrightarrow{u(t)},t \right), S_{2} \left( \overrightarrow{u(t)},t \right),... \right)$
 denote the first function, the second function,..., respectively. We also allow the value of $S_{1} \left( u(t),t \right)$, $S_{2} \left( u(t),t \right)$,...
 $\left( \, \text{resp.} \, S_{1} \left( \overrightarrow{u(t)},t \right), S_{2} \left( \overrightarrow{u(t)},t \right),... \right)$ to vary from one line to the other one.  \\
\\
An application of the Leibnitz rule shows that

\EQQARR{
\p_{t} P \left( u(t),t \right)  & = \bar{X}_{P}(\overrightarrow{u(t)},t) \p_{t} \p_{x}^{4} u(t) + \bar{Y}_{P}(u(t),t)
\p_{t} \p_{x}^{3} u(t) + S \left( u(t),t \right), \; \text{and} \\
\p_{t} \p_{\omega_3} F ( \overrightarrow{u(t)},t ) & = \bar{Y}_{F}( \overrightarrow{u(t)},t ) \p_{t} \p_{x}^{3} u(t) + S \left( u(t),t \right),
}
with $\bar{X}_{P}(\vec{f},t) := 6 \p^{2}_{\omega_3} F(\vec{f},t)$,
$ \bar{Y}_{P}(f,t) := 6 \sum \limits_{m=-1}^{3} \p_{\omega_{3} \omega_3 \omega_m}^{3} F(\vec{f},t) \p_{x}^{m+1} f
 +  \p^{2}_{\omega_3 \omega_2} F(f,t) $ \footnote{Notation convention in the formula of $\bar{Y}_{P}(f,t)$: $\p_{x}^{0} f := 1$.}, and
$\bar{Y}_{F}(\vec{f},t) := \p_{\omega_3}^{2} F(\vec{f},t) $. \\
We have $\p_{t} \p_{x}^{4} u(t) = - \epsilon \p_{x}^{8} u(t) + \p_{x}^{4} \left( F (\overrightarrow{u(t)},t) \right)$. Observe that

\EQQARR{
\p_{x}^{8} u(t)  = \p_{x}^{2} \left( \Phi_{k'}^{-1} ( u(t),t )  v(t) \right) = \p_{x}^{2} \Phi_{k'}^{-1}( u(t),t ) v(t) + 2 \p_{x} \Phi_{k'}^{-1} ( u(t), t ) \p_{x} v(t)
+  \Phi_{k'}^{-1} ( u(t),t ) \p_{x}^{2} v(t)  \cdot
}
An application of the Leibnitz rule shows that

\EQQARR{
\p_{x}^{4} \left( F ( \overrightarrow{u(t)}, t ) \right) & = \p_{\omega_3} F \left( \overrightarrow{u(t)},t \right) \p_{x}^{7} u
+  S \left( u(t),t \right) \cdot
}
We also have $\p_{t} \p_{x}^{3} u(t) = - \epsilon \p_{x}^{7} u(t) + S \left( u(t),t \right)$ with

\EQQARR{
\p_{x}^{7} u(t)   =  \p_{x} \left( \Phi_{k'}^{-1} (u(t),t) v(t) \right) = \p_{x} \Phi_{k'}^{-1} (u(t),t) v(t)  + \Phi_{k'}^{-1} ( u(t), t) \p_{x} v(t) \cdot
}
Hence $\p_{t} \p_{x}^{4} u(t) = A_4(t) \p_{x}^{2} v(t) + B_4(t) \p_{x} v(t) + C_4(t) $ with
$ A_4(t)  := - \epsilon \Phi_{k'}^{-1} \left( u(t), t \right)$,  $ B_4(t) :=   \p_{\omega_3} F  \left( \overrightarrow{u(t)}, t \right) \Phi_{k'}^{-1} \left( u(t), t \right)
- 2 \epsilon \p_{x} \Phi_{k'}^{-1} \left( u(t), t \right)$, and \\
$ C_4(t) :=   \left( \p_{\omega_3} F  \left( \overrightarrow{u(t)},t \right) \p_{x} \Phi_{k'}^{-1} \left( u(t), t \right) - \epsilon \p_{x}^{2} \Phi_{k'}^{-1} \left(  u(t), t \right) \right) v(t) + S \left( u(t) ,t \right) $.
We also have $ \p_{t} \p_{x}^{3} u(t) = B_3(t) \p_{x} v(t) + C_3(t)  $ with $B_3(t) := - \epsilon \Phi_{k'}^{-1} \left( u(t) ,t \right) $ and
$C_3(t) := - \epsilon \p_{x} \Phi_{k'}^{-1} ( u(t),t ) v(t) + S \left( u(t), t \right) $. \\
\\
Hence

\EQQARR{
\p_{t} P \left( u(t) ,t \right) & = X_{P} \left( u(t), t \right)  \p_{x}^{2} v(t) + Y_{P} \left( u(t),t \right) \p_{x} v(t) + Z_{P} \left( u(t), t \right) \; \text{and} \\
\p_{t} \p_{\omega_3} F( \overrightarrow{u(t)}, t ) & = Y_{F} \left( u(t),t \right) \p_{x} v(t) + Z_{F} \left( u(t), t \right),
}
with $ X_{P} \left( u(t), t \right) := \bar{X}_{P}( \overrightarrow{u(t)},t ) A_{4}(t) $, $ Y_{P} \left( u(t),t \right) :=  \bar{X}_{P}( \overrightarrow{u(t)}, t ) B_{4}(t) +
\bar{Y}_{P}(\overrightarrow{u(t)}) B_3(t) $, $ Z_{P} \left( u(t), t \right) := \bar{X}_{P} \left( u(t), t \right) C_4(t) + C_{3}(t) \bar{Y}_{P}(\overrightarrow{u(t)}, t)
+ S \left( u(t), t \right) $,  $ Y_{F} \left( u(t), t \right) :=  \bar{Y}_{F}(\overrightarrow{u(t)},t) B_{3}(t) $, and $Z_{F} \left( u(t), t \right) := \bar{Y}_{F}(\overrightarrow{u(t)}, t) C_3(t) + S \left( u(t),t \right)$.\\
\\
Hence

\EQQARR{
\Phi^{-1}_{k'} \left( u(t) ,t \right) \p_{t} \Phi_{k'} \left( u(t), t \right) v(t) & =  \tilde{a}(u(t),t)
+ \epsilon \tilde{b}_{1}(u(t),t) \p_{x} v(t) + \epsilon b_{0}(u(t),t),
}
with $\tilde{b}_{1}(u(t),t) := S \left( \overrightarrow{u(t)},t \right)  \Phi^{-1}_{k'} \left( u(t),t \right) v(t)$,
$b_{0} (u(t),t) :=S \left( \overrightarrow{u(t)},t \right) \p_{x} \Phi_{k'}^{-1} \left( u(t),t \right) v^{2}(t)$, and

\begin{equation}
\begin{array}{l}
\tilde{a}(u(t),t) := S_{1} (u(t),t) v(t) + v(t) \times \\
\int_{0}^{x}
\left[
\begin{array}{l}
S_{1} \left(\overrightarrow{u(t)},t \right) \Phi^{-1}_{k'} \left( u(t),t \right)  \p_{x} v(t)  + S_{2} \left(\overrightarrow{u(t)},t \right) v(t) + S_{3} \left(\overrightarrow{u(t)},t \right)  \\
+ \epsilon \left( \Phi^{-1}_{k'} \left( u(t),t \right) S_{4} \left( \overrightarrow{u(t)},t \right)  \right) \partial_{x}^{2} v(t)  \\
+ \epsilon \left(  S_{2} \left(u(t),t \right) \Phi^{-1}_{k'} \left( u(t),t \right) + S_{5} \left( \overrightarrow{u(t)},t \right) \partial_{x} \Phi^{-1}_{k'} \left( u(t), t \right)  \right) \partial_{x} v(t) \\
+ \epsilon \left(  S_{6} \left( \overrightarrow{u(t)},t \right) \p_{x}^{2} \Phi^{-1}_{k'}
+ S_{3}( u(t),t) \p_{x} \Phi^{-1}_{k'} \right) v(t) + S_{7} (\overrightarrow{u(t)},t) \p_{x} \Phi_{k'}^{-1}  v(t) + S_{4}(u(t),t)
\end{array}
\right] \; d x' \cdot
\end{array}
\label{Eqn:Deftildea0}
\end{equation}
Hence

\EQQARR{
\p_{t} v + \epsilon \p_{x}^{4} v & =  a_3 \p_{x}^{3} v  + a_2 \p_{x}^{2} v + a_1 \p_{x} v  +  a_0 + \epsilon \left( b_3 \p_{x}^{3} v + b_2 \p_{x}^{2} v + b_1 \p_x v + b_{0} v(t) +  \p_{x}^{4} c v \right),
}
with

\EQQARR{
a_0(u(t),t) := \bar{a}_{0}(u(t),t) + \tilde{a}(u(t),t), \; \text{and} \; b_{1}(u(t),t) :=  \bar{b}_{1}(u(t),t) + \tilde{b}_{1}(u(t),t) \cdot
}
We get from (\ref{Eqn:CrucEq})

\EQQARRLAB{
 \left( k' - \frac{15}{2} \right) \p_{x} a_{3}(u(t),t) + a_{2} (u(t),t) & =  \p_{\omega_{3}} F \left( \overrightarrow{u(t)},t \right)
 \left[ \frac{P \left( u(t),t \right)}{ \p_{\omega_{3}} F \left( \overrightarrow{u(t)},t \right)}  \right]_{ave}  \cdot
\label{Eqn:Beg}
}\
In the sequel we use the above claim, Lemma \ref{Lem:prod}, Lemma \ref{lem:gauge}, Proposition \ref{prop:compEHnorm}, and (\ref{Eqn:EstPf}).
We have

\EQQARRLAB{
 \sum \limits_{m=1}^{3} \| a_{m}(u(t),t) \|_{H^{k'-6+m}} \| v(t) \|_{H^{\frac{1}{2}+}}  \lesssim 1 + E_{k'}^{\frac{1}{2}}(u(t),t), \; \text{and} \\
\sum \limits_{m=1}^{3} \| a_{m}(u(t),t) \|_{H^{ \left( m+ \frac{1}{2} \right)+}} \lesssim 1 + E_{k'}^{\frac{1}{2}} (u(t),t) \cdot
\label{Eqn:Estam}
}
Using also (\ref{Eqn:ExpIntBasic})  (and the interpolation inequality  $\| v(t) \|_{\dot{H}^{k'-5}} \lesssim
\| v(t) \|^{\frac{1}{2}}_{\dot{H}^{k'-4}} \| v(t) \|^{\frac{1}{2}}_{\dot{H}^{k'-6}}$ followed by the Young inequality
$ ab \leq \frac{a^{2}}{2} + \frac{b^{2}}{2} $ to estimate the  $\partial_{x}^{2} v(t)-$ term of (\ref{Eqn:Deftildea0}) )   we get

\EQQARRLAB{
\| a_{0}(u(t),t) \|^{2}_{H^{k'-6}} \lesssim 1 + E_{k'}(u(t),t) + \epsilon^{2} \| D^{k'-4} v(t) \|^{2}_{L^{2}} \cdot
\label{Eqn:Esta0}
}
We have

\EQQARRLAB{
 \| b_0(u(t),t) \|_{H^{k'-6}}   \lesssim 1 + E_{k'}^{\frac{1}{2}}(u(t),t), \, \text{and} \,
\sum \limits_{m=1}^{3} \| b_{m}(u(t),t) \|_{H^{ \left( m+ \frac{1}{2} \right)+ } } + \| c(u(t),t) \|_{H^{\frac{9}{2}+}}  \lesssim 1 \cdot
\label{Eqn:Estbm}
}
We have $\| b_{2}(u(t),t) \|_{H^{k'-4}} + \| b_{3}(u(t),t) \|_{H^{k'-3}} \lesssim  1 + \| u(t) \|_{H^{k'+1}} $. We also have

\EQQARR{
\| v(t) \|_{L^{2}} \lesssim \left\| \Phi_{k'}(u(t),t) \right\|_{L^{\infty}} \| \p_{x}^{6} u(t) \|_{L^{2}} \lesssim   \| \Phi_{k'} (u(t),t) \|_{H^{\frac{1}{2}+}} \| u(t) \|_{H^{6}} \lesssim 1 \cdot
}
By the inequalities mentioned just above (\ref{Eqn:Esta0}) we get

\EQQARR{
\| b_{1}(u(t),t) \|_{H^{k'-5}} \lesssim  \| v(t) \|_{ \dot{H}^{k'-4}} + 1 + \|  u(t) \|_{H^{k'+1}} \cdot
}
We have

\EQQARR{
\| u(t) \|_{H^{k'+1}} & \lesssim \| u(t) \|_{L^{2}} + \| u(t) \|_{\dot{H}^{k'+1}} \\
& \lesssim \| u(t) \|_{L^{2}} + \| \Phi^{-1}_{k'} \left( u(t),t \right) v(t) \|_{\dot{H}^{k'-5}} \\
& \lesssim  \| u(t) \|_{L^{2}} + \| \Phi^{-1}_{k'} \left( u(t),t \right) \|_{H^{k'-5}} \| v (t) \|_{H^{\frac{1}{2}+}} + \| \Phi^{-1}_{k'} \left( u(t),t \right) \|_{H^{\frac{1}{2}+}} \| v(t) \|_{H^{k'-5}} \\
& \lesssim  1 + E_{k'}^{\frac{1}{2}}(u(t),t) + \| v(t) \|_{\dot{H}^{k'-4}} \cdot
}
Hence

\EQQARRLAB{
 \| v(t) \|_{H^{\frac{1}{2}+}} \max \left( \| b_{3}(u(t),t) \|_{H^{k'-3}}, \| b_{2}(u(t),t) \|_{H^{k'-4}}, \|  b_{1}(u(t),t) \|_{H^{k'-5}} \right)  \\
 \lesssim 1 + E_{k'}^{\frac{1}{2}} (u(t),t) + \| v(t) \|_{\dot{H}^{k'-4}} \cdot
\label{Eqn:Estb1}
}
We have

\EQQARRLAB{
\| v(t) \|_{H^{\frac{1}{2}+}} \| c(u(t),t) \|_{H^{k'-2}} \lesssim  1 + \| u(t) \|_{H^{k'+1}} \lesssim  1 + E_{k'}^{\frac{1}{2}}(u(t),t) \cdot
\label{Eqn:Estc}
}
By the inequalities  mentioned just above (\ref{Eqn:Esta0}) we get

\EQQARRLAB{
\epsilon \left( \| b_{3}(t) \|_{H^{\frac{3}{2}+}} + \| b_{2}(t) \|_{H^{\frac{1}{2}+}} \right) \| \p_{x} D^{k'-6} v(t) \|_{L^{2}}^{2}
& \lesssim \epsilon^{2} \| D^{k'-4} v(t) \|_{L^{2}}^{2} + E_{k'}(u(t),t) \cdot
\label{Eqn:End}
}
Hence taking into account all the estimates above and taking into account that all the terms
$\epsilon^{2}  \| D^{k'-4} v(t) \|_{L^{2}}^{2}$ that appear can be absorbed by the term $\epsilon \| D^{k'-4} v(t) \|^{2}_{L^{2}}$ of
(\ref{Eqn:Meth1}), we get (\ref{Eqn:EstDerE}) from Proposition \ref{prop:linearv} and Proposition \ref{prop:linearL2v}.

\subsection{Proof of Lemma \ref{lem:EnergyDiffEst}}

Let $q \in \{ 1,2 \}$. If $y(\overrightarrow{u_{q}(t)},t)$ is a function depending on $\overrightarrow{u_{q}}(t)$  and $t$  then
$\triangle y(t) := y(\overrightarrow{u_{1}(t)},t) - y(\overrightarrow{u_{2}(t)},t)$. More generally a similar definition applies to a function $y(u_{q}(t),t)$.
Let $\alpha$ be a constant that is allowed to change from one line to the other one and such that all the statements and estimates below are true. \\
The proof relies upon the result below.

\subsubsection{One result}

\begin{res}

Let $ 0 \leq r \leq k'-6 $ and $q \in \{ 1,2 \} $. Let $S (u_{q}(t),t)$  be a polynomial in $u_{q}(t)$, $\p_{x} u_{q}(t)$...,$\p_{x}^{6} u_{q}(t)$,
$\p_{t} \p_{x}^{2} u_{q}(t)$, $\p_{t} \p_{x} u_{q}(t)$, $\p_{t} u_{q}(t) $, derivatives of $F$ evaluated at
$ \left( \overrightarrow{u_{q}(t)}, t \right)$, and  $ \left( \p_{\omega_3} F(\overrightarrow{u_{q}(t)},t) \right)^{-1} $. Then

\EQQARRLAB{
\left\| \triangle S(t) \right\|_{H^{r}} & \lesssim \| \triangle u(t) \|_{H^{r+6}} \cdot
\label{Eqn:EstSutDiff}
}
Let $ S(\overrightarrow{u_{q}(t)},t)$ be a polynomial in $u_{q}(t)$,..., $\p_{x}^{3}u_{q}(t)$, derivatives of $F$ evaluated at
$\left( \overrightarrow{u_{q}(t)},t \right)$, and  $ \left( \p_{\omega_3}  F(\overrightarrow{u_{q}(t)},t) \right)^{-1} $. Then

\EQQARRLAB{
\| \triangle S(t) \|_{H^{r}} & \lesssim \| \bar{u}(t) \|_{H^{r+3}}  \cdot
\label{Eqn:EstSubartDiff}
}
\label{Res:aux}
\end{res}

We postpone the proof of this result to the last subsection.

\subsubsection{The proof}

Let  $q \in \{ 1,2 \}$. Recall (see the proof of Lemma \ref{lem:OneEnergyEst}) that

\EQQARR{
\p_{t} v_{q} + \epsilon \p_{x}^{4} v_{q} & =  a_{3,q} \p_{x}^{3} v_{q}  + a_{2,q} \p_{x}^{2} v_{q} + a_{1,q} \p_{x} v_{q}  +  a_{0,q}  \\
& + \epsilon_{q} \left( b_{3,q} \p_{x}^{3} v_{q}
 + b_{2,q} \p_{x}^{2} v_{q} +  b_{1,q} \p_x v_{q}
+ \p_{x}^{4} c_{q} v_{q} \right) \cdot
}
Here the parameters $a_{3,q}$,..., $c_{q}$ are defined in the proof of Lemma  \ref{lem:OneEnergyEst}, replacing ``$u$'' with ``$u_{q}$''. We would like to apply Proposition \ref{prop:linearbarv} and Proposition \ref{prop:linearbarvL2}. To this end we estimate some norms. \\
Lemma \ref{lem:gauge}, and Proposition \ref{prop:compEHnorm}  yield

\EQQARRLAB{
\| \triangle c(t) \|_{H^{k'-2}}  \lesssim \left\| \triangle u(t) \right\|_{k'+1}, \; \text{and} \; \| \triangle c(t) \|_{H^{\frac{1}{2}+}}  \lesssim \left\| \triangle u(t) \right\|_{\frac{9}{2}+} \cdot
\label{Eqn:Diffc}
}
Lemma \ref{lem:EstDerivF}  yields

\EQQARRLAB{
\| \triangle a_{3}(t) \|_{H^{k'-6}} \lesssim \left\|  \triangle u(t) \right\|_{H^{k'-3} } \; \text{and}
\| \triangle a_{3}(t) \|_{H^{\frac{1}{2}+}} \lesssim \left\|  \triangle u(t) \right\|_{H^{\frac{7}{2}+} } \cdot
\label{Eqn:Diffa3}
}
Let  $r \in \{ k'-6, \frac{1}{2}+ \}$. Applying several times Lemma \ref{Lem:prod} yields $\| \triangle a_{2} (t) \|_{H^{r}} \lesssim X + \| \triangle P(t) \|_{H^{r}} $ with

\EQQARR{
X & := \| \triangle \Phi_{k'}(t) \|_{H^{r}}  \| \p_{x} \Phi_{k'}^{-1}(u_{1}(t),t) \|_{H^{\frac{1}{2}+}} \| \p_{\omega_{3}} F( \overrightarrow{u_{1}(t)} , t) \|_{H^{\frac{1}{2}+}} + ... + \\
& \|  \triangle \Phi_{k'}(u(t),t) \|_{H^{\frac{1}{2}+}}  \| \p_{x} \Phi_{k'}^{-1}(u_{1}(t),t) \|_{H^{\frac{1}{2}+}} \| \p_{\omega_{3}} F( \overrightarrow{u_{1}(t)} , t) \|_{H^{r}}  \\
& + ...... + \\
& \| \Phi_{k'}(u_{2}(t),t) \|_{H^{r}}  \| \p_{x} \Phi_{k'}^{-1}(u_{2}(t),t) \|_{H^{\frac{1}{2}+}}  \| \triangle \p_{\omega_{3}} F(\overrightarrow{u(t)},t) \|_{H^{\frac{1}{2}+}} + ... + \\
& \| \Phi_{k'}(u_{2}(t),t) \|_{H^{\frac{1}{2}+}}  \| \p_{x} \Phi_{k'}^{-1}(u_{2}(t),t) \|_{H^{\frac{1}{2}+}} \| \triangle \p_{\omega_{3}} F(t) \|_{H^{r}}
}
Hence, taking also into account (\ref{Eqn:DiffEstP}) and (\ref{Eqn:EstSut}) we get

\EQQARRLAB{
\| \triangle a_{2} (t) \|_{H^{k'-6}} \lesssim \left\| \triangle u(t) \right\|_{H^{k'}}, \; \text{and} \;
\| \triangle a_{2} (t) \|_{H^{\frac{1}{2}+}} \lesssim \left\| \triangle u(t) \right\|_{H^{\frac{9}{2}+}} \cdot
\label{Eqn:Diffa2}
}
Applying several times Lemma \ref{Lem:prod} we get $ \| \triangle a_{1}(u(t),t) \|_{H^{r}} \lesssim X + Y + \| \triangle \tilde{Q}(t) \|_{H^{r}} $ with

\EQQARR{
X &  \lesssim  \|  \triangle \Phi_{k'} (t) \|_{H^{r}} \| \partial_{x}^{2} \Phi_{k'}^{-1} (u_{1}(t),t) \|_{H^{\frac{1}{2}+}}
 \| \p_{\omega_{3}} F ( \overrightarrow{u_{1}(t)},t ) \|_{H^{\frac{1}{2}+}} + ... + \\
  & \|  \triangle \Phi_{k'} (t) \|_{H^{\frac{1}{2}+}} \| \partial_{x}^{2} \Phi_{k'}^{-1} (u_{1}(t),t) \|_{H^{\frac{1}{2}+}}
  \| \p_{\omega_{3}} F ( \overrightarrow{u_{1}(t)},t ) \|_{H^{r}} \\
 & + ....... + \\
 &  \| \Phi_{k'} (u_{2}(t),t) \|_{H^{r}}  \| \partial_{x}^{2} \Phi_{k'}^{-1} (u_{2}(t),t) \|_{H^{\frac{1}{2}+}}
 \| \triangle \p_{\omega_{3}} F(t) \|_{H^{\frac{1}{2}+}} + ... + \\
 &   \| \Phi_{k'} (u_{2}(t),t) \|_{H^{\frac{1}{2}+}}  \| \partial_{x}^{2} \Phi_{k'}^{-1} (u_{2}(t),t) \|_{H^{\frac{1}{2}+}}
 \| \triangle \p_{\omega_{3}} F(t) \|_{H^{r}}, \; \text{and}
 }

\EQQARR{
Y &  \lesssim  \|  \triangle  \partial_{x} \Phi_{k'}^{-1} (t) \|_{H^{r}} \| P (u_{2}(t),t) \|_{H^{\frac{1}{2}+}} + \|  \triangle \partial_{x}
\Phi_{k'} (u(t),t) \|_{H^{\frac{1}{2}+}} \| P (u_{2}(t),t) \|_{H^{r}} + \\
  &  \|  \partial_{x} \Phi_{k'}^{-1} (u_{2}(t),t) \|_{H^{r}}  \| \triangle P(u(t),t) \|_{H^{\frac{1}{2}+}} + \|  \partial_{x} \Phi_{k'}^{-1} (u_{2}(t),t) \|_{H^{\frac{1}{2}+}} \| \triangle P(t) \|_{H^{r}} \cdot
 }
Hence taking also into account (\ref{Eqn:EstSutDiff}) we get \footnote{Actually one can prove that
$ \| \triangle \tilde{Q}(t) \|_{H^{\frac{1}{2}+}}  \lesssim \left\| \triangle u(t) \right\|_{H^{\frac{11}{2}+}} $
by slightly modifying the proof of Result \ref{Res:aux}. Hence $ \| \triangle a_{1}(t) \|_{H^{\frac{1}{2}+}}  \lesssim  \left\| \triangle u(t) \right\|_{H^{\frac{11}{2}+}} $
}

\EQQARRLAB{
\| \triangle a_{1}(t) \|_{H^{k'-6}}  \lesssim  \left\| \triangle u(t) \right\|_{H^{k'}}, \; \text{and}
\; \| \triangle a_{1}(t) \|_{H^{\frac{1}{2}+}}  \lesssim  \left\| \triangle u(t) \right\|_{H^{\frac{13}{2}+}} \cdot
 \label{Eqn:Diffbara1}
}
We then estimate $ \| \triangle a_{0}(u(t),t) \|_{H^{k'-6}} $. We first estimate $ \| \triangle \bar{a}_{0}(t) \|_{H^{k'-6}} $.
Applying several times Lemma \ref{Lem:prod} we get

\EQQARR{
\| \triangle \bar{a}_{0}(t) \|_{H^{k'-6}} \lesssim X_{1,a} + X_{2,a} + Y_{1,a} + Y_{2,b} + Z_{a} + Z_{b} + Z_{c} \cdot
}
Here $X_{1,a}$, $Y_{1,a}$, $X_{2,a}$, and $ Y_{2,a}$ are defined as follows: $X_{1,a} :=  \| \triangle \Phi_{k'}(t) \|_{H^{k'-6}}  \| R (u_{1}(t),t) \| _{H^{\frac{1}{2}+}} $, $ X_{2,a} := \| \triangle \Phi_{k'}(t) \|_{H^{\frac{1}{2}+}} \| R (u_{1}(t),t) \|_{H^{k'-6}} $,  $Y_{1,a} :=  \| \Phi_{k'} (u_{2}(t),t) \|_{H^{k'-6}}
\| \triangle R(t) \|_{H^{\frac{1}{2}+}} $, $Y_{2,a} := \| \Phi_{k'}(u_{2}(t),t) \|_{H^{\frac{1}{2}+}} \| \triangle R(t) \|_{H^{k'-6}}$,

\EQQARR{
Z_{a} :=  & \| \triangle \Phi_{k'}(t) \|_{H^{k'-6}} \| \partial_{x}^{3} \Phi_{k'}(u_{1}(t),t) \|_{H^{\frac{1}{2}+}}
\left\| \partial_{\omega_{3}} F \left( \overrightarrow{u_{1}(t)},t \right) \right\|_{H^{\frac{1}{2}+}}
\| v_{1}(t) \|_{H^{\frac{1}{2}+}} + ... + \\
& \| \triangle \Phi_{k'}(t) \|_{H^{\frac{1}{2}+}} \| \partial_{x}^{3} \Phi_{k'}(u_{2}(t),t) \|_{H^{\frac{1}{2}+}}
\left\| \partial_{\omega_{3}} F \left( \overrightarrow{u_{2}(t)},t \right) \right\|_{H^{\frac{1}{2}+}}
\| v_{2}(t) \|_{H^{k'-6}}  \\
& + ......  + \\
& \| \Phi_{k'}(u_{2}(t),t) \|_{H^{k'-6}} \| \partial_{x}^{3} \Phi_{k'}(u_{2}(t),t) \|_{H^{\frac{1}{2}+}}
\left\| \partial_{\omega_{3}} F \left( \overrightarrow{u_{2}(t)},t \right) \right\|_{H^{\frac{1}{2}+}}
\| \triangle v (t) \|_{H^{\frac{1}{2}+}} + ... + \\
& \| \Phi_{k'}(t) \|_{H^{\frac{1}{2}+}} \| \partial_{x}^{3} \Phi_{k'}(u_{2}(t),t) \|_{H^{\frac{1}{2}+}}
\left\| \partial_{\omega_{3}} F \left( \overrightarrow{u_{2}(t)},t \right) \right\|_{H^{\frac{1}{2}+}}
\| \triangle v (t) \|_{H^{k'-6}},
}

\EQQARR{
Z_{b} := & \| \triangle  \Phi_{k'}(t) \|_{H^{k'-6}} \| \partial_{x}^{2} \Phi_{k'}^{-1}(u_{1}(t),t) \|_{H^{\frac{1}{2}+}}
\left\| P \left( u_{1}(t),t \right) \right\|_{H^{\frac{1}{2}+}}
\| v_{1}(t) \|_{H^{\frac{1}{2}+}} + ... + \\
& \| \triangle \Phi_{k'}(u(t),t) \|_{H^{\frac{1}{2}+}} \| \partial_{x}^{2} \Phi_{k'}^{-1}(u_{1}(t),t) \|_{H^{\frac{1}{2}+}}
\left\| P \left( u_{1}(t),t \right) \right\|_{H^{\frac{1}{2}+}} \| v_{1}(t) \|_{H^{k'-6}}  \\
& + ...... + \\
& \| \Phi_{k'}(u_{2}(t),t) \|_{H^{k'-6}}  \| \partial_{x}^{2} \Phi_{k'}^{-1}(u_{2}(t),t) \|_{H^{\frac{1}{2}+}}
\left\| P \left( u_{2}(t),t \right) \right\|_{H^{\frac{1}{2}+}} \| \triangle v (t) \|_{H^{\frac{1}{2}+}} + ... + \\
&  \| \Phi_{k'}(u_{2}(t),t) \|_{H^{\frac{1}{2}+}}  \| \partial_{x}^{2} \Phi_{k'}^{-1}(u_{2}(t),t) \|_{H^{\frac{1}{2}+}}
\left\| P \left( u_{2}(t),t \right) \right\|_{H^{\frac{1}{2}+}} \| \triangle v (t) \|_{H^{k'-6}}, and
}

\EQQARR{
Z_{c} := & \| \triangle  \Phi_{k'}(t) \|_{H^{k'-6}} \| \partial_{x} \Phi_{k'}^{-1}(u_{1}(t),t) \|_{H^{\frac{1}{2}+}}
\left\| \tilde{Q} \left( u_{1}(t),t \right) \right\|_{H^{\frac{1}{2}+}}
\| v_{1}(t) \|_{H^{\frac{1}{2}+}} + ... + \\
& \| \triangle \Phi_{k'}(t) \|_{H^{\frac{1}{2}+}} \| \partial_{x} \Phi_{k'}^{-1}(u_{2}(t),t) \|_{H^{\frac{1}{2}+}}
\left\|  \tilde{Q} \left( u_{2}(t),t \right) \right\|_{H^{\frac{1}{2}+}} \| v_{1}(t) \|_{H^{k'-6}}  \\
& + ...... + \\
& \| \Phi_{k'}(u_{2}(t),t) \|_{H^{k'-6}}  \| \partial_{x} \Phi_{k'}^{-1}(u_{2}(t),t) \|_{H^{\frac{1}{2}+}}
\left\| \tilde{Q} \left( u_{2}(t),t \right) \right\|_{H^{\frac{1}{2}+}} \| \triangle v (t) \|_{H^{\frac{1}{2}+}} + ... + \\
&  \| \Phi_{k'}(u_{2}(t),t) \|_{H^{\frac{1}{2}+}}  \| \partial_{x} \Phi_{k'}^{-1}(u_{2}(t),t) \|_{H^{\frac{1}{2}+}}
\left\| \tilde{Q} \left( u_{2}(t),t \right) \right\|_{H^{\frac{1}{2}+}} \| \triangle v (t) \|_{H^{k'-6}} \cdot
}
Hence

\EQQARRLAB{
\left\| \triangle \bar{a}_{0}(t) \right\|_{H^{k'-6}} \lesssim   \left\| \triangle u(t) \right\|_{H^{k'}}  \cdot
\label{Eqn:EstDiffBara0}
}
We now estimate  $ \| \triangle \tilde{a}_{0}(t) \|_{H^{k'-6}} $. Lemma \ref{Lem:prod}, Result \ref{Res:aux}, and
$(*):= A_{1} B_{1} - A_{2} B_{2} = (A_{1} - A_{2}) B_{1} + A_{2} (B_{1} - B_{2})$ yield
$ \| \triangle S_{1} v(t) \|_{H^{k'-6}} \lesssim  \| \triangle u(t) \|_{H^{k'}}$. Next we  use (\ref{Eqn:ExpIntBasic}). Let $I(u(t),t): S_{1} (\overrightarrow{u(t)},t) \Phi_{k'}^{-1} (u(t),t) \partial_{x} v(t) + ...
+ S_{4}(u(t),t)$. Hence $\tilde{a} (u(t),t) = v(t) \int_{0}^{x} I(u(t),t) \; dx' $. We have
$  \left\| \triangle \int_{0}^{x} I (t) \; dx' \right\|_{H^{k'-6}} \lesssim   \| \triangle I (t) \|_{H^{k'-7}} $
and $  \left\| \triangle \int_{0}^{x} I (t) \; dx' \right\|_{H^{\frac{1}{2}+}} \lesssim   \| \triangle I (t) \|_{L^{2}} $.
We apply several times Lemma \ref{Lem:prod}; we use $(*)$,
$A_{1} B_{1} C_{1} - A_{2} B_{2} C_{2} = (A_{1} - A_{2}) B_{1} C_{1} + A_{2} (B_{1} - B_{2}) C_{1}
+ A_{2} B_{2} (C_{1} - C_{2}) $, $ \epsilon_{1} A_{1} B_{1} C_{1} - \epsilon_{2} A_{2} B_{2} C_{2}  =
\epsilon_{1} (A_{1} B_{1} C_{1}- A_{2} B_{2} C_{2}) + (\epsilon_{1} - \epsilon_{2}) A_{2} B_{2} C_{2} $, and
the estimate $ \| \triangle u(t) \|_{H^{r}} \lesssim  \| u_{1}(t) \|_{H^{r}} + \| u_{2}(t) \|_{H^{r}} $ for $r \geq 0$ to get

\EQQARR{
\left\| \triangle \tilde{a}_{0}(t) \right\|_{H^{k'-6}} \lesssim \sum \limits_{j=1}^{2} \epsilon_{j} \left( 1 + \| u_{j}(t) \|_{k'+1} \right)
+ \left\| \triangle u(t) \right\|_{H^{k'}} \cdot
}
Hence

\EQQARR{
\| \triangle a_{0}(u(t),t) \|_{H^{k'-6}} \lesssim \sum \limits_{j=1}^{2} \epsilon_{j} \left( 1 + \| u_{j}(t) \|_{k'+1} \right)
+ \left\| \triangle u(t) \right\|_{H^{k'}} \cdot
\label{Eqn:Diffa0}
}
We see from the above estimates that

\EQQARR{
\sum \limits_{m=1}^{3} \| \triangle a_{m}(t) \|^{2}_{H^{k'-6}}  \| v_{2}(t) \|^{2}_{H^{ \left( m+ \frac{1}{2} \right)+}}
& \lesssim \left\|  \triangle u(t) \right\|^{2}_{H^{k'}} E_{\frac{19}{2}+} \left( u_{2}(t), t \right), \; \text{and}
}

\EQQARR{
\sum \limits_{m=1}^{3}  \| \triangle a_{m}(t) \|^{2}_{H^{\frac{1}{2}+}} \| v_{2}(t) \|^{2}_{H^{m+k'-6}} \\
\lesssim \left\|  \triangle u(t) \right\|^{2}_{H^{\frac{7}{2}+}} E_{k'+3} \left( u_{2}(t), t \right) +
\left\| \triangle u (t) \right\|^{2}_{H^{\frac{9}{2}+}} E_{k'+2} \left( u_{2}(t), t \right) \\
+ \left\| \triangle u(t) \right\|^{2}_{H^{\frac{13}{2}+}} E_{k'+1} \left( u_{2}(t), t \right) \cdot
}
In the sequel we use implicitly Lemma \ref{Lem:prod}, Proposition \ref{prop:compEHnorm}, Lemma \ref{lem:gauge},
and the claim in the proof of Lemma \ref{lem:OneEnergyEst} in order to prove the estimates below. We have

\EQQARR{
\epsilon_{2}^{2} \| c_{2}(u_{2}(t),t) \|^{2}_{H^{k'-2}} \| v_{2}(t) \|_{H^{\frac{1}{2}+}}^{2}
& \lesssim \epsilon_{2}^{2} \left( 1 + E_{k'+1}( u_{2}(t),t ) \right),  \\
}

\EQQARR{
\epsilon_{2}^{2} \left( \| c_{2}(u_{2}(t),t) \|_{H^{\frac{9}{2}+}} \| v_{2}(t) \|^{2}_{H^{k'-6}} + \| v_{2}(t) \|^{2}_{H^{k'-2}} \right)
& \lesssim \epsilon_{2}^{2} E_{k'+4}(u_{2}(t),t),
}

\EQQARR{
\epsilon_{2}^{2} \sum \limits_{j=1}^{2} \sum \limits_{m=1}^{3} \| b_{m,j}(u_{j}(t),t) \|^{2}_{H^{k'-6}} \| v_{2}(t) \|^{2}_{H^{ \left( m + \frac{1}{2} \right)+}}
& \lesssim \epsilon_{2}^{2} E_{\frac{19}{2}+} ( u_{2}(t),t ),
}

\EQQARR{
\epsilon_{2}^{2} \sum \limits_{j=1}^{2} \sum \limits_{m=1}^{3} \| b_{m,j}(u_{j}(t),t) \|^{2}_{H^{\frac{1}{2}+}} \| v_{2}(t) \|^{2}_{H^{k'- 6 + m}}
& \lesssim \epsilon_{2}^{2} E_{k'+3} (u_{2}(t),t),
}

\EQQARR{
\epsilon_{1} \left( \sum \limits_{m=1}^{3} \| b_{m,1}(u_{1}(t),t) \|_{H^{ \left( m + \frac{1}{2} \right)+}} + \| c_{1}(u_{1}(t),t) \|_{H^{\frac{9}{2}+}} \right)
\| \triangle v(t) \|^{2}_{H^{k'-6}}
\lesssim  \epsilon_{1} E_{k'} \left( u_{1}(t), u_{2}(t), t \right),
}

\EQQARR{
\epsilon_{1}^{2} \left( \| \triangle c(t) \|^{2}_{H^{k'-2}} \| v_{2}(t) \|^{2}_{H^{\frac{1}{2}+}}
+  \| \triangle c(t) \|^{2}_{H^{\frac{9}{2}+}} \| v_{2}(t) \|^{2}_{H^{k'-2}}  \right) \\
\lesssim  \epsilon_{1}^{2} \left(
E_{k'+1} \left( u_{1}(t), t \right) +   E_{k'+4} \left( u_{2}(t) ,t \right)
\right),
}

\EQQARR{
\epsilon_{1}^{2} \| b_{0,1}(u_{1}(t),t) \|^{2}_{H^{k'-6}} + \epsilon_{2}^{2} \| b_{0,2} (u_{2}(t),t) \|^{2}_{H^{k'-6}}
& \lesssim  \epsilon_{1}^{2} \left( 1 + E_{k'}^{\frac{1}{2}} \left( u_{1}(t) ,t \right) \right)^{2}
+ \epsilon_{2}^{2} \left( 1 + E_{k'}^{\frac{1}{2}} \left(  u_{2}(t),t \right) \right)^{2} \lesssim \epsilon_{2}^{2},
}

\EQQARR{
\sum \limits_{m=2}^{3}  \| a_{m,1}(u_{1}(t),t) \|_{H^{ \left( m -  \frac{3}{2} \right)+ }}  \| \triangle \p_{x} v(t) \|^{2}_{L^{2}}
& \lesssim E_{k'} \left( u_{1}(t), u_{2}(t) ,t \right) \cdot
}
Then we use similar estimates that appear between  below (\ref{Eqn:Beg}) and ` by the term
$\epsilon \| D^{k'-4} v(t) \|^{2}_{L^{2}}$ of (\ref{Eqn:EstHv})' below (\ref{Eqn:Meth1}). We have

\EQQARR{
 \left( k' - \frac{15}{2} \right) \p_{x} a_{3,1}(t) + a_{2,1} (t) & = \p_{\omega_{3}} F \left( \overrightarrow{u_{1}(t)}, t \right)
 \left[\frac{P \left( u_{1}(t), t \right)}{ _{\p_{\omega_{3}} F \left( \overrightarrow{u_{1}(t)}, t \right)} } \right]_{ave},
}

\EQQARR{
\left( 1 + \sum \limits_{m=1}^{3}  \| a_{m,1}(t) \|_{H^{ \left( m + \frac{1}{2} \right)+}} \right) \| \triangle v(t) \|^{2}_{H^{k'-6}} & \lesssim
E_{k'} \left( u_{1}(t),u_{2}(t), t \right),
}

\EQQARR{
\sum \limits_{m=1}^{3} \| a_{m,1}(t) \|^{2}_{H^{k'-6+m}} \| \triangle v(t) \|^{2}_{H^{\frac{1}{2}+}}
& \lesssim  E_{k'} \left( u_{1}(t), u_{2} (t) , t \right),
}

\EQQARR{
\epsilon_{1} \left(  \| b_{3,1} (t) \|_{H^{\frac{3}{2}+}} + \| b_{2,1}(t) \|_{H^{\frac{1}{2}+}} \right)
\| \triangle \p_{x} D^{k'-6} v(t) \|^{2}_{L^{2}}
& \lesssim \epsilon_{1}^{2} \| \triangle D^{k'-4} v (t) \|^{2}_{L^{2}} + E_{k'} \left( u_{1}(t), u_{2}(t), t \right), \; \text{and}
}

\EQQARR{
\epsilon_{1}^{2} \left( \sum \limits_{m=1}^{3} \| b_{m,1}(t) \|^{2}_{H^{k'-6+m}} +  \| c_{1}(t) \|^{2}_{H^{k'-2}} \right)
\| \triangle v(t) \|^{2}_{H^{\frac{1}{2}+}} \\
\lesssim \epsilon_{1}^{2} \left( 1 + E_{k'}(u_{1}(t),t) + \| v_{1}(t) \|^{2}_{\dot{H}^{k'-4}} \right) \left\|  \triangle u (t) \right\|^{2}_{H^{\frac{13}{2}+}} \\
\lesssim  \epsilon_{1}^{2} \left( 1 + E_{k'+2} (u_{1}(t),t) \right) \cdot
}
Combining all the estimates above, we see from Propositions \ref{prop:linearbarv} and \ref{prop:linearbarvL2} that Lemma \ref{lem:EnergyDiffEst} holds,
taking into account that the term $\epsilon_{1}^{2} \left\| \triangle D^{k'-4} v(t) \right\|^{2}_{L^{2}}$ can be absorbed by the term
$2 \epsilon_{1} \left\| \triangle D^{k'-4} v(t) \right\|^{2}_{L^{2}}$ of (\ref{Eqn:Estbarv}).

\subsubsection{Proof of Result \ref{Res:aux}}

We only prove (\ref{Eqn:EstSutDiff}). The proof of  (\ref{Eqn:EstSubartDiff}) is similar and is left as an exercise to the reader. \\
\\
One can find $I_0$,...,$I_6$, $J_{0}$, $J_1$, and $J_2$ finite subsets of $\{1,2,... \} \times ... \times \{ 1,2,... \} $ such that for all
$ \vec{i} := (i_0,...,i_6) \in I_0 \times... \times I_6 $ and for all $\vec{j} := (j_0,j_1,j_2) \in J_0 \times J_1 \times J_2 $, one can find
$a_{\vec{i},\vec{j}} \left( \overrightarrow{u_{q}(t)}, t \right) \in \mathbb{R}$  such that $a_{\vec{i},\vec{j}} \left( \overrightarrow{u_{q}(t)}, t \right)$ is a constant multiplied by a finite product of derivatives of $F$ and the inverse of $\p_{\omega_3} F$ evaluated at $\left( \overrightarrow{u_{q}(t)},t \right)$  and

\EQQARR{
S \left( u_{q}(t),t \right) & = \sum \limits_{ \substack{\vec{i} \in I_{0} \times ... \times I_{6} \\ \vec{j} \in J_0 \times J_1 \times J_2  } }
a_{\vec{i},\vec{j}} \left( \overrightarrow{u_{q}(t)},t  \right)  (\p_{x}^{6} u_{q}(t))^{i_6}... (u_{q}(t))^{i_0}
( \p_{t} \p_{x}^{2} u_{q}(t) )^{j_2} ( \p_{t} \p_{x} u_{q}(t) )^{j_1} (\p_{t} u_{q}(t))^{j_0} \cdot
}
We prove two claims.

\underline{Claim}: Let $p \in \{ 0,..,6 \}$. Let $\beta \in \{ 1,2,... \}$. Then

\EQQARRLAB{
\| \triangle  \left( \p_{x}^{p} u \right)^{\beta} (t) \|_{H^{r}} & \lesssim  \| \triangle u(t) \|_{H^{\max \left(  r + p ,  \left( p + \frac{1}{2} \right)+ \right) }}
\label{Eqn:pxpu}
}

\begin{proof}

If $\beta \in \{1,2,...\}$ elementary algebraic identities show that
there exists a polynomial $P$ in  $\p_{x}^{p} u_{1}(t)$,...,$u_{1}(t)$,  $\p_{x}^{p} u_{2}(t)$,...,$u_{2}(t)$  such that

\EQQARR{
\triangle \left( \p_{x}^{p} u \right)^{\beta}(t) :=  P \left( \p_{x}^{p} u_{1}(t),...,u_{1}(t), \p_{x}^{p} u_{2}(t),...,u_{2}(t) \right) \triangle \p_{x}^{p} u(t) \cdot
}
Taking into account Proposition \ref{prop:compEHnorm},  applying several times Lemma \ref{Lem:prod} and proceeding similarly as the end of the proof of (\ref{Eqn:EstSut}) we see that there exists $ \alpha > 0$  such that

\EQQARR{
\| P \left( \p_{x}^{p} u_{1}(t),.., u_{1}(t), \p_{x}^{p} u_{2}(t),...,u_{2}(t) \right) \|_{H^{r}} & \lesssim
\left( \| u_{1}(t) \|_{H^{r+p}} + \| u_{2}(t) \|_{H^{r+p}} \right) \\
&
\langle  \| u_{1}(t) \|_{H^{ \left( p + \frac{1}{2} \right)+}} \rangle^{\alpha}
\langle \| u_{2}(t) \|_{H^{\left( p + \frac{1}{2}  \right)+ }} \rangle^{\alpha}  \\
& \lesssim \| u_{1}(t) \|_{H^{r+p}} + \| u_{2}(t) \|_{H^{r+p}} \cdot
}
Hence

\EQQARR{
\| \triangle  \left( \p_{x}^{p} u \right)^{\beta}(t) \|_{H^{r}} & \lesssim
\| P \left( \p_{x}^{p} u_{1}(t),...,u_{1}(t), \p_{x}^{p} u_{2}(t),...,u_{2}(t) \right) \|_{H^{r}}
\| \triangle \p_{x}^{p} u(t) \|_{H^{\frac{1}{2}+}} \\
& +   \| P (  \p_{x}^{p} u_{1}(t),...,u_{1}(t), \p_{x}^{p} u_{2}(t),...,u_{2}(t) ) \|_{H^{\frac{1}{2}+}} \| \triangle \p_{x}^{p} u(t) \|_{H^{r}} \\
& \lesssim  \left( \| u_{1}(t) \|_{H^{r+p}} + \| u_{2}(t) \|_{H^{r+p}} \right)  \| \triangle \p_{x}^{p}  u(t) \|_{H^{ \frac{1}{2}+ }} \\
& + \left( \| u_{1}(t) \|_{H^{ \left( p + \frac{1}{2} \right)+}} + \| u_{2}(t) \|_{H^{ \left( p + \frac{1}{2} \right)+}}  \right)  \| \triangle \p_{x}^{p} u(t) \|_{H^{r}} \\
& \lesssim  \| \triangle u(t) \|_{H^{\max \left(  r + p ,  \left( p + \frac{1}{2} \right)+ \right) }} \cdot
}

\end{proof}

\underline{Claim}: Let $m \in \{ 0,1,2 \}$. Let $\beta \in \{ 1,2,... \} $. Then we have

\EQQARRLAB{
\| \triangle  \left( \p_{t} \p_{x}^{m} u \right)^{\beta}(t) \|_{H^{r}} & \lesssim
\epsilon_{1}  \| \triangle u(t) \|_{H^{\max \left(  \left( \frac{9}{2} + m \right)+, r + m + 4 \right)}} \\
& + ( \epsilon_{2} - \epsilon_{1} ) \| u_{2}(t) \|_{H^{\max \left(  \left(\frac{9}{2} + m \right)+, r + m + 4 \right)}}
+ \| \triangle u_(t) \|_{H^{\max \left( r + m + 3, \left( \frac{7}{2} + m \right)+ \right)}}
\label{Eqn:ptpxmu}
}

\begin{proof}
(\ref{Eqn:Regul}) yields $ \triangle  \p_{t} \p_{x}^{m} u (t)  =  -\epsilon_{1}  \triangle  \p_{x}^{4+m} u(t) - ( \epsilon_{1} - \epsilon_{2} ) \p_{x}^{4+m} u_{2}(t)
+ \triangle \p_{x}^{m} F (t)$. Lemma \ref{lem:EstDerivF} and Proposition \ref{prop:compEHnorm} yield the claim with $\beta = 1$. \\
Let $\beta > 1$. Elementary algebraic properties show that

\EQQARR{
\triangle  \left( \p_{t} \p_{x}^{m} u \right)^{\beta}(t) & = P \left( \p_{t} \p_{x}^{m} u_{1}(t),...,u_{1}(t), \p_{t}\p_{x}^{m}u_{2}(t),..,u_{2}(t) \right) \triangle \p_{t} \p_{x}^{m} u(t),
}
with $P$ a polynomial in $\p_{t} \p_{x}^{m} u_{1}(t)$,..., $u_{1}(t)$, $\p_{t} \p_{x}^{m} u_{2}(t)$,..., and $u_{2}(t)$. Proceeding similarly as in the proof of (\ref{Eqn:EstSut}) we get

\EQQARR{
 \left\|  P \left( \p_{t} \p_{x}^{2} u_{1}(t),...,u_{1}(t), \p_{t}\p_{x}^{2}u_{2}(t),..,u_{2}(t) \right) \right\|_{H^{r}} & \lesssim
 \langle \| u_{1}(t) \|_{H^{r+6}} \rangle +  \langle \| u_{2}(t) \|_{H^{r+6}} \rangle  \cdot
}
Hence by application of Lemma \ref{Lem:prod}

\EQQARR{
\| \triangle \left( \p_{t} \p_{x}^{m} u \right)^{\beta}(t) \|_{H^{r}} & \lesssim
\left\|  P \left( \p_{t} \p_{x}^{2} u_{1}(t),...,u_{1}(t), \p_{t}\p_{x}^{2}u_{2}(t),..,u_{2}(t) \right) \right\|_{H^{r}}
\| \triangle \p_{t} \p_{x}^{m} u(t) \|_{H^{\frac{1}{2}+}}  \\
& + \left\|  P \left( \p_{t} \p_{x}^{2} u_{1}(t),...,u_{1}(t), \p_{t}\p_{x}^{2}u_{2}(t),..,u_{2}(t) \right) \right\|_{H^{\frac{1}{2}+}}
\| \triangle \p_{t} \p_{x}^{m} u(t) \|_{H^{r}} \\
& \lesssim  R.H.S \; \text{of} \; (\ref{Eqn:ptpxmu}) \cdot
}

\end{proof}

We continue the proof of (\ref{Eqn:EstSutDiff}). We have $ \triangle  S (t)  := \triangle S_{0} (t) + \sum \limits_{m=0}^{6} \triangle S_{m} (t) + \sum \limits_{m=0}^{2} \triangle \bar{S}_{m}(t) $ with  \footnote{ In (\ref{Eqn:DeltaSm}) we use the following convention: we ignore the terms that do not make sense for some values of $m$ ( such as $ \left( \p_{x}^{7-m} u_{2}(t) \right)^{i_{7}- m}$ if $ m= 0$ ) }

\EQQARRLAB{
\triangle S_{0}(t) & := \sum \limits_{ \substack{\vec{i} \in I_{0} \times ... \times I_{6} \\ \vec{j} \in J_0 \times J_1 \times J_2  } }
 \triangle a_{\vec{i},\vec{j}} (t) (\p_{x}^{6} u_{1}(t))^{i_6}... (u_{1}(t))^{i_0} ( \p_{t} \p_{x}^{2} u_{1}(t) )^{j_2}
 ( \p_{t} \p_{x} u_{1}(t) )^{j_1} (\p_{t} u_{1}(t))^{j_0},
 \label{Eqn:DeltaS0}
}

\EQQARRLAB{
\triangle S_{m}(t) & := \sum \limits_{ \substack{\vec{i} \in I_{0} \times ... \times I_{6} \\ \vec{j} \in J_0 \times J_1 \times J_2  } }
\left[
\begin{array}{l}
a_{\vec{i},\vec{j}} \left( \overrightarrow{u_{2}(t)},t \right) (\p_{x}^{6} u_{2}(t))^{i_6} ... (\p_{x}^{7-m} u_{2}(t))^{i_{7-m}}
\triangle \left( \p_{x}^{6-m} u(t) \right)^{i_{6-m}} \\
 (\p_{x}^{5-m} u_{1}(t))^{i_{5-m}}...(u_{1}(t))^{i_0} ( \p_{t} \p_{x}^{2} u_{1}(t) )^{j_2} ( \p_{t} \p_{x} u_{1}(t) )^{j_1} (\p_{t} u_{1}(t))^{j_0}
\end{array}
\right], \; \text{and}
\label{Eqn:DeltaSm}
}

\EQQARR{
\triangle \bar{S}_{m}(t) & := \sum \limits_{ \substack{\vec{i} \in I_{0} \times ... \times I_{6} \\ \vec{j} \in J_0 \times J_1 \times J_2  } }
Q_{\vec{i},\vec{j}}(t) c_{m,\vec{j}}(t),
}
with $ Q_{\vec{i},\vec{j}}(t)  := a_{\vec{i},\vec{j}} \left( \overrightarrow{u_{2}(t)},t \right) (\p_{x}^{6} u_{2}(t))^{i_6}... (u_{2}(t))^{i_0}$,
$ c_{0,\vec{j}} (t)  := \triangle  ( \p_{t} \p_{x}^{2} u (t) )^{j_2} (\p_{t} \p_{x} u_{1}(t) )^{j_1} (\p_{t} u_{1}(t))^{j_0}$,
$ c_{1,\vec{j}} (t) := ( \p_{t} \p_{x}^{2} u_{2} (t) )^{j_2} \triangle ( \p_{t} \p_{x} u (t) )^{j_1}  ( \p_{t} u_{1}(t))^{j_0}$, and
$ c_{2,\vec{j}} (t)  := ( \p_{t} \p_{x}^{2} u_{2}(t) )^{j_2} \left( \p_{t} \p_{x} u_{2}(t) \right)^{j_1}  \triangle ( \p_{t} u(t))^{j_0} $.  \\
Hence $\triangle S_{0}(t)$, $\triangle S_{m}(t)$, and $\triangle \bar{S}_{m}(t)$ are finite sums made of terms that are uniquely characterized
by $\vec{i}$ and $\vec{j}$.   \\
\\
In the sequel let $\alpha$ be a constant of which the value is allowed to change from one line to the other one and even within the same line and
such all the statements below are true. \\
\\
\underline{Claim}: Let $r$ such that $0 \leq  r \leq k'-3 $. Then

\EQQARRLAB{
\left\| a_{\vec{i},\vec{j}} \left( \overrightarrow{u_{q}(t)}, t \right) \right\|_{H^{r}} & \lesssim  1, \; \text{and}
\label{Eqn:avecij}
}

\EQQARRLAB{
\left\| \triangle a_{\vec{i},\vec{j}} (t) \right\|_{H^{r}} & \lesssim  \| \triangle u(t) \|_{H^{\max{ \left( r + 3, \frac{7}{2}+ \right)}}} \cdot
\label{Eqn:Diffavecij}
}

\begin{proof}

In the sequel we use Lemma \ref{lem:EstDerivOp}, Lemma \ref{lem:EstDerivF}, Lemma \ref{Lem:prod} more than once if necessary, and Proposition \ref{prop:compEHnorm}.
We can write $a_{\vec{i},\vec{j}} \left( \overrightarrow{u_{q}(t)},t \right)  =  C_{\vec{i},\vec{j}} \frac{R ( \overrightarrow{u_{q}(t)} ,t)}
{ \left( \p_{\omega_{3}} F(\overrightarrow{u_{q}(t)} ,t) \right)^{\alpha}} $
with $C_{\vec{i},\vec{j}} \in \mathbb{R}$ and  $ R( \overrightarrow{u_{q}(t)}, t )$  a product of derivatives of $F$ evaluated at $ \left( \overrightarrow{u_{q}}(t), t \right) $. Clearly (\ref{Eqn:avecij}) holds. We have $ \triangle a_{\vec{i},\vec{j}} (t)  = X + Y $ with

\EQQARR{
X & := \frac{ R ( \overrightarrow{u_{1}(t)},t ) - R (\overrightarrow{u_{2}(t)},t ) } { \left( \p_{\omega_{3}} F ( \overrightarrow{u_{1}(t)},t ) \right)^{\alpha} }, \; \text{and} \\
Y & :=
\frac{
R( \overrightarrow{u_{2}(t)}, t)
\left( \left( \p_{\omega_{3}} F ( \overrightarrow{u_{2}(t)}, t ) \right)^{\alpha} -  \left( \p_{\omega_{3}} F (\overrightarrow{u_{1}(t)},t ) \right)^{\alpha} \right)
}
{ \left( \p_{\omega_{3}} F ( \overrightarrow{u_{2}(t)},t)  \right)^{\alpha}  \left( \p_{\omega_{3}} F ( \overrightarrow{u_{1}(t)},t)  \right)^{\alpha}  } \cdot
}
We have

\EQQARR{
\| X \|_{H^{r}} & \lesssim \| R (\overrightarrow{u_{1}(t)},t) - R(\overrightarrow{u_{2}(t)},t) \|_{H^{\max \left( r , \frac{1}{2}+ \right)}} \\
& \lesssim \| \triangle u(t)  \|_{H^{\max \left( r + 3 , \frac{7}{2}+ \right)}} \cdot
}
We have

\EQQARR{
\| Y \|_{H^{r}} & \lesssim
\| R(\overrightarrow{u_{2}(t)},t) \left( \left( \p_{\omega_{3}} F( \overrightarrow{u_{2}(t)},t) \right)^{\alpha} - \left( \p_{\omega_{3}} F( \overrightarrow{u_{1}(t)},t) \right)^{\alpha}  \right) \|_{H^{\max \left( r, \frac{1}{2}+ \right) }} \\
  & \lesssim  \| R(\overrightarrow{u_{2}(t)},t) \|_{H^{\max \left( r, \frac{1}{2}+ \right)}}
  \left\| \left( \p_{\omega_{3}} F( \overrightarrow{u_{2}(t)},t) \right)^{\alpha} - \left( \p_{\omega_{3}} F( \overrightarrow{u_{1}(t)},t) \right)^{\alpha} \right\|_{H^{\frac{1}{2}+}} \\
  & + \| R(\overrightarrow{u_{2}(t)},t) \|_{H^{\frac{1}{2}+}}
  \left\|   \left( \p_{\omega_{3}} F (\overrightarrow{u_{2}(t)},t) \right)^{\alpha} - \left( \p_{\omega_{3}} F(\overrightarrow{u_{1}(t)},t) \right)^{\alpha}  \right\|_{H^{\max \left( r, \frac{1}{2}+ \right)}} \\
  & \lesssim \| \triangle u(t) \|_{H^{\max \left( r + 3, \frac{7}{2}+ \right)}}
  }
Hence (\ref{Eqn:Diffavecij}) holds.

\end{proof}

Let $Z_{\vec{i},\vec{j}}(t)$ be the term of the sum  $\triangle S_{0}(t)$ determined by $\vec{i}$ and $\vec{j}$. We see
from Lemma \ref{Lem:prod} that

\EQQARRLAB{
\| Z_{\vec{i},\vec{j}}(t) \|_{H^{r}} & \lesssim \left\| \triangle a_{\vec{i},\vec{j}} (t)  \right\|_{H^{r}}
\langle \| u_{1}(t) \|_{H^{\frac{13}{2}+}} \rangle^{\alpha}
 \sup_{m \in \{ 0,1,2 \} } \langle \| \p_{t} \p_{x}^{m} u_{1}(t) \|_{H^{\frac{1}{2}+}} \rangle^{\alpha} \\
 & + \left\| \triangle a_{\vec{i},\vec{j}}  (t)  \right\|_{H^{\frac{1}{2}+}}
 \| u_{1}(t) \|_{H^{6+r}}  \langle \| u_{1}(t) \|_{H^{\frac{13}{2}+}} \rangle^{\alpha}
 \sup_{ m \in \{ 0,1,2 \} } \langle \| \p_{t} \p_{x}^{m} u_{1}(t) \|_{H^{\frac{1}{2}+}} \rangle^{\alpha} \\
 & + \left\| \triangle a_{\vec{i},\vec{j}}  (t)  \right\|_{H^{\frac{1}{2}+}}
  \langle \| u_{1}(t) \|_{H^{\frac{13}{2}+}} \rangle^{\alpha}
  \sup_{ m \in \{ 0,1,2 \} } \| \p_{t} \p_{x}^{m} u_{1}(t) \|_{H^{r}} \\
 &
  \sup_{ m \in \{ 0,1,2 \} } \langle \| \p_{t} \p_{x}^{m} u_{1}(t) \|_{H^{\frac{1}{2}+}} \rangle^{\alpha} \cdot
\label{Eqn:TriangleS0}
}
Hence we see from (\ref{Eqn:EstSut}) and (\ref{Eqn:Diffavecij}) that

\EQQARR{
\| \triangle S_{0}(t) \|_{H^{r}} & \lesssim \| \triangle u(t) \|_{H^{\max{\left( r+3 , \frac{7}{2} + \right)}}} \cdot
}
Let $Z_{\vec{i},\vec{j}}(t)$ be the term of the sum $\triangle S_{m}(t)$ determined by $\vec{i}$ and $\vec{j}$. We have

\EQQARR{
\| Z_{\vec{i},\vec{j}}(t) \|_{H^{r}} & \lesssim \| a_{\vec{i},\vec{j}}  ( \overrightarrow{u_{2}(t)},t ) \|_{H^{r}}
\langle \| u_{1}(t) \|_{H^{\frac{13}{2}+}} \rangle^{\alpha} \langle \| u_{2}(t) \|_{H^{\frac{13}{2}+}} \rangle^{\alpha}  \\
& \sup_{m \in [0,..,6]}  \| \triangle (\p_{x}^{6-m} u)^{i_{6}-m}(t) \|_{H^{\frac{1}{2}+}}
\sup_{m \in \{ 0,1,2 \} } \langle \| \p_{t} \p_{x}^{m} u_{1}(t) \|_{H^{\frac{1}{2}+}} \rangle^{\alpha} \\
& + \| a_{\vec{i},\vec{j}}   ( \overrightarrow{u_{2}(t)},t) \|_{H^{\frac{1}{2}+}} \langle \| u_{1}(t) \|_{H^{\frac{13}{2}+}} \rangle^{\alpha} \\
& \sup_{m \in [0,..,6]}  \| \triangle (\p_{x}^{6-m} u)^{i_{6}-m}(t) \|_{H^{r}}
 \sup_{m \in \{ 0,1,2 \} } \langle \| \p_{t} \p_{x}^{m} u_{1}(t) \|_{H^{\frac{1}{2}+}} \rangle^{\alpha} \\
 & + \| a_{\vec{i},\vec{j}} ( \overrightarrow{u_{2}(t)},t) \|_{H^{\frac{1}{2}+}} \sum \limits_{q \in \{ 1,2 \}} \| u_{q}(t) \|_{H^{r}}
 \langle \| u_{1}(t) \|_{H^{\frac{13}{2}+}} \rangle^{\alpha} \langle \| u_{2}(t) \|_{H^{\frac{13}{2}+}} \rangle^{\alpha}  \\
&  \sup_{m \in [0,..,6] } \| \triangle ( \p_{x}^{6-m} u )^{i_{6} -m}(t) \|_{H^{\frac{1}{2}+}}
 \sup_{m \in \{ 0, 1,2 \}} \langle  \| \p_{t} \p_{x}^{m} u_{1}(t) \|_{H^{\frac{1}{2}+}} \rangle^{\alpha} \\
& + \| a_{\vec{i},\vec{j}} (\overrightarrow{u_{2}(t)},t) \|_{H^{\frac{1}{2}+}} \langle  \| u_{1}(t) \|_{H^{\frac{13}{2}+}} \rangle^{\alpha}
\langle \| u_{2}(t) \|_{H^{\frac{13}{2}+}} \rangle^{\alpha} \\
& \sup_{m \in \{ 0,1,2 \} } \| \p_{t} \p_{x}^{m} u_{1}(t) \|_{H^{r}}
\langle \sup_{m \in \{ 0,1,2 \} } \| \p_{t} \p_{x}^{m} u_{1}(t) \|_{H^{\frac{1}{2}+}} \rangle^{\alpha} \cdot
\label{Eqn:EstZij}
}
Using also (\ref{Eqn:pxpu}) and (\ref{Eqn:avecij})  we get

\EQQARR{
\| \triangle S_{m}(t) \|_{H^{r}} & \lesssim \| \triangle u (t) \|_{H^{\max \left( r +6, \frac{13}{2}+  \right) }}
}
Let $Z_{\vec{i},\vec{j}}(t)$ be the term of $\triangle \bar{S}_{m}(t)$ defined by
$\vec{i}$ and $\vec{j}$. We have

\EQQARR{
 \| Z_{\vec{i},\vec{j}}(t) \|_{H^{r}} & \lesssim  \| Q_{\vec{i},\vec{j}} (t) \|_{H^{r}}   \| c_{m,\vec{j}} (t) \|_{H^{\frac{1}{2}+}}
+ \| Q_{\vec{i},\vec{j}}  (t) \|_{H^{\frac{1}{2}+}} \| c_{m,\vec{j}} (t) \|_{H^{r}}
}
We have

\EQQARR{
\| Q_{\vec{i},\vec{j}} (t) \|_{H^{r}} & \lesssim
\langle \| u_{2}(t) \|_{H^{\frac{13}{2}+}} \rangle^{\alpha}
\left(
\begin{array}{l}
\| a_{\vec{i},\vec{j}}( \overrightarrow{u_{2}(t)},t ) \|_{H^{r}}
+ \| a_{\vec{i},\vec{j}}( \overrightarrow{u_{2}(t)},t ) \|_{H^{\frac{1}{2}+}}
\| u_{2}(t) \|_{H^{6+r}}
\end{array}
\right)  \lesssim 1 \cdot
}
There exist polynomials $S_{i} (u_{2}(t),t)$  for $i \in \{ 1,2 \}$  in $u_{i}(t)$, $\p_{x} u_{i} (t)$,...,$\p_{x}^{6} u_{i}(t)$, $\p_{t} \p_{x}^{2} u_{i}(t)$,
$\p_{t} \p_{x} u_{i}(t) $, derivatives of $F$ evaluated  at $ \left( \overrightarrow{u_{i}(t)}, t \right) $ , and
$ \left( \p_{\omega_{3}} F \left( \overrightarrow{u_{i}(t)}, t \right) \right)^{-1}$ such that
$ c_{m,\vec{j}}(t) := \triangle \left( \p_{t} \p_{x}^{2} u(t) \right)^{j_{m}} S_{1} \left(  u_{1}(t) ,t \right)
S_{2} \left(  u_{2}(t) ,t \right) $. (\ref{Eqn:ptpxmu}) yields

\EQQARR{
\| c_{m,\vec{j}} (t) \|_{H^{r}}  \\
\lesssim \| \triangle \left( \p_{t} \p_{x}^{m} u \right)^{j_{m}}(t) \|_{H^{r}}
\| S_{2} ( u_{2}(t), t ) \| _{H^{\frac{1}{2}+}} \| S_{1} ( u_{1}(t), t ) \| _{H^{\frac{1}{2}+}} + ... \\
+ \| \triangle \left( \p_{t} \p_{x}^{m} u \right)^{j_{m}}(t) \|_{H^{\frac{1}{2}+}} \| S_{2} ( u_{2}(t),t) \|_{H^{\frac{1}{2}+}}
\| S_{1} ( u_{1}(t), t ) \| _{H^{r}} \\
\\
\lesssim \epsilon_{1}  \| \triangle u(t) \|_{H^{\max \left(  \left( \frac{9}{2} + m \right)+, r + m + 4 \right)}} \\
+ ( \epsilon_{2} - \epsilon_{1} ) \| u_{2}(t) \|_{H^{\max \left(  \left(\frac{9}{2} + m \right)+, r + m + 4 \right)}}
\| \triangle u (t) \|_{H^{\max \left( r + m + 3, \left( \frac{7}{2} + m \right)+ \right)}} \cdot
}
Hence we also get

\EQQARR{
\| \triangle \bar{S}_{m}(t) \|_{H^{r}} & \lesssim \| \triangle u(t) \|_{H^{ \max \left( r + 6, \frac{13}{2}+ \right) }} \cdot
}
This implies that  (\ref{Eqn:EstSutDiff}) holds.

\section{Appendix}
\label{Sec:App}

In this appendix we prove the following lemma:

\begin{lem}

Let $I$ be an interval such that $|I| \leq 1$. Let $ k' \geq 0 $. Let $K \in \mathbb{R}^{+}$.
Let $f$ and $g$ be two functions such that $ \| (f,g) \|_{ H^{\frac{9}{2}+}
\times H^{\frac{9}{2}+}} \leq K $. Then the following holds:

\begin{itemize}

\item There exists $C:= C(K) > 0 $ such that for all  $t_{1},t_{2} \in I$

\EQQARR{
\| X(f,g,t_{2},t_{1}) \|_{L^{\infty}} \leq C  \left( |t_{2} - t_{1}| +  \| f - g \|_{H^{\frac{7}{2}+}} \right) \cdot
}
Here  $ X(f,g,t_{2},t_{1}) := \left| \partial_{\omega_{3}} F ( \vec{f},t_{2} ) \right| -  \left| \partial_{\omega_{3}} F ( \vec{g},t_{1} ) \right| $.

\item Let $\bar{\delta} > 0$.  There exists $ C := C(K,\bar{\delta} ) > 0 $ such that for all
$t_{1},t_{2} \in I$ satisfying $ \delta ( \vec{f}, t_{2} ) \geq \bar{\delta}$
and $ \delta ( \vec{g}, t_{1} ) \geq \bar{\delta} $ we have

\EQQARR{
\| Y(f,g,t_{2},t_{1}) \|_{L^{\infty}} \leq C  \left( |t_{2} - t_{1}| +  \| f - g \|_{H^{\frac{9}{2}+}} \right) \cdot
}
Here

\EQQARR{
Y(f,g,t_{2},t_{1}) :=
\left| \partial_{\omega_{3}} F ( \vec{f},t_{2} ) \right|
\left[ \frac{P( f,t_{2} )}{ \left| \partial_{\omega_{3}} F ( \vec{f},t_{2} ) \right| }   \right]_{ave}
-  \left| \partial_{\omega_{3}} F ( \vec{g}, t_{1} ) \right|
\left[ \frac{P(g,t_{1})}{ \left| \partial_{\omega_{3}} F ( \vec{g},t_{1} ) \right|}
\right]_{ave} \cdot
}

\item Let $\bar{\delta} > 0$. Let $t \in I$. Assume that $ \delta(\vec{f}, t) \geq \bar{\delta}$ and $\delta( \vec{g}, t )  \geq \bar{\delta} $. Then there exists
$C := C(K,\bar{\delta}) > 0$ such that $ Z(f,g,t) \leq C  \| f - g \|_{H^{\frac{9}{2}+}} $. Here

\EQQARR{
Z(f,g,t):= \left| \left[ \frac{P(f,t)}{\partial_{\omega_{3}} F(\vec{f},t)} \right]_{ave} - \left[ \frac{P(g,t)}{\partial_{\omega_{3}} F(\vec{g},t)} \right]_{ave} \right| \cdot
}

\end{itemize}

\end{lem}

\begin{rem}
The proof shows that one can choose the constants $C$ to be increasing functions as $K$ increases. Moreover one can choose the constant $C$ in the second estimate
to be an increasing function as $\bar{\delta}$ decreases.
\end{rem}

\begin{proof}

The bound for $ X(f,g,t_{2},t_{1}) $ results from the triangle inequality
$| X (f,g,t_{2},t_{1}) | \lesssim  \left| \partial_{\omega_{3}} F ( \vec{f},t_{2} ) -  \partial_{\omega_{3}} F ( \vec{g},t_{1} ) \right| $, the
Sobolev embedding $H^{\frac{1}{2}+} \hookrightarrow L^{\infty}$, and Lemma \ref{lem:EstDerivF}. \\
We write $Y(f,g,t_{2},t_{1}) = A_{1} + A_{2} + A_{3}$ with

\EQQARR{
A_{1} & := (2 \pi)^{-1} \left( \left| \partial_{\omega_{3}} F \left( \vec{f}, t_{2} \right) \right| -
\left| \partial_{\omega_{3}} F \left( \vec{g}, t_{1} \right)  \right| \right)
\int_{\T} \frac{P (f,t_{2})}{\left|  \partial_{\omega_{3}} F \left( \vec{f}, t_{2} \right) \right|} \; dx, \\
A_{2} & := (2 \pi)^{-1}  \left| \partial_{\omega_{3}} F \left( \vec{g}, t_{1} \right) \right|
\int_{\T} \frac{P (f,t_{2}) - P(g,t_{1})}{\left| \partial_{\omega_{3}} F \left( \vec{f} , t_{2} \right)  \right|} \; dx, \; \text{and}  \\
A_{3} & := (2 \pi)^{-1}  \left| \partial_{\omega_{3}} F \left( \vec{g}, t_{1} \right) \right|   \int_{\T} P( g,t_{1} ) \left(
\frac{1}{\left| \partial_{\omega_{3}} F  \left( \vec{f} , t_{2}  \right)  \right|} -   \frac{1}{\left| \partial_{\omega_{3}}
F \left( \vec{g}, t_{1} \right)  \right|}
\right) \; dx \cdot
}
We see from the embeddings $L^{\infty} \hookrightarrow H^{\frac{1}{2}+} $ and $L^{2} \hookrightarrow L^{1}$, from Lemma \ref{lem:EstDerivF}, and
from (\ref{Eqn:DiffEstP}) that

\EQQARRLAB{
|A_{1}| & \lesssim \bar{\delta}^{-1}
\left\| \partial_{\omega_{3}} F \left( \vec{f}, t_{2} \right)
- \partial_{\omega_{3}} F \left( \vec{g}, t_{1} \right) \right\|_{L^{\infty}} \| P (f,t_{2}) \|_{L^{2}} \\
& \lesssim \| f  - g \|_{H^{\frac{9}{2}+}} + | t_{2} - t_{1}| \cdot
\label{Eqn:A1App}
}
Similarly

\EQQARR{
|A_{2}| & \lesssim \bar{\delta}^{-1}  \left\| \partial_{\omega_{3}} F ( \vec{g}, t_{1} ) \right\|_{L^{\infty}}
\left\| P (f,t_{2}) - P(g,t_{1}) \right\|_{L^{2}} \\
& \lesssim \| f  - g \|_{H^{\frac{9}{2}+}} + |t_{2}- t_{1}| \cdot
\label{Eqn:A2app}
}
Similarly

\EQQARR{
|A_{3}| & \lesssim \bar{\delta}^{-2}  \left\| \partial_{\omega_{3}} F \left( \vec{g}, t_{1} \right) \right\|_{L^{\infty}}
\left\| P ( g , t_{1} ) \right\|_{L^{2}}
\left\| \partial_{\omega_{3}} F ( \vec{f}, t_{2} ) -  \partial_{\omega_{3}}
F ( \vec{g}, t_{1} ) \right\|_{L^{2}} \\
& \lesssim \| f  - g \|_{H^{\frac{9}{2}+}} + |t_{2} - t_{1}|  \cdot
\label{Eqn:A3app}
}
We have $Z(f,g,t) \lesssim B_{1} + B_{2} $ with $B_{1} := \int_{\T} \frac{ \left| \left( P(f,t) - P(g,t) \right) \right| \, \left| \partial_{\omega_{3}} F(\vec{g},t) \right| }
{\left| \partial_{\omega_{3}} F(\vec{f},t) \right| \left| \partial_{\omega_{3}} F(\vec{g},t) \right|} \; dx $ and
$ B_{2} := \int_{\T}  \frac{ |P(g,t)| \, \left| \partial_{\omega_{3}} F(\vec{f},t) - \partial_{\omega_{3}} F(\vec{g},t) \right| }
{\left| \partial_{\omega_{3}} F(\vec{f},t) \right| \left| \partial_{\omega_{3}} F(\vec{g},t) \right|} \; dx $. Hence from $L^{1} \hookrightarrow L^{2}$,
$L^{\infty} \hookrightarrow H^{\frac{1}{2}+}$, Lemma \ref{lem:EstDerivF}, and (\ref{Eqn:DiffEstP}), we get $ Z(f,g,t) \lesssim  \| f - g \|_{H^{\frac{9}{2}+}} $.

\end{proof}

\end{document}